\documentclass[11pt]{amsart}

\usepackage[top=1.5in,bottom=1.5in,left=1.3in,right=1.3in,marginparwidth=1in]{geometry} 

\usepackage{mathrsfs}
\usepackage{amssymb}
\usepackage{amsmath, amsthm}
\usepackage{amsfonts}
\usepackage{longtable}
\usepackage{paralist}
\usepackage[displaymath, mathlines]{lineno}

\usepackage[pdftex]{hyperref}
\hypersetup{
    colorlinks=true, 
    linktoc=all,    
    linkcolor=blue,  
}

\newtheorem{theorem}{Theorem}[section]

\newtheorem{definition}[theorem]{Definition}
\newtheorem{Cor}[theorem]{Corollary}
\newtheorem{fact}[theorem]{Fact}
\newtheorem{lemma}[theorem]{Lemma}

\newtheorem{remark}[theorem]{Remark}

\newtheorem{que}[theorem]{Question}

\newtheorem{notation}[theorem]{Notation}

\newcommand{\lra}{\longrightarrow}
\newcommand{\res}{\upharpoonright}

\newcommand{\seq}{\subseteq}
\newcommand{\bsl}{\backslash}
\newcommand{\es}{\emptyset}

\newcommand{\we}{\wedge}

\newcommand{\ps}{\mathbb{P}}
\newcommand{\R}{\mathbb{R}}
\newcommand{\Q}{\mathbb{Q}}

\newcommand{\N}{\mathbb{N}}
\newcommand{\C}{\mathbb{C}}
\renewcommand{\S}{\mathbb{S}}

\newcommand{\mf}{\mathfrak}
\newcommand{\bb}{\mathbb}
\newcommand{\cal}{\mathcal}

\newcommand{\al}{\alpha}
\newcommand{\be}{\beta}
\newcommand{\ga}{\gamma}

\newcommand{\de}{\delta}

\newcommand{\ka}{\kappa}
\newcommand{\lam}{\lambda}

\newcommand{\vp}{\varphi}
\newcommand{\om}{\omega}

\newcommand{\lb}{\left\{}
\newcommand{\rb}{\right\}}
\newcommand{\la}{\langle}
\newcommand{\ra}{\rangle}

\newcommand{\ran}{\operatorname{ran}}
\newcommand{\dom}{\operatorname{dom}}

\newcommand{\cof}{\operatorname{cof}}
\newcommand{\cf}{\operatorname{cf}}

\newcommand{\Add}{\operatorname{Add}}

\newcommand{\crit}{\operatorname{crit}}
\newcommand{\Prk}{\mathrm{Prk}}
\newcommand{\Pc}{\mathrm{PrkCol}}
\newcommand{\Col}{\mathrm{Coll}}
\newcommand{\CU}{\mathsf{CU}}
\newcommand{\ult}{\operatorname{Ult}}
\newcommand{\stem}{\operatorname{stem}}
\newcommand{\pstem}{\operatorname{pstem}}

\newcommand{\CSR}{\mathsf{CSR}}
\newcommand{\SR}{{\sf SR}}
\newcommand{\AP}{{\sf AP}}
\newcommand{\zfc}{\sf ZFC}

\newcommand{\gch}{\sf GCH}

\newcommand{\TP}{\sf TP}

\begin{document}

\title{Club Stationary Reflection and other Combinatorial Principles at $\aleph_{\omega+2}$}

\date{\today}

\author{Thomas Gilton}\address[Gilton]{University of Pittsburgh
Department of Mathematics. The Dietrich School of 
Arts and Sciences, 301 Thackeray Hall,
Pittsburgh, PA, 15260 United States
} \email{tdg25@pitt.edu} \urladdr{https://sites.pitt.edu/~tdg25/}

\author{{\v S}{\'a}rka Stejskalov{\'a}}\address[Stejskalov{\'a}]{
Charles University, Department of Logic,
Celetn{\' a} 20, Prague~1, 
116 42, Czech Republic
}
\email{sarka.stejskalova@ff.cuni.cz}
\urladdr{logika.ff.cuni.cz/sarka}

\thanks{{\v S}.~Stejskalov{\'a} was supported by FWF/GA{\v C}R grant \emph{Compactness principles and combinatorics} (19-29633L).}

\begin{abstract}
    In this paper we continue the study in \cite{GLS:8fold} of compactness and incompactness principles at double successors, focusing here on the case of double successors of singulars of countable cofinality. We obtain models which satisfy the tree property and club stationary reflection at these double successors. Moreover, we can additionally obtain either approachability or its failure. We also show how to obtain our results on $\aleph_{\om+2}$ by incorporating collapses; particularly relevant for these circumstances is a new indestructibility theorem of ours showing that posets satisfying certain linked assumptions preserve club stationary reflection.
\end{abstract}



\maketitle

\tableofcontents

\section{Introduction}

A long-standing and fruitful line of research in contemporary set theory is concerned with the tension between two classes of combinatorial principles (see, for example, \cite{ScalesSquares}, \cite{FontHayut}, \cite{LambieHansonHayut}, \cite{8fold}, \cite{BNG}). The first of these classes we refer to as compactness and reflection principles, and examples of principles in this class are the tree property and stationary reflection. In tension with this first class, we have the class of incompactness and non-reflection principles which are typified by square principles and their consequences. Since principles in the latter class often imply the failure of principles in the other, it is of interest when principles from these classes are jointly consistent.

A central character in this line of research is the so-called square principle $\Box_\mu$, for a cardinal $\mu$, which Jensen (\cite{JENfine}) first showed is consistent with $\zfc$. While we won't need the definition here, $\Box_\mu$ intuitively asserts the existence of a coherent system of clubs which singularize all ordinals in the interval $[\mu,\mu^+)$. A $\Box_\mu$ sequence is frequently used to continue through limit stages of inductive constructions of length $\mu^+$. 

A $\Box_\mu$ sequence has a deep impact on the combinatorics of $\mu^+$. For instance, it implies that the approachability property holds at $\mu^+$, but that the tree property and the stationary reflection property both fail at $\mu^+$ (these principles will all be defined later). In \cite{8fold}, Cummings, Friedman, Magidor, Rinot, and Sinapova  showed that these three consequences of $\Box_\mu$ are mutually independent in the sense that all eight of their Boolean combinations are consistent, usually from large cardinals. They obtain their results in the case when $\mu$ is itself a successor, say $\kappa^+$ (and thus their results are on the double successor of $\kappa$), both when $\ka$ is regular and when $\ka$ is singular of cofinality $\om$. Moreover, they showed that they can make $\ka$ become $\aleph_\om$, the first singular cardinal.

Recent results in this vein (\cite{Gilton-thesis}, \cite{GK:8fold}) have built upon the work of \cite{8fold}. In addition, in \cite{GLS:8fold}, the authors (together with Levine) continued the work in \cite{8fold}, showing, in particular, that one can obtain any Boolean combination of approachability and the tree property at $\ka^{++}$ together with club stationary reflection at $\ka^{++}$. These results were obtained in the case when $\ka$ is regular. (Note that the Boolean combination of club stationary reflection, the approachability property, and the failure of the tree property already holds in Magidor's model \cite{M:sr}.)

In this paper, we take the next step and show how to obtain, when $\ka$ is singular of cofinality $\om$, a variety of combinatorics at $\ka^{++}$ simultaneously with club stationary reflection at $\ka^{++}$. We also show how to make $\ka$ become $\aleph_\om$. More precisely, we prove the following theorem (the notation will be defined later):

\begin{theorem}\label{thm:WeaklyCompactCases}
Assuming the existence of a supercompact cardinal $\ka$ and a weakly compact cardinal $\lam>\ka$, there are models witnessing each of the following:
\begin{enumerate}
    \item $\cf(\ka)=\om$, $\lam=\ka^{++}$, and $\CSR(\lam)\we\AP(\lam)\we\TP(\lam)$;
    \item $\cf(\ka)=\om$, $\lam=\ka^{++}$, and $\CSR(\lam)\we\neg\AP(\lam)\we\TP(\lam)$.
\end{enumerate}
Moreover, we can obtain (1) and (2) with $\ka=\aleph_\om$.
\end{theorem}

Our results combine a variety of techniques such as Mitchell-style forcings, club adding iterations, guessing functions from weakly compact cardinals, Woodin's fast function forcing and Prikry-type forcings (both with and without collapses). In more detail, our paper is organized as follows:

In Section \ref{sec:prelims}, we review the preliminaries. We record the definitions of almost all of the forcings which we use throughout the paper, and we review their central features. 

In Section \ref{sec:indestructibility}, we state a number of known indestructibility results which we will use liberally in our paper. We also include a new indestructibility result showing that a $\ka^+$-c.c., $\ka^+$-linked poset (in particular, a poset which is $\ka$-linked) preserves club stationary reflection at $\ka^{++}$. This theorem simplifies and generalizes Theorem 4.12(iii) of \cite{HS:u} and shows that a variety of Prikry-type posets (including many that add collapses) preserve club stationary reflection at $\ka^{++}$. We also review the types of guessing functions from weakly compact cardinals which we use throughout the paper. In particular, we tweak Hamkins' construction (\cite{Hamkins:LotteryPreparation} and \cite{Hamkins:LaverDiamond}) for obtaining weakly compact Laver diamonds in order to suit some of our later purposes. Finally, we will review the results from \cite{GLS:8fold} which we will need later in the paper.

In Section \ref{sec:ggs}, we show how to obtain guiding generics which are appropriate for the collapses in the Prikry forcing. There are a variety of different cases in which we need these guiding generics, and we build off of the work of \cite{GS:sch} and \cite{SU:tree} to show how to tailor the preparatory forcings in each of the desired circumstances. We then show that these guiding generics are preserved by forcing with the club adding posets of interest.

Section \ref{sec:wc} includes the proofs of Theorem \ref{thm:WeaklyCompactCases} which assume that $\lam$ is weakly compact. Section \ref{sec:wc} is by far the longest section in the paper, given the complexity of the arguments and given our commitment to include quite a bit of detail.

\textbf{Acknowledgements}: We would like to thank James Cummings for a number of helpful conversations about the material in this paper. We would also like to thank Hannes Jakob for pointing out an error in the previous version.

\section{Preliminaries}\label{sec:prelims}

\subsection{Combinatorial Principles}

In this subsection, we review the definitions of the combinatorial principles which we are studying in this paper. We begin with the varieties of stationary reflection of interest. Recall that if $S\seq\ka^+$ is stationary and $\al\in\ka^+\cap\cof(>\om)$, then $\al$ is a \emph{reflection point} of $S$ if $S\cap\al$ is stationary in $\al$; we also say in this case that $S$ \emph{reflects at $\al$}. We say that $\ka^+$ satisfies the \emph{stationary reflection property}, denoted $\SR(\ka^+)$, if every stationary subset of $\ka^+\cap\cof(<\ka)$ reflects at an ordinal $\ga<\ka^+$ of cofinality $\ka$.

One obtains stronger reflection principles by requiring that stationary sets reflect simultaneously. We are interested in a yet further strengthening which intuitively asserts that every stationary set which could reflect reflects \emph{almost everywhere}. 

\begin{definition}
We say that $\ka^+$ satisfies the \emph{club stationary reflection property}, denoted $\CSR(\ka^+)$, if for every stationary $S\seq\ka^+\cap\cof(<\ka)$, there is a club $C\seq\ka^+$ so that $S$ reflects at all $\ga\in C\cap\cof(\ka)$.
\end{definition}

\begin{remark}
    It is worth noting that there is no analogue to $\CSR(\mu^+)$ when $\mu$ is singular, by a Solovay splitting argument:\footnote{James Cummings informed us of this fact and attributed it to Menachem Magidor.}
\end{remark}

In addition to stationary reflection, we are interested in the approachability property, which can be thought of as a weak version of the square principle. Shelah originally introduced this ideal to clarify questions about preserving stationary sets (see \cite{ShelahApproachability}).

\begin{definition}
Suppose that $\ka$ is regular. Let $\vec{a}=\la a_\al:\al<\ka^+\ra$ be a sequence of subsets of $\ka^+$ each of size less than $\ka$. An ordinal $\ga<\ka^+$ is said to be \emph{approachable} with respect to $\vec{a}$ if there is an unbounded subset $A\seq\ga$ of order type $\cf(\ga)$ satisfying that for all $\be<\ga$, there is an $\al<\ga$ so that $A\cap\be=a_\al$.
\end{definition}

Thus a point $\ga$ is approachable with respect to $\vec{a}$ if there is some unbounded $A\seq\ga$ of minimal order type all of whose initial segments are enumerated by $\vec{a}$ before stage $\ga$. There is a natural ideal, called the \emph{approachability ideal}, associated to this:

\begin{definition}
The approachability ideal $I[\ka^+]$ consists of all $S\seq\ka^+$ for which there is a sequence $\vec{a}$ and a club $C\seq\ka^+$ so that every $\ga\in S\cap C$ is approachable with respect to $\vec{a}$.

We say that the \emph{approachability property} holds at $\ka^+$, denoted $\AP(\ka^+)$, if $\ka^+\in I[\ka^+]$.
\end{definition}

Finally, we are interested in properties of trees. Recall that a \emph{$\ka^+$-tree} is a tree $T$ of height $\ka^+$ all of whose levels have size $\leq\ka$.

\begin{definition} $\ka^+$ satisfies the \emph{tree property}, denoted $\TP(\ka^+)$, if every $\ka^+$-tree has a cofinal branch. A $\ka^+$-tree without a cofinal branch is called a \emph{$\ka^+$-Aronszajn tree.}
\end{definition}

\subsection{Cohen, Collapse, Club Adding, and Complete Embeddings}

In this subsection, we review the definitions of some of the standard posets which we use throughout the paper and some facts about complete (and dense) embeddings.

\begin{definition} Suppose that $\ka$ is regular and $X$ a set of ordinals. $\Add(\ka,X)$ is the poset consisting of conditions $p$ which are partial functions from $\ka\times X$ to $2$  with $|\dom(p)|<\ka$. The ordering is by reverse inclusion.
\end{definition}

Usually, we will be looking at Cohen forcing of the form $\Add(\ka,\lam)$ where $\lam$ is a regular cardinal.

\begin{definition} Suppose that $\ka<\lam$ are regular cardinals. $\Col(\ka,<\lam)$ is the poset consisting of conditions $p$ which are partial functions from $\ka\times\lam$ into $\lam$, of size $<\ka$, so that if $(\al,\mu)\in\dom(p)$, then $p(\al,\mu)<\mu$. The ordering is reverse inclusion.
\end{definition}

Next, we define the club adding iterations which we will use throughout the paper, and we state a few useful facts. Here we are following the exposition from \cite{GLS:8fold}, Section 2.4, which in turn followed Cummings' Handbook article \cite{CummingsHandbook}.

\begin{definition} Suppose that $\ka$ is regular and $X\seq\ka$ is stationary. The \emph{poset to add a club subset of $X$}, denoted $\CU(X)$, consists of closed, bounded $c\seq X$. The ordering is by end-extension: $d\leq c$ iff $c=d\cap (\max(c)+1)$.
\end{definition}

\begin{definition}\label{def:SCAI} Suppose that $\ka$ is regular. We say that an iteration $\la \ps_\alpha,\dot{\Q}_\alpha:\alpha < \ka^+\ra$ with  supports of size $<\ka$ is a \emph{standard club adding iteration of length} $\ka^+$ if for all $\alpha<\ka^+$, there is a $\ps_\alpha$-name $\dot{X}_\alpha$ for a stationary subset of $\ka$ so that $\ps_\alpha \Vdash \dot \Q_\alpha = \CU(\dot X_\alpha), \dot X_\alpha \subset \ka$.
\end{definition}

The following item is a standard fact about the chain condition of such iterations; the item thereafter states a useful consequence which we will use throughout the paper.

\begin{fact}\label{iterationchaincondition} 
Suppose that $\ka$ is regular, that $\ka^{<\ka} = \ka$, and that $\ps_{\ka^+} := \la \ps_\alpha,\dot \Q_\alpha:\alpha < \ka^+\ra$ is a standard club adding iteration of length $\ka^+$. Then $\ps_{\ka^+}$ has the $\ka^+$-chain condition.
\end{fact}

\begin{fact}
\label{nicenamestrick} If $\ps_{\ka^+}$ is a standard club adding iteration of length $\ka^+$ and $f: \ka \to\operatorname{Ord}$ is a member of $V[\ps_{\ka^+}]$, then there is some $\alpha<\ka^+$ such that $f \in V[\ps_\alpha]$.
\end{fact}

Now we recall the definition of two types of embeddings between posets.

\begin{definition}\label{def:ce}
Given forcing posets $\ps$ and $\Q$, a map $i:\Q\lra\ps$ is a \emph{complete embedding} if
\begin{enumerate}
    \item $i(1_\Q)=1_\ps$;
    \item for all $q_1,q_2\in\Q$, $[q_2\leq_\Q q_1\lra i(q_2)\leq_\ps i(q_1)]$;
    \item for all $q_1,q_2\in\Q$, $[q_1\perp_\Q q_2\lra i(q_1)\perp_\ps i(q_2)]$;\footnote{Note that the converse holds by an application of (2).}
    \item for all $A\seq\Q$, if $A$ is a maximal antichain in $\Q$, then $i[A]$ is a maximal antichain in $\ps$.\footnote{This also has a characterization in terms of ``reductions" of conditions. See Definition 3.3.71 in \cite{KunenBible}.}
\end{enumerate}
$i$ is said to be a \emph{dense embedding} if (1), (2), and (3) hold, along with
\begin{enumerate}
    \item[(5)] $i[\Q]$ is dense in $\ps$.
\end{enumerate}
\end{definition}

Recall that two forcing posets $\ps$ and $\Q$ are \emph{forcing equivalent} if for each $V$-generic $G$ over $\ps$, there is a $V$-generic $H$ over $\Q$ so that $V[G]=V[H]$, and vice versa. If $i:\Q\lra\ps$ is a complete embedding, then forcing with $\ps$ is equivalent to forcing with $\Q$, say to obtain a generic $H$, and then forcing over $V[H]$ with $\ps/H$, which consists of all $p\in\ps$ which are compatible with every condition in $i[H]$, with the same order as in $\ps$. The following lemma contains two standard facts about this situation:

\begin{fact}\label{fact:quotientStandard}
Suppose that $i:\Q\lra\ps$ is a complete embedding and that $H$ is $V$-generic over $\Q$.
\begin{enumerate}
    \item If $D$ is a dense subset of $\ps$ in $V$, then $D$ remains dense in $\ps/H$.
    \item If $p\in\ps/H$ and $q\in H$, then there is a $p^*\in\ps/H$ (not just in $\ps$) which extends $p$ and $i(q)$.
\end{enumerate}
\end{fact}

Dense embeddings are even stronger, as the following fact shows:

\begin{fact}
If $i:\Q\lra\ps$ is a dense embedding, then $\Q$ and $\ps$ are forcing equivalent.
\end{fact}

The last item in this subsection is a version of Easton's lemma which holds for distributive forcings.

\begin{lemma}\label{lemma:distributiveEaston}
    Suppose that $\ps$ is $\ka$-c.c., that $\Q$ is $\ka$-distributive, and that $\Q$ forces that $\ps$ is $\ka$-c.c. Then $\ps$ forces that $\Q$ is $\ka$-distributive.
\end{lemma}

\subsection{Varieties of Mitchell Forcing} We next turn to defining the variants of Mitchell forcing (\cite{M:tree}) which we will use. Suppose that $\ka$ is regular and that $\lam>\ka$ is at least Mahlo. 

For our purposes, we use versions which incorporate a guessing function. The purpose of the guessing function is to ensure preservation of the tree property after some further forcing. In \cite{ABR:tree}, Abraham built Mitchell forcings with a third component to achieve this preservation, and in \cite{CummingsForeman}, Cummings and Foreman further developed this idea by using guessing functions from a supercompact Laver diamond (\cite{Laver:indestructible}). Similarly, in \cite{GLS:8fold}, the authors and Levine also used guessing functions, but ones from weakly compact and Mahlo cardinals. We will describe these guessing functions later in Section \ref{sec:ggs}. We need a bit of notation first:

\begin{notation}
Let $A\seq (\ka,\lam)$ denote the set of inaccessible cardinals between $\ka$ and $\lam$, and let $A^*$ denote $A\bsl\lim(A)$.
\end{notation}

We will define one version of Mitchell forcing which collapses at all stages in $A$ and another that collapses only at stages in $A^*$. The former will force that $\AP(\lam)$ holds. The latter, by contrast, will force the failure of $\AP(\lam)$, due to the more sparse collapsing.

\begin{definition}\label{def:MGuessAP} Let $\mf{l}:\lam\lra V_\lam$ be any function. We define the poset $\bb{M}_{\mf{l}}(\ka,\lam)\res\be$ by recursion on $\be\in A$, setting $\bb{M}_{\mf{l}}(\ka,\lam):=\bigcup_{\be\in A}\bb{M}_{\mf{l}}(\ka,\lam)\res\be$. Conditions in $\bb{M}_{\mf{l}}(\ka,\lam)\res\be$ are triples $(p,q,g)$ where
\begin{enumerate}
\item $p$ is a condition in $\Add(\ka,\be)$;
\item $q$ is a partial function with $\dom(q)\seq (\ka,\be)\cap A$ so that $|\dom(q)|\leq\ka$;
\item for all $\al\in\dom(q)$, $q(\al)$ is an $\Add(\ka,\al)$-name for a condition in $\dot{\Add}(\ka^+,1)$;
\item $g$ is a function with $\dom(g)\seq (\ka,\be)\cap A$ so that $|\dom(g)|\leq\ka$;
\item for all $\al\in\dom(g)$, if $\mf{l}(\al)$ is an $(\bb{M}_{\mf{l}}(\ka,\lam)\res\al)$-name for a $\ka^+$-directed\footnote{This will ensure that the whole poset is $\ka$-directed closed.} closed poset, then $g(\al)$ is an $(\bb{M}_{\mf{l}}(\ka,\lam)\res\al)$-name for a condition in $\mf{l}(\al)$. Otherwise $g(\al)$ is trivial.
\end{enumerate}
We say that $(p',q',g')\leq (p,q,g)$ if
\begin{enumerate}
\item $p'\leq p$;
\item $\dom(q)\seq\dom(q')$ and $\dom(g)\seq\dom(g')$;
\item for all $\al\in\dom(q)$, $p'\res\al\Vdash q'(\al)\leq q(\al)$;
\item for all $\al\in\dom(g)$, $(p',q',g')\res\al\Vdash_{\bb{M}_{\mf{l}}(\ka,\lam)\res\al} g'(\al)\leq g(\al).$
\end{enumerate}
\end{definition}

\begin{remark}
It is not absolutely necessary to use posets of the form $\bb{M}_{\mf{l}}(\ka,\lam)$ to obtain some of our later results, since an absorption lemma for Mitchell forcing (\cite{GLS:8fold}) can play a role similar to that of the guessing function in Definition \ref{def:MGuessAP}. However, it does make it a bit easier to work with the quotients (especially when we discuss branch lemmas in Section \ref{sec:wc}) if we do include the guessing function, and that is why we do so.
\end{remark}

We can also define a variation which uses a guessing function and forces the failure of approachability. The guessing function plays a more prominent role in these variations, since there is no natural (and useful) absorption lemma analogous to the one from \cite{GLS:8fold} with respect to $\bb{M}(\ka,\lam)$. 

\begin{definition}\label{def:MGuessNotAP}
Given $\mf{l}:\lam\lra V_\lam$, we define $\bb{M}^*_{\mf{l}}(\ka,\lam)$ as in Definition \ref{def:MGuessAP}, replacing (2) with
\begin{enumerate}
    \item[(2)$^*$] $q$ is a partial function with $\dom(q)\seq(\ka,\be)\cap A^*$ so that $|\dom(q)|\leq\ka$.
\end{enumerate}
The ordering is as in Definition \ref{def:MGuessAP}.
\end{definition}

When $\ka$ and $\lam$ are clear from context, we will often drop them from the notation. The forcings $\bb{M}^*_{\mf{l}}$ and $\bb{M}_{\mf{l}}$ have very different effects on approchability; this comes through in the use of $A$ or $A^*$ in defining the posets. The following fact is now standard (for example, see \cite{8fold}):

\begin{fact}
Suppose that $\ka$ is regular and $\lam>\ka$ is Mahlo. Let $\mf{l}:\lam\lra V_\lam$. Then $\bb{M}_{\mf{l}}(\ka,\lam)$ forces that $\AP(\lam)$ holds, and $\bb{M}^*_{\mf{l}}(\ka,\lam)$ forces that $\AP(\lam)$ fails.
\end{fact}

A now-standard technique, also originating with Abraham \cite{ABR:tree}, is to analyze Mitchell-type posets, as well as their quotients, in terms of forcing projections which include a ``term ordering" (which in turn goes back to Laver; see \cite{CummingsHandbook}, section 22). We state many of these facts without proofs; the interested reader can see, for example, \cite{8fold}.

We describe in more detail the analysis of $\bb{M}_\mf{l}$, and we will then indicate how to make the necessary changes for $\bb{M}^*_\mf{l}$.

Let $\bb{T}_\mf{l}(\ka,\lam)$ be the poset, in $V$, which consists of triples $(1,q,g)\in\bb{M}_\mf{l}$, with the restriction of the $\bb{M}_\mf{l}$ ordering. Since the names in the ranges of $q$ and $g$ are forced to be in posets which are $\ka^+$-closed, $\bb{T}_\mf{l}(\ka,\lam)$ is $\ka^+$-closed. Moreover, we have a natural forcing projection from $\Add(\ka,\lam)\times\bb{T}_\mf{l}(\ka,\lam)$ onto $\bb{M}_\mf{l}$, namely $(p,(1,q,g))\mapsto (p,q,g)$.

Now we consider quotients; we are following Lemma 2.12 of \cite{ABR:tree}. Fix an inaccessible $\de\in(\ka,\lam)$. There is then the question of whether or not $\mf{l}(\de)$ gives a name for a viable poset; see (5) of Definition \ref{def:MGuessAP}. We will only need to analyze the quotient in the case when $\mf{l}(\de)$ does name a $\ka^+$-directed closed poset, and so we assume this for the purposes of discussion.

Let $V[\bar{G}\ast H]$ be an arbitrary generic extension by $(\bb{M}_{\mf{l}}\res\de)\ast\mf{l}(\de)$. We define the poset $\N_\mf{l}[\de,\lam)\res\be$ by recursion on $\be\in(\de,\lam)\cap A$, setting $\bb{N}_\mf{l}[\de,\lam):=\bigcup_{\be\in(\de,\lam)\cap A}\bb{N}_\mf{l}[\de,\lam)\res\be$. For simplicity of notation in the discussion that follows, we write $\bb{N}_\mf{l}$ for $\bb{N}_\mf{l}[\de,\lam)$. Suppose, then, that $\be\in(\de,\lam)\cap A$, and that we've defined $\bb{N}_\mf{l}\res\al$ for all $\al\in (\de,\be)\cap A$. We also need to specify a recursive assumption, and we introduce the following notation to state the recursive assumption. Note, however, that we are going to skip the usual details about ``shifting" names. Given $m\in \bb{M}_\mf{l}\res\al$, and writing $m=(p^m,q^m,g^m)$, we set 
$$
m\res\de:=(p^m\res\de,q^m\res\de,g^m\res\de)\text{ and }m\bsl\de:=(p^m\res[\de,\al),q^m\res[\de,\al),g^m\res(\de,\al)).
$$
Note that $\de$ is not in the domain of the $g$-part of $m\bsl\de$. Our recursive assumption is that for each such $\al$, the map taking $m\in \bb{M}_\mf{l}\res\al$  to $(m\res\de,g^m(\de),m\bsl\de)$ is a dense embedding from $\bb{M}_\mf{l}\res\al$ to $(\bb{M}_\mf{l}\res\de)\ast\mf{l}(\de)\ast(\dot{\bb{N}}_\mf{l}\res\al)$. This recursive assumption will be used in item (5) below.

Now we may define $\bb{N}_\mf{l}\res\be$. Conditions in this poset are triples $(p,q,g)$ so that
\begin{enumerate}
    \item $p$ is a condition in $\Add(\ka,[\de,\be))$;
    \item $q$ is a partial function with $\dom(q)\seq [\de,\be)\cap A$ so that $|\dom(q)|\leq\ka$;
    \item for all $\al\in\dom(q)$, $q(\al)$ is an $\Add(\ka,[\de,\al))$-name for a condition in $\dot{\Add}(\ka^+,1)$;
    \item $g$ is a function with $\dom(g)\seq (\de,\be)\cap A$ so that $|\dom(g)|\leq\ka$ (note that $\de$ is excluded from $\dom(g)$ since we already forced with $\mf{l}(\de)[\bar{G}]$); and
    \item for all $\al\in\dom(g)$, if $\mf{l}(\al)$ is an $(\bb{M}_\mf{l}\res\al)$-name for a $\ka^+$-directed closed poset, then $g(\al)$ is an $(\bb{N}_\mf{l}\res\al)$-name forced to be a member of $\mf{l}(\al)[\bar{G}\ast H][(\bb{N}_\mf{l}\res\al)]$.
\end{enumerate}

Note that if $(p,q,g)$ is a triple in $\bb{N}_\mf{l}\res\be$, then neither $\dom(q)$ nor $\dom(g)$ need be elements of $V$, since $V[\bar{G}]$ contains new subsets of $\ka$. However, these sets are covered by sets of size $\ka$ in $V$. This is crucial for verifying that the above maps are indeed dense embeddings. The reader can also check that the map taking $m\in\bb{M}_\mf{l}$ to $(m\res\de,g^m(\de),m\bsl\de)$ is a dense embedding from $\bb{M}_\mf{l}$ to $(\bb{M}_\mf{l}\res\de)\ast\mf{l}(\de)\ast\dot{\N}_\mf{l}$.

To conclude the discussion of the quotients of $\bb{M}_\mf{l}$, we note the relevant projections onto $\bb{N}_\mf{l}$. In $V[\bar{G}\ast H]$, we let $\bb{T}_\mf{l}[\de,\lam)$ consist of triples $(1,q,g)$ which are in $\bb{N}_\mf{l}$ with the restriction of the $\bb{N}_\mf{l}$ ordering. The map taking $(p,(1,q,g))\in\Add(\ka,[\de,\lam))\times\bb{T}_\mf{l}[\de,\lam)$ to $(p,q,g)$ is a forcing projection to $\bb{N}_\mf{l}$.

The analysis of projections and quotients of $\bb{M}^*_\mf{l}$ is almost the exact same. The only change is that the $q$ parts (the ones responsible for collapsing) are restricted to have domain contained in $A^*$. The objects $\bb{T}^*_\mf{l}(\ka,\lam)$, $\dot{\bb{N}}^*_\mf{l}[\de,\lam)$, and $\bb{T}^*_\mf{l}[\de,\lam)$ are defined in the obvious way.

The next item summarizes the above discussion (we refer the reader to Lemma 2.12 of \cite{ABR:tree}, as well as the paper \cite{8fold}, for more details).

\begin{lemma}\label{lemma:tailprojections} Let $\ka$ be regular and $\lam>\ka$ Mahlo. Let $\mf{l}:\lam\lra V_\lam$ be a function.
\begin{enumerate}
    \item $\bb{M}_{\mf{l}}(\ka,\lam)$  (resp. $\bb{M}^*_{\mf{l}}(\ka,\lam)$) is the projection of the product of $\Add(\ka,\lam)$ and $\bb{T}_{\mf{l}}(\ka,\lam)$ (resp. $\bb{T}^*_{\mf{l}}(\ka,\lam)$) via the natural projection map. Moreover, both Mitchell posets are $\lam$-c.c. and $\ka$-directed closed.
    \item Suppose that $\de$ is inaccessible with $\ka<\de<\lam$, and that $\mf{l}(\de)$ is an $(\bb{M}_{\mf{l}}(\ka,\lam)\res\de)$-name for a $\ka^+$-directed closed poset. Then there is a dense embedding from $\bb{M}_{\mf{l}}(\ka,\lam)$ to $(\bb{M}_{\mf{l}}(\ka,\lam)\res\de)\ast\mf{l}(\de)\ast\dot{\N}_{\mf{l}}[\de,\ka)$ given by
    $$
    m\mapsto (m\res\de,g^m(\de),m\bsl\de),
    $$
    where $m=(p^m,q^m,g^m)$, and where $m\bsl\de$ is the triple $(p^m\res[\de,\lam),q^m\res[\de,\lam),g^m\res(\de,\lam))$ (note the different restriction for $g^m$). Moreover, $\dot{\N}_{\mf{l}}[\de,\ka)$ is forced to be a projection of the product of $\Add(\ka,[\de,\lam))$ and $\dot{\bb{T}}_{\mf{l}}[\de,\lam)$ via the natural projection map. Finally, $\dot{\bb{T}}_{\mf{l}}[\de,\lam)$ is forced to be $\ka^+$-closed.
    \item Item (2) also applies, with analogous notation, to $\bb{M}^*_{\mf{l}}(\ka,\lam)$.
\end{enumerate}
\end{lemma}

\subsection{Prikry Forcing}

In this subsection, we review the two standard varieties of Prikry forcing which we use to singularize $\ka$ (see \cite{GitikHandbook} for details). We assume throughout this subsection that $\ka$ is measurable and that $U$ is a normal ultrafilter on $\ka$. Fixing notation, we let $M$ denote the ultrapower of $V$ by $U$ and $j_U$ denote the ultrapower embedding.

\begin{definition}
The poset $\Prk(U)$ consists of all pairs $(s,A)$ so that
\begin{enumerate}
    \item $s$ is a finite, increasing sequence of inaccessibles below $\ka$;
    \item $A\in U$ and $\max(\ran(s))<\min(A)$.
\end{enumerate}
We say $(t,B)\leq (s,A)$ if
\begin{enumerate}
    \item $s\sqsubseteq t$, and $\ran(t\bsl s)\seq A$;
    \item $B\seq A$.
\end{enumerate}
We say that $(t,B)$ is a \emph{direct} extension of $(s,A)$, denoted $(t,B)\leq^*(s,A)$, if $(t,B)\leq (s,A)$ and $s=t$. 
\end{definition}

$\Prk(U)$ is $\ka^+$-Knaster since there are $\ka$-many possible stems and any two conditions with the same stem are compatible. $\Prk(U)$ also satisfies the following significant property, known as the \emph{Prikry Property}: if $\vp$ is a sentence in the forcing language of $\Prk(U)$ and $(s,A)$ a condition, then there exists a direct extension $(s,B)\leq^*(s,A)$ so that $(s,B)$ decides $\vp$. Since the direct extension order $\leq^*$ is $\ka$-closed, the Prikry property implies that no bounded subsets of $\ka$ are added. Combining all of this, we see that all cardinals are preserved and $\ka$ becomes singular of cofinality $\om$.

We next define a variation of the above which turns $\ka$ into $\aleph_\om$. Roughly, the forcing will simultaneously add an $\om$-sequence cofinal in $\ka$ and collapse all but finitely-many cardinals between the points on the $\om$-sequence. It is important that the two tasks (adding an $\omega$-sequence and collapsing) are performed at the same time; it is known that performing them one-by-one may lead to collapsing (see \cite{H:prikry} for more details). Prikry forcing with collapses was introduced in \cite{GITIKo2}, and we follow their presentation.

In order to make this task work, we need a way of constraining the collapses in order to prove the Prikry lemma. We use a collapse constraining function. This function will delimit the possibilities for new collapses added for additional Prikry points, and we ensure that these constraining functions are well-behaved by requiring that their $U$-equivalence classes come from a \emph{guiding generic}, which we now define.

\begin{definition}
Fix $k\in\om\bsl 2$. A \emph{$(U,k)$-guiding generic} is a filter $G^g$ which is an element of $V$ and which is $M=\mathrm{Ult}(V,U)$-generic over the poset $\Col^M(\ka^{+k},<j_U(\ka))$.
\end{definition}

Next up is the definition of the Prikry-type forcing:

\begin{definition}\label{def:PrikryWithCollapses}
Fix $k\in\om\bsl 2$ and a $(U,k)$-guiding generic $G^g$. \emph{Prikry forcing with collapses}, denoted $\Pc(U,G^g)$, consists of all pairs $(s,F)$ so that
\begin{enumerate}
    \item $s$ is a finite sequence of the form $\la f_0,\al_1,f_1,\dots,\al_{m-1},f_{m-1}\ra$;
    \item $\la\al_1,\dots,\al_{m-1}\ra$ is a strictly increasing sequence of inaccessibles below $\ka$;
    \item $f_0\in\Col(\om,<\al_1)$ if $n\geq 1$, and otherwise $f_0\in\Col(\om,<\ka)$;
    \item for all $i$ with $0<i<m-1$, $f_i\in\Col(\al_i^{+k},<\al_{i+1})$, and $f_{m-1}\in\Col(\al_{m-1}^{+k},<\ka)$;
    \item $F$ is a function so that $\dom(F)\in U$ consists of inaccessibles, so that $\dom(F)\seq[\de,\ka)$ for some $\de$ large enough that $s\in V_\de$, and so that for all $\al\in\dom(F)$, $F(\al)\in\Col(\al^{+k},<\ka)$;
    \item $[F]_U\in G^g$.
\end{enumerate}
We refer to $s$ as a \emph{stem} and $F$ as an \emph{upper part}. Additionally, we refer to $m$ as the \emph{length} of $s$ and denote it by $\ell(s)$. Finally, we refer to the sequence of ordinals $\la\al_1,\dots,\al_{m-1}\ra$ from $s$ as the \emph{Prikry Stem} of $s$, denoted $\pstem(s)$.

We say that $(\la g_0,\be_1,g_1,\dots,\be_{n-1},g_{n-1}\ra,H)\leq (\la f_0,\al_1,f_1,\dots,\al_{m-1},f_{m-1}\ra,F)$ if
\begin{enumerate}
    \item $m\leq n$;
    \item for all $i<m$, $\al_i=\be_i$ and $g_i\leq f_i$;
    \item $\dom(H)\seq\dom(F)$, and for all $\al\in\dom(H)$, $H(\al)\leq F(\al)$;
    \item for all $i$ with $m\leq i<n$, $\be_i\in\dom(F)$, and $g_i\leq F(\be_i)$.
\end{enumerate}
\end{definition}

The direct extension ordering is defined in the natural way: $(t,H)\leq^*(s,F)$ denotes $(t,H)\leq (s,F)$ and $\ell(t)=\ell(s)$. However, unlike the direct extension ordering on $\Prk(U)$, the stems $t$ and $s$ needn't be equal, since the collapses on $t$ may strengthen those on $s$ (but they will have the same Prikry points).

$\Pc(U,G^g)$ is also $\ka^+$-Knaster since all elements in $G^g$ are compatible, and hence any two conditions with the same stem (of which there are only $\ka$-many) are compatible. This poset also satisfies the Prikry property. A product-like analysis in combination with the Prikry property shows that below $\ka$, only the cardinals explicitly collapsed by $\Pc(U,G^g)$ are collapsed. In particular, $\ka$ becomes $\aleph_\om$ in the extension.

For later arguments, it will be helpful to have some notation for orderings on stems and upper parts.

\begin{definition}
Let $s$ and $t$ be stems. We write $t\leq^*s$ if $\ell(s)=\ell(t)$ and condition (2) of the ordering is satisfied. Also, given upper parts $F$ and $H$, we write $H\leq F$ if condition (3) of the ordering is satisfied.
\end{definition}

\section{Indestructibility Results, Guessing Functions, and their Uses}\label{sec:indestructibility}

In this section we will first collect the indestructibility results which we need to construct our models. The main new result of this section is our theorem that a $\ka^+$-c.c. and $\ka^+$-linked (see below for the exact definition) poset preserves club stationary reflection at the double successor of $\ka$. The second objective of this section is to introduce the types of guessing functions which we use to build useful versions of the  Mitchell forcings from Definitions \ref{def:MGuessAP} and \ref{def:MGuessNotAP}. We will also review and adjust the proofs for how to obtain these guessing functions since we will need the details later.

\subsection{Indestructibility results}

Before we state the first indestructibility result of interest, we need to sort out some terminology in order to uniformize definitions across different papers. We take Kunen's book \cite{KunenBible} as the standard.

\begin{definition}\label{def:linked}
A poset $\ps$ is $\ka$-\emph{linked} (for a cardinal $\ka$) if $\ps$ can be partitioned into sets $\ps_\nu$, for $\nu<\ka$, which satisfy the following: for all $\nu<\ka$ and all $p,q\in\ps_\nu$, $p$ and $q$ are compatible in $\ps$. If, in addition, for any $\nu<\ka$ and $p,q\in\ps_\nu$ there is a condition $r\in\ps_\nu$ which is below $p$ and $q$, then we say that $\ps$ is \emph{strongly $\ka$-linked}.

We say that $\ps$ is $\ka$-\emph{centered} if there is a partition $\la\ps_\nu:\nu<\ka\ra$ of $\ps$ so that for all $\nu<\ka$, all $n\in\om$, and all $p_0,\dots,p_{n-1}\in\ps_\nu$, there is a condition $q$ below $p_i$ for all $i<n$. \emph{Strongly $\ka$-centered} asserts that such a $q$ can be found in $\ps_\nu$.
\end{definition}

\begin{remark}Note that $\ka$-linked and $\ka$-centered assert the existence of $\ka$-many ``cells" (the $\ps_\nu$). This is in contrast to other locutions, such as ``$\ka$-c.c." and ``$\ka$-closed", which use the parameter $\ka$ to say something about objects of size \emph{less than} $\ka$.

Also note that strongly $\ka$-linked and strongly $\ka$-centered are equivalent.
\end{remark}

\begin{fact}[\cite{GK:a}]\label{th:GK}
Suppose that $\mu$ is regular, that $\Q$ is strongly $\mu$-linked, and that $\neg \AP(\mu^{++})$ holds. Then $\Q$ forces $\neg \AP(\mu^{++})$.\footnote{The fact quoted here corresponds to Corollary 2.2 of \cite{GK:a}. Note that the authors of that paper use ``$\mu$-centered" to mean what we are calling ``strongly $\mu$-linked" (see the first paragraph of Section 2 of \cite{GK:a}), and note that they use $\mathsf{AP}_{\mu^+}$ to mean what we are calling $\AP(\mu^{++})$.}
\end{fact}

We will often be applying Fact \ref{th:GK} to some version of $\Add(\mu,\nu)$, and so we summarize the situation here (see Lemma 6.4 of \cite{GLS:8fold} for the details):

\begin{fact}\label{fact:AddCentered}
Suppose that $\ka^{<\ka}=\ka$ and that $\mu\leq 2^\ka$. Then $\Add(\ka,\mu)$ is strongly $\ka$-centered.
\end{fact}

\begin{Cor}\label{Cor:Add,AP}
Suppose that $\ka^{<\ka}=\ka$, that $2^\ka\geq\ka^{++}$, and that $\neg\AP(\ka^{++})$ holds. Then for any cardinal $\mu$, $\Add(\ka,\mu)$ forces that $\neg\AP(\ka^{++})$ holds.
\end{Cor}
\begin{proof}
There are two cases on $\mu$. First, if $\mu\leq\ka^{++}\leq 2^\ka$, then this follows by Facts \ref{fact:AddCentered} and \ref{th:GK}. On the other hand, if $\mu>\ka^{++}$, then we see that if $\Add(\ka,\mu)$ were to add a sequence witnessing $\AP(\ka^{++})$, then $\Add(\ka,\ka^{++})$ would add such a sequence. This then reduces to the previous case.
\end{proof}

Now we turn to preserving club stationary reflection.

\begin{fact}[\cite{HS:u}]\label{th:AddCSR}
Suppose that $\ka^{<\ka}=\ka$ and that $\CSR(\ka^{++})$ holds. Then for any $\rho$, $\Add(\ka,\rho)$ forces $\CSR(\ka^{++})$.
\end{fact}

While we do not have a general result for the preservation of the tree property, there are results over the Mitchell model, see \cite{HS:ind}.

The last item in this subsection is our new general preservation theorem for $\CSR(\ka^{++})$. As in many proofs in this vein (\cite{ScalesSquares}, \cite{HS:u}, \cite{SigmaPrikry1}, and \cite{SigmaPrikry2}), we use the ``traces with a fixed stem" idea and the fact that the desired reflection principle holds prior to the forcing. 

\begin{theorem}\label{th:CSRPreserve}
If $\ps$ is both $\ka^+$-linked and $\ka^+$-c.c. (in particular, if $\ps$ is simply $\ka$-linked), and if $\CSR(\ka^{++})$ holds, then $\ps$ forces $\CSR(\ka^{++})$.
\end{theorem}
\begin{proof} 
Before beginning the main part of the proof, we recall the following standard forcing fact: if $\mu$ is regular and $\Q$ is $\mu$-c.c., then for any condition $q\in\Q$, any ordinal $\de$ of cofinality $\mu$, and any $\Q$-name $\dot{D}$ for a club subset of $\de$, there is a club $C$ in $V$ so that $q\Vdash \check{C}\seq\dot{D}$.

Let $\la\ps_\nu:\nu<\ka^+\ra$ be a partition witnessing that $\ps$ is $\ka^+$-linked, and let $\vp:\ps\lra\ka^+$ be the function which takes $p$ to the unique $\nu$ with $p\in\ps_\nu$.

Suppose that $\CSR(\ka^{++})$ holds and that $\dot{S}$ is a $\ps$-name for a stationary subset of\footnote{Note that $\ps$ could singularize $\ka$, for instance, if $\ps$ is Prikry forcing from a normal measure on $\ka$. Then $\dot{S}$ would be a name for a stationary subset of $\ka^{++}\cap\cof(<\ka)$.} $\ka^{++}\cap\cof(\leq\ka)$. For a $\nu<\ka^+$, corresponding to a cell in the partition, we let 
$$
T_\nu:=\lb \al<\ka^{++}:(\exists p\in\ps_\nu)\;\left[p\Vdash\al\in\dot{S}\right]\rb.
$$
We say that $\nu$ is \textbf{strong} if $T_\nu$ is stationary in $\ka^{++}$. Note that if $\al\in T_\nu$, then $\cf(\al)\leq\ka$, since $\ps$ preserves $\ka^+$. In particular, $T_\nu\seq\ka^{++}\cap\cof(\leq\ka)$.

We claim that for all $p\in\ps$, we can extend $p$ to a condition $q$ so that $\vp(q)$ is strong. So fix $p$. Let $G$ be $V$-generic over $\ps$ with $p\in G$, and note that $S:=\dot{S}[G]$ is stationary in $\ka^{++}$. We assert that
$$
S\seq\bigcup\lb T_{\vp(r)}:r\leq p\we r\in G\rb;
$$
indeed, if $\al\in S$, then there is some $u\in G$ so that $u\Vdash\al\in\dot{S}$. Let $r\leq p,u$ with $r\in G$. Then $r\Vdash\al\in\dot{S}$ too, so $\al\in T_{\vp(r)}$ (as witnessed by $r$). However, since $S$ is a stationary subset of $\ka^{++}$ in $V[G]$ and $S$ is covered by $T_\mu$ for at most $\ka^+$-many $\mu$, there must be some $q\leq p$ with $q\in G$ so that $T_{\vp(q)}$ is stationary in $V[G]$. Since $T_{\vp(q)}$ is an element of $V$, it is stationary in $V$ also. This completes the proof of this claim.

For each $\nu<\ka^+$ so that $\nu$ is strong, let $C_\nu$ be a club subset of $\ka^{++}$ so that for all $\de\in C_\nu\cap\cof(\ka^+)$, $T_{\nu}$ reflects at $\de$. Let $C:=\bigcap\lb C_\nu:\nu\text{ is strong}\rb$. Then $C$ is a club subset of $\ka^{++}$ since it is the intersection of at most $\ka^+$-many clubs in $\ka^{++}$.

We claim that $\ps$ outright forces that for all $\de\in C\cap\cof(\ka^+)$, $\dot{S}$ reflects at $\de$. Suppose that this is false, for a contradiction. Then there is some condition $p$ which forces that it is false that $\dot{S}$ reflects at every point in $C\cap\cof(\ka^+)$. Extending $p$ to $q$, we may find a specific $\de\in C\cap\cof(\ka^+)$ so that $q$ forces that $\dot{S}$ does not reflect at $\de$. Let $\dot{D}$ be a $\ps$-name for a club subset of $\de$ so that $q\Vdash\dot{D}\cap\dot{S}\cap\de=\es$. Since $\ps$ is $\ka^+$-c.c., we may apply the fact from the beginning of the proof to find a club subset, say $E$, of $\de$ which is in $V$ and so that $q \Vdash\check{E}\seq\dot{D}$. Hence $q\Vdash \check{E}\cap\dot{S}\cap\de=\es$. Now let $r\leq q$ be a condition so that $\vp(r)$ is strong. Then $\de\in C\seq  C_{\vp(r)}$, and therefore $T_{\vp(r)}$ reflects at $\de$. In particular, $E\cap T_{\vp(r)}\cap\de$ is nonempty. Let $\al$ be in this intersection. Then there is some $u\in \ps_{\vp(r)}$ so that $u\Vdash\al\in\dot{S}$. Since $r,u\in \ps_{\vp(r)}$, we may find some $v\in\ps$ with $v\leq r,u$. Then $v\Vdash\al\in\dot{S}$, and since $\al\in E\cap\de$, $v$ also forces that $\al\in E\cap\dot{S}\cap\de$. However, $v\leq r\leq q$, and $q$ forces that this intersection is empty. This is a contradiction.
\end{proof}

\begin{remark}
The previous theorem applies to a variety of forcings, including $\Prk(U)$ and $\Pc(U,G^g)$. Note also that the previous theorem subsumes Fact \ref{th:AddCSR} when $2^\ka\geq\ka^{++}$.
\end{remark}

\subsection{Guessing Functions}

In \cite{CummingsForeman}, Cummings and Foreman constructed Mitchell forcings using a guessing function from a supercompact cardinal known as a Laver diamond. This is a function $\ell:\ka\lra V_\ka$ so that for any $x$, there is a supercompactness measure $\cal{U}$ on some $P_\ka(\mu)$ so that $j_{\cal{U}}(\ell)(\ka)=x$, where $j_{\cal{U}}$ is the corresponding ultrapower embedding. Therefore, by choosing an appropriate $\cal{U}$, one can ensure that the $j_{\cal{U}}$-image of the Mitchell forcing contains additional forcings of interest. This is in turn used for showing that the tree property is preserved by certain forcings. 

However, there are other guessing functions which are both sufficient for our purposes and only require smaller large cardinal hypotheses, such as Mahlo or weakly compact. See Hamkins (\cite{Hamkins:LotteryPreparation} and \cite{Hamkins:LaverDiamond}) for the weakly compact case and (\cite{GLS:8fold}) for the Mahlo case. We cover the weakly compact case in depth, since later we will need variations of Hamkins' arguments and since the details will also be important.

We begin by recalling the definition of weak compactness.

\begin{definition}\label{def:wc}
An inaccessible cardinal $\lam$ is \emph{weakly compact} if for every transitive set $M$ of size $\lam$ so that $M$ satisfies enough $\zfc$ and $\,^{<\lam}M\seq M$, there is an elementary embedding $k:M\lra N$ so that $\crit(k)=\lam$ and so that $N$ is transitive, has size $\lam$, and is closed under $<\lam$-sequences.
\end{definition}

\begin{definition}\label{def:WC}
If $M,N$, and $k$ are as in Definition \ref{def:wc}, then we say that $k$ is a \emph{weakly compact embedding} from $M$ to $N$.
\end{definition}

The following result shows that if $\lam$ is weakly compact, then a forcing of size $<\lam$ will preserve the weak compactness of $\lam$. We are not sure of a citation, so we take it to be among the class of ``most likely folklore" results.

\begin{lemma}\label{lemma:preserveWC}
   Suppose that $\lam$ is weakly compact and $|\ps|<\lam$. Then $\ps$ preserves that $\lam$ is weakly compact.
\end{lemma}
\begin{proof}
    By passing to a poset on $|\ps|$ isomorphic to $\ps$, we assume that $\ps\in H(\lam)$. Let $\dot{M}'$ be a $\ps$-name for a transitive set of size $\lam$ which satisfies enough of $\zfc$ and which is closed under $<\lam$-sequences. Since $\ps$ is $\lam$-c.c., $\dot{M}'\in H(\lam^+)$. Let $M$ be a transitive set in $V$ of size $\lam$ which satisfies enough of $\zfc$, which is closed under $<\lam$-sequences, and which contains $\dot{M}'$ (it contains $\ps$ since $H(\lam)\seq M$). Since $\lam$ is weakly compact, let $k:M\to N$ be as in Definition \ref{def:WC}. Let $G$ be $V$-generic over $\ps$, and since $k(\ps)=\ps$, lift $k$ to $k:M[G]\to N[G]$. Let $M':=\dot{M}'[G]\in M[G]$. Then the restriction of $k$ to $M'$ is an elementary embedding from $M'$ to $k(M')$. Moreover, as $M'$ is closed under $<\lam$-sequences from $V[G]$, $M[G]$ satisfies this, and hence by the elementarity of $k$, $N[G]$ satisfies that $k(M')$ is closed under $<k(\lam)$-sequences. However, $N[G]$ is closed under $<\lam$-sequences from $V[G]$, since $N$ is closed under $<\lam$-sequences from $V$ and since $\ps$ is $\lam$-c.c. Thus $k(M')$ is in fact closed under $<\lam$-sequences from $V[G]$, and this completes the proof.
\end{proof}

Now we define a guessing function suitable for the case when $\lam$ is weakly compact; our definition includes a slight addition to Hamkins' original definition, as we demand some closure of the target model.

\begin{definition}\label{def:wcld}  Suppose that $\lam$ is weakly compact. A \emph{weakly compact Laver diamond} is a function $\mf{l}:\lam\lra V_\lam$ so that for any $A\in H(\lam^+)$ and any transitive set $M$ of size $\lam$ which satisfies enough of $\zfc$, which contains $A$ and $\mf{l}$ as parameters, and which is closed under $<\lam$-sequences, there is a transitive set $N$ of size $\lam$ and an elementary embedding $k:M\lra N$ so that $\crit(k)=\lam$, $k(\mf{l})(\lam)=A$, and so that $N$ is also closed under $<\lam$-sequences.
\end{definition}

\begin{remark}\label{rmk:ContainsH}
    Note that in the previous definition, $M$ will satisfy enough of $\zfc$ to code every element of $H(\lam^+)\cap M$ as a subset of $\lam$. From this, we can see that if $k:M\lra N$ is as in the above definition, then $M\cap H(\lam^+)\seq N$.
\end{remark}

Hamkins showed in \cite{Hamkins:LaverDiamond} that if $\lam$ is weakly compact, then after forcing with Woodin's fast function forcing $\bb{F}_\lam$ at $\lam$, $\lam$ is still weakly compact, and there is a function $\mf{l}$ which satisfies all of Definition \ref{def:wcld}, except the requirement that $N$ is $<\lam$-closed. Here we provide a variation of Hamkins' proof which includes the closure assumption explicitly as well as additional features which will be useful in later sections when we want to preserve large cardinal properties of $\ka<\lam$.

\begin{definition} Suppose that $\mu$ is inaccessible. Woodin's \emph{fast function forcing} at $\mu$, denoted $\bb{F}_\mu$, consists of partial functions $p:\mu\lra\mu$ of size $<\mu$ so that for all $\ga\in\dom(p)$
\begin{enumerate}
\item $\ga$ is inaccessible;
\item $p[\ga]\seq\ga$;
\item $|p\res\ga|<\ga$.
\end{enumerate}
Conditions are ordered by reverse inclusion.

If $\bar{\mu}<\mu$ is also inaccessible, then $\bb{F}_{[\bar{\mu},\mu)}$ is the variant in which we restrict to conditions with domain contained in $[\bar{\mu},\mu)$, and $\bb{F}_{(\bar{\mu},\mu)}$ is the variant in which we restrict to conditions with domain contained in $(\bar{\mu},\mu)$.
\end{definition}

Note that conditions (1),(2), and (3) together are equivalent to (1) and the statement that $p[\ga]$ is a bounded subset of $\ga$.

There is a natural way to split up the fast function forcing $\bb{F}_\mu$ below a condition $p$ and a point $\ga$ in the domain of $p$. Namely, $p=p_0\cup p_1$, where $p_0:=p\res\ga$ and $p_1:=p\res[\ga,\mu)$. By definition, $p[\ga]\seq\ga$ and $|p\res\ga|<\ga$, and therefore $p_0\in\bb{F}_\ga$. Moreover, $p_1\in\bb{F}_{[\ga,\mu)}$. Applying this to conditions below $p$, we see that $\bb{F}_\mu/p$ is isomorphic to $(\bb{F}_\ga/p_0)\times(\bb{F}_{[\ga,\mu)}/p_1)$. Similarly, in the case where $\bar{\mu}<\mu$, $p\in\bb{F}_{(\bar{\mu},\mu)}$, and $\ga\in\dom(p)$, we see that $\bb{F}_{(\bar{\mu},\mu)}/p$ is isomorphic to $(\bb{F}_{(\bar{\mu},\ga)}/p_0)\times(\bb{F}_{[\ga,\mu)}/p_1)$.

\begin{fact}\label{ffdc} $\bb{F}_{[\ga,\mu)}$ and $\bb{F}_{(\ga,\mu)}$ are both $\ga^+$-directed closed. Moreover, they are both $\mu^+$-c.c. (but not $\mu$-c.c.).
\end{fact}

Our main goal is to show that $\bb{F}_{(\ka,\lam)}$ forces the existence of a weakly compact Laver diamond on $\lam$, in the sense of Definition \ref{def:wcld}. Before the proof, we have a new lemma about preserving closure:

\begin{lemma}\label{lemma:ffPreservesClosure}
Suppose that $\lam$ is Mahlo and that $M$ is a transitive set of size $\lam$ which is  closed under $<\lam$ sequences and which satisfies enough of $\zfc$. Then for any $\ka<\lam$, $\bb{F}:=\bb{F}_{(\ka,\lam)}$ forces that $M[\dot{G}_{\bb{F}}]$ is closed under $<\lam$ sequences.\footnote{Note that since $\bb{F}$ is not $\lam$-c.c., but only $\lam^+$-c.c., this doesn't follow immediately by citing the standard arguments for preserving closure (see \cite{CummingsHandbook}).}
\end{lemma}
\begin{proof}
It suffices to show that $\bb{F}$ forces that $M[\dot{G}_{\bb{F}}]$ is closed under $<\lam$-sequences of ordinals, since $M[\dot{G}_{\bb{F}}]$ is forced to satisfy a fragment of $\zfc$ which is sufficiently large to code every set as a set of ordinals. Let $\rho:=M\cap\text{Ord}=M\cap\lam^+\in\lam^+$.

Fix a condition $p$, as well as an ordinal $\nu<\lam$ and an $\bb{F}$-name $\dot{g}$ for a function from $\nu$ to $\rho$. We find an extension of $p$ which forces that $\dot{g}\in M[\dot{G}_{\bb{F}}]$.

Since $|p|<\lam$, the range of $p$ is bounded below $\lam$. Since $\lam$ is Mahlo, let $\ga<\lam$ be an inaccessible above both $\sup(\ran(p))$ and $\nu$, and then set\footnote{The value of $q(\ga)$ is not important, but the fact that $\ga\in\dom(q)$ is important, in light of condition (2) in the definition of $\bb{F}$.} $q:=p\cup\lb\la\ga,\ga\ra\rb$. For any $r\leq q$, let $r_0$ denote $r\res\ga$, and let $r_1$ denote $r\res[\ga,\lam)$. Note that $\bb{F}/q$ is isomorphic to the product $\bb{F}_{(\ka,\ga)}/q_0\times\bb{F}_{[\ga,\lam)}/q_1$.

Let $G_0\times G_1$ be generic for $\bb{F}/q$ (and hence for $\bb{F}$), and let $g:=\dot{g}[G_0\times G_1]$. Since $\bb{F}_{[\ga,\lam)}$ is $\ga^+$-closed in $V$ and $\bb{F}_{(\ka,\ga)}$ is $\ga^+$-c.c. in $V$ (see Fact \ref{ffdc}), Easton's lemma implies that $\bb{F}_{[\ga,\lam)}$ is $\ga^+$-distributive in $V[G_0]$. Since $\dom(g)=\nu<\ga$, $g$ is a member of $V[G_0]$. Let $\dot{h}$ be an $\bb{F}_{(\ka,\ga)}$-name for a function from $\nu$ to $\rho$ so that $\dot{h}[G_0]=g$. Let $r$ be an extension of $q$ which forces in $\bb{F}/q$ that $\dot{h}[\dot{G}_0]=\dot{g}$.

Recalling that $\bb{F}_{(\ka,\ga)}$ is $\ga^+$-c.c., we have that it is certainly $\lam$-c.c. Thus $\bb{F}_{(\ka,\ga)}$ forces that $M[\dot{G}_0]$ is closed under $<\lam$-sequences. We may therefore find an extension $u$ of $r\res\ga$ in $\bb{F}_{(\ka,\ga)}$ so that $u$ forces that $\dot{h}\in M[\dot{G}_0]$.

To finish, let $u^*$ be the $\bb{F}$-condition $u\cup (r\res[\ga,\lam))$, a condition which extends $r$. Then $u^*$ forces that $\dot{h}[\dot{G}_0]=\dot{g}$, and $u^*$ therefore forces that $\dot{g}\in M[\dot{G}_0]\seq M[\dot{G}_0\times\dot{G}_1]$. This completes the proof.
\end{proof}

The proof of Lemma \ref{lemma:preserveWC} is a prototype for the proof of the next result:

\begin{lemma}\label{lemma:wcld} (Hamkins \cite{Hamkins:LaverDiamond}, with some modifications) Suppose that $\ka$ is inaccessible and that $\lam>\ka$ is weakly compact. Then after forcing with $\bb{F}_{(\ka,\lam)}$ $\lam$ is weakly compact, and there exists a weakly compact Laver diamond (as in Definition \ref{def:wcld}) for $\lam$.
\end{lemma}

\begin{proof} Let $\bb{F}:=\bb{F}_{(\ka,\lam)}$.
The proof has three stages: first we show how to lift weakly compact embeddings $k:M\lra N$ for $M,N$, and $k$ in $V$. Then we show, given some $M'$ in $V[\bb{F}]$, how to create a weakly compact embedding $k':M'\lra N'$. Finally, we show how to massage the generic fast function into the desired guessing function.

Let $k:M\lra N$ be a weakly compact embedding with $\crit(k)=\lam$, and let $f$ be the generic fast function. We show how to lift $k$. For each $\al<k(\lam)$, let $p_\al:=\lb\la\lam,\al\ra\rb$, a condition in $k(\bb{F})$. Then $k(\bb{F})/p_\al$ factors in $N$ as $\bb{F}\times(\bb{F}_{[\lam,k(\lam))}/p_\al)$. The poset $\bb{F}_{[\lam,k(\lam))}/p_\al$ is $\lam$-closed in $V$, since it is $\lam$-closed in $N$ and since $N$ is closed under $<\lam$-sequences. Moreover, it has $\lam$-many dense sets in $N$, since $|N|=\lam$. Thus we may build an $N$-generic $f_{\text{tail}}$ for $\bb{F}_{[\lam,k(\lam))}/p_\al$. Since $f$ is $V$-generic over $\bb{F}$ and $N[f_{\text{tail}}]\subset V$, $f$ is still $\bb{F}$-generic over $N[f_{\text{tail}}]$. Thus $f\cup f_{\text{tail}}$ is $N$-generic for $k(\bb{F})$. Since $k[f]=f$, we may lift $k$ to $k:M[f]\lra N[k(f)]$, where $k(f)=f\cup f_{\text{tail}}$.

Now suppose that $\dot{M}'$ is an $\bb{F}$-name for a transitive set of size $\lam$ which satisfies enough of $\zfc$ and which is closed under $<\lam$-sequences. Since $\bb{F}$ is $\lam^+$-c.c., $\dot{M}'\in H(\lam^+)$. Now let $M$ be a transitive set of size $\lam$ in $V$ with $\dot{M}'\in M$, which satisfies enough $\zfc$, and which is closed under $<\lam$-sequences. Let $k:M\lra N$ be a weakly compact embedding. As in the previous paragraph, in $V[f]$, we lift $k$ to $k:M[f]\lra N[k(f)]$, where $k(f)=f\cup f_{\text{tail}}$, and where $f_{\text{tail}}$ is a filter in $V$ which is $N$-generic over $\bb{F}_{[\lam,k(\lam))}$. Then $M':=\dot{M}'[f]\in M[f]$, and $k\res M'$ is an elementary embedding from $M'$ to $k(M')$. Since $M'$ is $<\lam$-closed in $V[f]$ (by assumption), we certainly have that $M[f]\models ``M'\text{ is closed under }<\lam\text{-sequences}"$. So $N[k(f)]$ satisfies that $k(M')$ is closed under $<k(\lam)$-sequences, and hence closed under $<\lam$-sequences. To finish showing that $k(M')$ is actually closed under $<\lam$-sequences from $V[f]$, we check that $N[k(f)]$ is closed under $<\lam$-sequences from $V[f]$. Indeed, $N$ is closed under $<\lam$-sequences of ordinals from $V$, and since $N[f_{\text{tail}}]\seq V$ has the same ordinals as $N$, $N[f_{\text{tail}}]$ is also closed under $<\lam$-sequences of ordinals from $V$. By Lemma \ref{lemma:ffPreservesClosure}, $N[f_{\text{tail}}][f]=N[k(f)]$ is closed under $<\lam$-sequences of ordinals from $V[f]$. And finally, since $N[k(f)]$ satisfies enough of the Axiom of Choice, it is closed under $<\lam$-sequences from $V[f]$. Thus we know that $\lam$ is still weakly compact in the extension.

We finish by showing how to obtain the weakly compact Laver diamond from the generic fast function $f$. In $V$, let $\vec{a}=\la a_\xi:\xi<\lam\ra$ be a well-ordering of $V_\lam$ so that for all inaccessible $\ga<\lam$, $V_\ga$ is enumerated by $\la a_\xi:\xi<\ga\ra$. Working now in $V[f]$, we define $\mf{l}$. Given $\nu<\lam$, we set $\mf{l}(\nu)=\es$ unless the following conditions are met: 
\begin{itemize}
\item $\nu>\ka$ is inaccessible,
\item $a_{f(\nu)}$ is an $\bb{F}_{(\ka,\nu)}$-name, and 
\item $f\res\nu$ is $V$-generic over $\bb{F}_{(\ka,\nu)}$. 
\end{itemize}
If these conditions are met, we set $\mf{l}(\nu)=a_{f(\nu)}[f\res\nu]$. 

We claim that $\mf{l}$ is a weakly compact Laver diamond. Thus let $A\in H^{V[f]}(\lam^+)$, noting that $A=\dot{A}[f]$ for some $\dot{A}\in H^V(\lam^+)$. Let $M$ be a transitive set of size $\lam$ satisfying enough of $\zfc$ so that $\dot{A}$ and $\vec{a}$ are in $M$ and so that $M$ is $<\lam$-closed. Additionally, let $k:M\lra N$ be a weakly compact embedding. Note that $\dot{A}\in N$, since $\dot{A}\in M\cap H(\lam^+)$ (see Remark \ref{rmk:ContainsH}), and therefore we may let $\be<k(\lam)$ so that $\dot{A}=k(\vec{a})(\be)$.  As shown earlier in the proof, we may lift $k$ to $k:M[f]\lra N[k(f)]$ where $k(f)$ is such that $k(f)(\lam)=\be$. Noting that $\mf{l}\in M[f]$ (being definable from $\vec{a}$ and $f$), we see that 
$$
k(\mf{l})(\lam)=k(\vec{a})(k(f)(\lam))[k(f)\res\lam]=k(\vec{a})(\be)[f]=\dot{A}[f]=A.
$$
This shows that $\mf{l}$ satisfies the desired property with respect to extensions of such $M$ coming from $V$. The full result, which applies to all appropriate $M'$ in $V[f]$, is obtained as in the previous paragraph.
\end{proof}

\subsection{A Summary of Some Results about Club Adding and Guessing Functions.}\hfill

We end this section by collecting a few more results about club adding after versions of Mitchell forcing built relative to a guessing function. The first item of the next fact summarizes Theorem 1.2 of \cite{GLS:8fold}.

\begin{fact}\label{fact:GLS} Let $\ka$ be a regular cardinal with $\ka^{<\ka}=\ka$, and fix a weakly compact cardinal $\lam>\ka$. Let $\mf{l}$ be a weakly compact Laver diamond on $\lam$.
\begin{enumerate}
    \item There is an $\bb{M}^*_{\mf{l}}(\ka,\lam)$-name $\dot{\ps}_{\lam^+}$ for a standard club adding iteration of length $\lam^+$ so that $\bb{M}^*_{\mf{l}}(\ka,\lam)\ast\dot{\ps}_{\lam^+}$ forces $\neg\AP(\lam)\we\CSR(\lam)$.
    \item An analogous claim to item (1) holds, with respect to the Mitchell forcing $\bb{M}_{\mf{l}}(\ka,\lam)$, to obtain the configuration $\AP(\lam)\we\CSR(\lam)$.\footnote{Note that while this is not the exact poset used in \cite{GLS:8fold} to obtain this configuration (which corresponds to Theorem 1.1 of \cite{GLS:8fold}), the proof that it does satisfy this configuration is a straightforward modification of the proof of Theorem 1.1 of \cite{GLS:8fold}.}
\end{enumerate}
Moreover, in all of the above cases, the forcings preserve all cardinals $\leq\ka^+$, all cardinals $\geq\lam$, and force that $\lam=\ka^{++}$. Finally, the club adding posets are forced by the corresponding Mitchell posets to be $\ka^+$-directed closed and $\lam$-distributive.
\end{fact}

The remaining items in this section are concerned with the interactions between the Mitchell-type posets from Fact \ref{fact:GLS} and weakly compact embeddings for a weakly compact cardinal $\lam$. The first result shows that for each $\al<\lam^+$, we can build a parameter $P$ so that if $k:M\lra N$ is a weakly compact embedding and $P\in M$, then $k\res\al\in N$. This was a necessary ingredient in the construction of master conditions for lifting embeddings in \cite{GLS:8fold}, which we are making use of in this paper.

\begin{lemma}\label{lemma:parameter}
Suppose that $\lam$ is weakly compact and that $\al<\lam^+$. Let $\vec{M}^\al=\la M_\ga:\ga<\lam\ra$ be a continuous, increasing sequence of elementary submodels of $H(\lam^+)$ of size $<\lam$, each of which contains $\al$ as an element. Suppose that $k:M\lra N$ is a weakly compact embedding with $\crit(k)=\lam$, and that $\vec{M}^\al\in M$. Then $k\res\al\in N$. 
\end{lemma}
\begin{proof}
Let $\vec{M}$ abbreviate $\vec{M}^\al$. By passing to a club of indices in $\lam$, we may assume that $M_\ga\cap\lam\in\lam$ for all $\ga<\ka$. Let $M'_\lam:=k(\vec{M})(\lam)$, and note that $M'_\lam=\bigcup_{\ga<\lam}k[M_\ga]$ by the elementarity of $k$ and the continuity of $k(\vec{M})$. We next observe that $\bigcup_{\ga<\lam}M_\ga$ is transitive: given any set $X$ in this union, $X$ is a member of some $M_\de$, with $\de<\lam$. Moreover, $M_\de$ contains a surjection $\vp_X$ from $\lam$ onto $X$, since $X\in H(\lam^+)$. Since $\lam\seq\bigcup_{\ga<\lam}M_\ga$, we see that $X=\ran(\vp_X)$ is also a subset of this union.

In light of this observation, we conclude that the restriction of $k^{-1}$ to $M'_\lam$ is the transitive collapse of $M'_\lam$ (onto $\bigcup_{\ga<\ka}M_\ga$) and hence is a member of $N$. Consequently, $N$ contains the restriction of $k$ to $\bigcup_{\ga<\lam}M_\ga$, and hence the restriction of $k$ to $\al$, since $\al$ is a subset of this union.
\end{proof}

For the next item, we introduce a small variant of some standard notation:

\begin{notation}\label{notation:parallel}
Given a poset $\Q$ and $q\in\Q$, recall that $\Q/q$ consists of all $r\in\Q$ so that $r\leq q$. If $\Q$ is a two-step poset $\Q_0\ast\dot{\Q}_1$ and $\dot{q}$ is a $\Q_0$-name for a condition in $\dot{\Q}_1$, then we use $\Q/\dot{q}$ to abbreviate $\Q/(1_{\Q_0},\dot{q})$.
\end{notation}

The next lemma is where we need the conclusion of Lemma \ref{lemma:parameter}, in order to construct a master condition. Recall Lemma \ref{lemma:tailprojections} for the notation.

\begin{lemma}\label{lemma:ce}
Let $\mf{l}$ be a weakly compact Laver diamond, and let $\bb{M}$ stand for either $\bb{M}^*_{\mf{l}}(\ka,\lam)$ or $\bb{M}_{\mf{l}}(\ka,\lam)$. Let $k:M\lra N$ be a weakly compact embedding with $\crit(k)=\lam$, where $M$ contains the following objects as parameters: $\mf{l}$, an ordinal $\al<\lam^+$, the poset $\bb{M}$, the name $\dot{\ps}_\al$ for the first $\al$-many stages in an iterated club adding to achieve $\CSR(\lam)$ (see Fact \ref{fact:GLS}), and the parameter $\vec{M}^\al$ from Lemma \ref{lemma:parameter}. Suppose additionally that $k(\mf{l})(\lam)=\dot{\ps}_\al$. 

Then there is, in $N$, a $k(\bb{M})$-name $\dot{q}^\al$ for a condition in $k(\dot{\ps}_\al)$ so that for any $G^*$ which is $N$-generic over $k(\bb{M})$ (and recalling Lemma \ref{lemma:tailprojections}), if we factor $G^*$ as $G\ast H_\al\ast G^*_{\text{tail}}$, then $q^\al$ is the greatest lower bound of $k[H_\al]$. In particular, $q^\al$ is a strong master condition for the embedding, and so after forcing below $q^\al$, we may further lift $k$ to have domain $M[G\ast H_\al]$.
\end{lemma}
\begin{proof}
    See, for example, Claims 4.4 and 4.9 from \cite{GLS:8fold}.
\end{proof}

\section{Obtaining and Preserving Guiding Generics}\label{sec:ggs}

In this section we build our guiding generics and show that our club adding posets preserve the genericity of these guiding generics. Throughout this section, we assume that $V$ is a model of ($\zfc$ plus) the $\gch$ in which $\ka$ is supercompact, and $\lam>\ka$ is the least weakly compact above $\ka$. Next, we fix a supercompactness measure $\cal{U}$ on $P_\ka(\lam)$, letting $N$ denote $\mathrm{Ult}(V,\cal{U})$ and letting $j$ denote the corresponding ultrapower embedding. Finally, let $\bar{U}$ denote the projection of $\cal{U}$ onto a normal measure on $\ka$.

\subsection{Obtaining the Guiding Generics}

In this subsection, we show how to build the desired guiding generics for the later arguments of Section \ref{sec:wc} when we move our constructions down to $\aleph_\om$. The difficulty arises since we will later be working in cases where $M$ is the ultrapower of some $V[G]$ by a normal measure on $\ka$, and when $2^\ka=\ka^{++}$. This implies that $M$ is insufficiently closed to build the guiding generic directly. Following \cite{GS:sch} and \cite{SU:tree}, we will first create a guiding generic over a generic extension $N^*$ of $N$. We then pull this back along the factor map $k$ to get a filter over the corresponding collapse poset in $M$. However, to ensure that this is a guiding \emph{generic}, we need that the factor map $k$ has a high critical point. We will secure this by a preparatory iteration which incorporates many instances of Cohen forcing (as well as versions of the forcing which we wish to do at stage $\ka$), the purpose of which is to supply the flexibility needed to arrange this feature of $k$. There are two circumstances in which we will need these guiding generics corresponding to the $\AP(\lam)$ and $\neg\AP(\lam)$ cases of Theorem \ref{thm:WeaklyCompactCases}.

Since we wish to do, at $\ka$, a version of Mitchell forcing based upon a weakly compact Laver diamond, we will prepare the universe by forcing with many instances of such a forcing below $\ka$. We will need to incorporate Woodin's fast function forcing, in addition to the Mitchell poset and the Cohen poset.

Given a regular cardinal $\mu$ and a weakly compact cardinal $\nu>\mu$, $\bb{F}_{(\mu,\nu)}$ forces that there is a weakly compact Laver diamond $\dot{\mf{l}}_\nu$ on $\nu$, by Lemma \ref{lemma:wcld}. If $\nu$ is the least weakly compact above $\mu$, then we use $\bb{A}(\mu)$ to denote $\Add(\mu,[\nu,\nu^+))$. We also define (recalling Definitions \ref{def:MGuessAP} and \ref{def:MGuessNotAP}) $\bb{B}(\mu)$ to be the poset
$$
\bb{B}(\mu):=\bb{F}_{(\mu,\nu)}\ast(\bb{M}_{\dot{\mf{l}}_\nu}(\mu,\nu)\times\bb{A}(\mu)),
$$
and we define $\bb{B}^*(\mu)$ to be the poset 
$$
\bb{B}^*(\mu):=\bb{F}_{(\mu,\nu)}\ast(\bb{M}^*_{\dot{\mf{l}}_\nu}(\mu,\nu)\times\bb{A}(\mu)).
$$

For each $\al<\ka$, let $\lam_\al$ denote the least weakly compact cardinal above $\al$. Let $Z\in\bar{U}$ so that for each $\be\in Z$, $\be$ is an inaccessible closed under $\al\mapsto\lam_\al$ and $\be$ is $<\lam_\be$-supercompact. We will recursively define Easton support iterations $\bb{E}_{\ka+1}$ and $\bb{E}^*_{\ka+1}$ at stages in $Z\cup\lb\ka\rb$. Given $\be\in Z$, and having defined $\bb{E}_\be$ (respectively $\bb{E}^*_\be)$, we have that $\bb{E}_\be$ (resp. $\bb{E}^*_\be)$ has size $\be<\lam_\be$. Hence $\bb{E}_\be$ (resp. $\bb{E}^*_\be$) preserves that $\lam_\be$ is weakly compact by Lemma \ref{lemma:preserveWC}. We then force with the $\bb{E}_\be$-name (resp. $\bb{E}^*_\be$-name) for $\bb{B}(\be)$ (resp. $\bb{B}^*(\be)$). This completes the definitions. The rest of the argument will be stated for $\bb{E}_{\ka+1}$ since the same argument works for $\bb{E}^*_{\ka+1}$.

We summarize some of the useful facts about $\bb{E}_{\ka+1}$, as well as $j$ and $N$ (which we specified at the beginning of the section).
\begin{itemize}
\item $|j(\ka)|^V=|j(\lam)|^V=|j(\lam^+)|^V=\lam^+$; 
\item $j$ is continuous at $\lam^+$;
\item $\bb{E}_\ka$ is $\ka$-c.c., and $\bb{E}_{\ka+1}$ is $\lam^+$-c.c. (since the fast function forcing at $\lam$ is only $\lam^+$-c.c.);
\item $j(\bb{E}_{\ka+1})$ can be factored as
$$
j(\bb{E}_{\ka+1})=\bb{E}_{\ka+1}\ast\dot{\bb{E}}^N_\text{tail}\ast j(\bb{B}(\ka)),
$$
where $\dot{\bb{E}}^N_\text{tail}$ is forced by $\bb{E}_{\ka+1}$ over $N$ to be $\lam^+$-closed (since the next stage in $j(Z)$ above $\ka$ is greater than $\lam_\ka=\lam$), and where $j(\bb{F}_{(\ka,\lam)}\ast\bb{M}_{\dot{\mf{l}}}(\ka,\lam))$ is forced to be $j(\ka)$-directed closed, using Lemma \ref{ffdc} and clause (5) of Definition \ref{def:MGuessAP}.
\end{itemize}

To begin, let $G_{\ka+1}$ be $V$-generic over $\bb{E}_{\ka+1}$. Since $N$ is closed under $\lam$-sequences from $V$ and $\bb{E}_{\ka+1}$ is $\lam^+$-c.c., $N[G_{\ka+1}]$ is closed under $\lam$-sequences from $V[G_{\ka+1}]$. Since $\bb{E}^N_{\text{tail}}$ is $\lam^+$-closed in $N[G_{\ka+1}]$, and since $\bb{E}^N_{\text{tail}}$ has $j(\ka)$-many antichains in $N[G_{\ka+1}]$, we may work in $V[G_{\ka+1}]$ to build an $N[G_{\ka+1}]$-generic filter, say $H$, over $\bb{E}^N_{\text{tail}}$. This allows a lift $j:V[G_\ka]\lra N[G_{\ka+1}\ast H]$ defined in $V[G_{\ka+1}]$. We abbreviate $G_{\ka+1}\ast H$ by $G_{j(\ka)}$.

Now we work to construct a master condition for the two-step iteration of the fast function forcing and the Mitchell forcing. Let $f\ast (I_0\times I_1)$ be the $V[G_\ka]$-generic filter over $\bb{B}(\ka)$ added by $G_{\ka+1}$; so $I_0$ is generic for the Mitchell forcing and $I_1$ for the Cohen forcing. The model $N[G_{j(\ka)}]$ is still closed under $\lam$-sequences from $V[G_{\ka+1}]$. Thus $j[f\ast I_0]\in N[G_{j(\ka)}]$ since $f\ast I_0$ has size $\lam$. Moreover, $j(\bb{F}_{(\ka,\lam)}\ast\bb{M}_{\dot{\mf{l}}}(\ka,\lam))$ is $j(\ka)$-directed closed. Since $\lam<j(\ka)$, we have that $j[f\ast I_0]$ has a lower bound in $j(\bb{F}_{(\ka,\lam)}\ast\bb{M}_{\dot{\mf{l}}}(\ka,\lam))$. Next, in $N[G_{j(\ka)}]$, $j(\bb{F}_{(\ka,\lam)}\ast\bb{M}_{\dot{\mf{l}}}(\ka,\lam))$ is $j(\lam)^+$-c.c. and has size $j(\lam)$. Thus $N[G_{j(\ka)}]$ satisfies that there are at most $2^{j(\lam)}=j(\lam)^+$-many antichains for this poset. Since $|j(\lam)^+|^V=\lam^+$ and since $\lam^+<j(\ka)$, we may work in $V[G_{\ka+1}]$ to build an $N[G_{j(\ka)}]$-generic filter over $j(\bb{F}_{(\ka,\lam)}\ast\bb{M}_{\dot{\mf{l}}}(\ka,\lam))$ which contains $j[f\ast I_0]$. This allows us to lift $j$ to $j:V[G_{\ka}][f\ast I_0]\to N[G_{j(\ka)}][j(f\ast I_0)]$.

It remains to build the desired $N[G_{j(\ka)}\ast j(f\ast I_0)]$-generic filter $I^*_1$ over $j(\bb{A}(\ka))$. Recalling that our goal is to ensure a high critical point for the factor map, we will represent every $\eta<j(\ka)$ in the form $j(f)(\ka)$ for some $f:\ka\lra\ka$ in $V[G_{\ka+1}]$. Let $\la x_\eta:\eta\in[\lam,\lam^+)\ra$ enumerate $j(\ka)$, and recall that we are viewing conditions in $\Add(\ka,\lam^+\bsl\lam)$ as functions from $(\lam^+\bsl\lam)\times\ka$ into $2$. For each $\eta\in[\lam,\lam^+)$, let $f_\eta$ be the $\eta$th Cohen-generic function from $\ka$ to 2 added by $I_1$. Next, for each $\zeta\in[\lam,\lam^+)$, let $\vp_\zeta$ be the set
$$
\bigcup_{\eta\in[\lam,\zeta)} (\lb j(\eta)\rb \times f_\eta) \cup\lb\la j(\eta),\ka,x_\eta\ra\rb.
$$
Note that $\vp_\zeta\in N[G_{j(\ka)}\ast j(f\ast I_0)]$ by the closure of that model and also that $\vp_\zeta\in j(\bb{A}(\ka))$. Finally, note that if $\zeta<\zeta'$, then the restriction of $\vp_{\zeta'}$ to $j[[\lam,\zeta)]\times(\ka+1)$ equals $\vp_{\zeta}$.

To obtain $I^*_1$, we will first obtain \emph{some} generic, and then we will alter it. Thus let $J$ be some $N[G_{j(\ka)}\ast j(f\ast I_0)]$-generic filter over $j(\bb{A}(\ka))$ in $V[G_{\ka+1}]$, and for each $\zeta<\lam^+$, let $J_\zeta$ be the restriction of $J$ to $\Add(j(\ka),[j(\lam),j(\zeta)))$. Since $j$ is continuous as $\lam^+$, we have that $J=\bigcup_{\zeta<\lam^+}J_\zeta$. Now define $I^*_{1,\zeta}$ to consist of all conditions gotten by taking a condition in $J_\zeta$ and altering it to agree with $\vp_\zeta$. Since we are altering conditions in $J_\zeta$ only on a small set (namely, a set of size $\lam<j(\ka)$), $I^*_{1,\zeta}$ is still generic. Moreover, $\zeta<\zeta'$ implies that $I^*_{1,\zeta}\seq I^*_{1,\zeta'}$. Once more using the continuity of $j$ at $\lam^+$, and also using the chain condition of $j(\bb{A}(\ka))$, we see that $I^*_1:=\bigcup_{\zeta<\lam^+}I^*_{1,\zeta}$ is an $N[G_{j(\ka)}\ast j(f\ast I_0)]$-generic filter over $j(\bb{A}(\ka))$.

Now we lift $j$ to $j:V[G_{\ka+1}]\lra N[G_{j(\ka)}\ast j(f) \ast j(I_0\times I_1)]$, where $j(I_1)=I^*_1$, and we abbreviate the latter model by $N^*$. Let $\cal{U}^*$ be the supercompactness measure on $P_\ka(\lam)$ generated by this lift of $j$. For each $\eta\in[\lam,\lam^+)$, let $F_{j(\eta)}$ be the $j(\eta)$th function from $j(\ka)$ to 2 added by $I^*_1$. Then for each $\eta\in[\lam,\lam^+)$, $F_{j(\eta)}(\ka)=x_\eta$, and $j(f_\eta)=F_{j(\eta)}$. Thus $j(f_\eta)(\ka)=x_\eta$ for all $\eta\in[\lam,\lam^+)$.

Now let $U$ be the normal ultrafilter on $\ka$ induced by $\cal{U}^*$. Let $M$ be the ultrapower of $V[G_{\ka+1}]$ by $U$, let $j_0$ be the ultrapower map, and let $k$ from $M$ to $N^*$ be the factor map taking $[f]_U$ to $j(f)(\ka)$. By construction, $j(\ka)$ is a subset of the range of $k$. Since $j(\ka)$ is also a member of the range of $k$, we have that the critical point of $k$ is above $j(\ka)$.

\begin{lemma}\label{lemma:gg}
In $V[G_{\ka+1}]$, there is a $(U,2)$-guiding generic. Additionally, if $G^*_{\ka+1}$ is $V$-generic over $\bb{E}^*_{\ka+1}$, then there is also a $(U,2)$-guiding generic in $V[G^*_{\ka+1}]$.
\end{lemma}
\begin{proof}
We introduce the abbreviations 
$$
\C:=\Col^M(\ka^{++},<j_0(\ka))\text{ and } \C^*:=\Col^{N^*}(\ka^{++},<j(\ka)).
$$ 
Note that $k(\C)=\C^*$. In $V[G_{\ka+1}]$, we may build an $N^*$-generic $\tilde{G}$ over $\C^*$. Since $\C$ is $j_0(\ka)$-c.c. in $M$ and $\crit(k)>j(\ka)\geq j_0(\ka)$, every max antichain $A$ of $\C$ which is a member of $M$ satisfies that $j_0(A)=j_0[A]$. Thus $k^{-1}[\tilde{G}\cap\ran(k)]$ is an $M$-generic filter over $\C$. 
\end{proof}

\subsection{Preserving the Guiding Generics by Club Adding}

In the next section, when we build our models involving various activity on $\aleph_{\om+2}$, the final components of our forcings will be products of the form $\Pc(U,G^g)\times\ps_{\lam^+}$, where $\ps_{\lam^+}$ is a club adding iteration from Fact \ref{fact:GLS}. We will need that $\Pc(U,G^g)$ is still well-behaved after forcing with $\ps_{\lam^+}$, which in turn requires that we show that $G^g$ is still a guiding generic after forcing with $\ps_{\lam^+}$. This will follow from the $\lam$-distributivity of $\ps_{\lam^+}$ and a lemma about ultrapowers. The proof is the same in either of the cases discussed in the previous subsection. Thus for the rest of this subsection, we let $V^*$ be either $V[G_{\ka+1}]$ or $V[G^*_{\ka+1}]$ from the previous subsection. In each case, we lifted the original $j$ to $j:V^*\lra N^*$ for some $N^*$. Finally, we let $U$ be the normal measure on $\ka$ derived from $j$, and we set $M:=\ult(V^*,U)$. Let $j_0:V^*\lra M$ be the ultrapower map.

Now fix some $\Q\in V^*$ which is $\lam$-distributive.\footnote{The proof only needs that $\Q$ is $\ka^+$-distributive, but we will only use the lemma in the case of $\lam$-distributivity.} Then $U$ is still a normal measure on $\ka$ in $V^*[\cal{Q}]$, where $\cal{Q}$ is $V^*$-generic over $\Q$. In particular, we can form the ultrapower $M^+:=\ult(V^*[\cal{Q}],U)$, and let $j^+_0:V^*[\cal{Q}]\lra M^+$ be the ultrapower embedding. We will show that $M^+=M[j^+_0(\cal{Q})]$.

\begin{lemma} Suppose that $\Q$ is $\lam$-distributive in $V^*$, and let $\cal{Q}$ be $V^*$-generic over $\Q$. Then
\begin{enumerate}
\item $M\seq M^+$ and $j_0^+$ extends $j_0$;
\item $j_0^+(\cal{Q})$ is an $\ult(V^*,U)$-generic filter over $j_0(\Q)$; and
\item $M^+=M[j^+_0(\cal{Q})]$.
\end{enumerate}
\end{lemma}
\begin{proof}
Throughout the proof, we make the standard abuse of notation in discussing ultrapowers by equivocating between the $U$-equivalence class of a function with domain $\ka$ and its image under the transitive collapse map. We also suppress mention of $U$ in the notation for equivalence classes.

For (1), it suffices to prove that for any $f:\ka\lra V^*$ with $f\in V^*$, $[f]^{V^*}=[f]^{V^*[\cal{Q}]}$, where the superscript indicates the model in which the equivalence class is computed. We will prove this by induction on the max of the ranks of (the transitive collapse of) $[f]^{V^*}$ and $[f]^{V^*[\cal{Q}]}$. The case when this max is zero is immediate. Thus fix an $f$ for which this max is not zero.

For the first inclusion, let $x\in [f]^{V^*}$, so that $x=[g]^{V^*}$ for some $g\in V^*$. Then the rank of $[g]^{V^*}$ is smaller than that of $[f]^{V^*}$. Moreover, $[g]^{V^*}\in [f]^{V^*}$ implies that $[g]^{V^*[\cal{Q}]}\in[f]^{V^*[\cal{Q}]}$. Therefore, the max of the ranks of $[g]^{V^*}$ and $[g]^{V^*[\cal{Q}]}$ is less than the max of the ranks of $[f]^{V^*}$ and $[f]^{V^*[\cal{Q}]}$. By induction, $x=[g]^{V^*}=[g]^{V^*[\cal{Q}]}\in[f]^{V^*[\cal{Q}]}$.

For the other direction, fix some $y\in [f]^{V^*[\cal{Q}]}$, and write $y=[h]^{V^*[\cal{Q}]}$ for some $h:\ka\lra V^*[\cal{Q}]$ with $h\in V^*[\cal{Q}]$. Then $X:=\lb\al<\ka:h(\al)\in f(\al)\rb\in U$. Since $f\in V^*$ and $V^*$ is transitive, $h(\al)\in V^*$ for all $\al\in X$. By modifying $h$ on a $U$-measure zero set, we may assume that for all $\al<\ka$, $h(\al)\in V^*$. Then $h:\ka\lra V^*$ is a function which lives in $V^*[\cal{Q}]$. However, $\Q$ is $\lam$-distributive in $V^*$, and hence $h\in V^*$. Since the max of the ranks of $[h]^{V^*}$ and $[h]^{V^*[\cal{Q}]}$ is less than the max of the ranks of $[f]^{V^*}$ and $[f]^{V^*[\cal{Q}]}$, the induction hypothesis implies that $y=[h]^{V^*[\cal{Q}]}=[h]^{V^*}\in[f]^{V^*}$. This completes the proof of (1).

(2) follows by the elementarity of $j_0^+$.

For (3), we know from (1) that $M\seq M^+$, and since $j^+_0(\cal{Q})\in M^+$, we get $M[j^+_0(\cal{Q})]\seq M^+$. For the other direction, let $y\in M^+$. Then $y=[f]^{V^*[\cal{Q}]}$ for some function $f\in V^*[\cal{Q}]$. Let $\underset{\sim}{f}\in V^*$ be a function so that for all $\al<\ka$, $\underset{\sim}{f}(\al)$ is a $\Q$-name satisfying $\underset{\sim}{f}(\al)[\cal{Q}]=f(\al)$. Then $[\underset{\sim}{f}]^{V^*}$ is a $j_0(\Q)$-name. By (1), $[\underset{\sim}{f}]^{V^*}=[\underset{\sim}{f}]^{V^*[\cal{Q}]}$, and thus $[\underset{\sim}{f}]^{V^*[\cal{Q}]}$ is a $j_0(\Q)$-name. We claim that $[\underset{\sim}{f}]^{V^*[\cal{Q}]}[j^+_0(\cal{Q})]=y$. Recalling that $y=[f]^{V^*[\cal{Q}]}$ and applying Los' theorem in $M^+$, we see that $[\underset{\sim}{f}]^{V^*[\cal{Q}]}[j^+_0(\cal{Q})]=[f]^{V^*[\cal{Q}]}$ iff
$$
\lb\al<\ka:\underset{\sim}{f}(\al)[\cal{Q}]=f(\al)\rb\in U,
$$
which holds by choice of $\underset{\sim}{f}$.
\end{proof}

\begin{Cor}\label{Cor:GGPres} Suppose that $\Q$ is $\lam$-distributive in $V^*$, that $\cal{Q}$ is $V^*$-generic over $\Q$, and that $G^g$ is a $(U,2)$-guiding generic in $V^*$. Then $G^g$ is still a $(U,2)$-guiding generic in $V^*[\cal{Q}]$.
\end{Cor}
\begin{proof}
By definition of $G^g$ being a $(U,2)$-guiding generic, $G^g$ is $M$-generic for the poset $\Col^M(\ka^{++},<j_0(\ka))$. However, $M^+=M[j_0^+(\cal{Q})]$ is a generic extension of $M$ by a $j_0(\lam)$-distributive forcing, and $j_0(\ka)<j_0(\lam)$. Consequently,
$$
\Col^M(\ka^{++},<j_0(\ka))=\Col^{M^+}(\ka^{++},<j_0(\ka)),
$$
and this poset has the same dense sets in $M$ as in $M^+$. Thus $G^g$ is still a $(U,2)$-guiding generic in $V^*[\cal{Q}]$.
\end{proof}

\section{The Models}\label{sec:wc}

The goal of this section is to prove Theorem \ref{thm:WeaklyCompactCases}. We fix a model $V$ of the $\gch$ in which $\ka$ is indestructibly supercompact and in which $\lam>\ka$ is the least weakly compact above $\ka$. We first deal with the cases which do not involve bringing $\ka$ down to $\aleph_\om$ (Subsection \ref{ss:wc}), and then we incorporate collapses into the Prikry forcing to bring $\ka$ down to $\aleph_\om$ (Subsection \ref{ss:wcDown}). Since we are concerned with making the tree property hold at $\lam$ in all of these cases, we will have to  understand how the club adding influences the arguments for the branch lemmas. In order to avoid excessive repetition, we will only prove these branch lemmas for the harder case in which the Prikry forcing involves collapses. However, we will still state the branch lemmas in the non-collapsing case.

\subsection{Without Collapsing}\label{ss:wc} Here we will prove the parts of Theorem \ref{thm:WeaklyCompactCases} which do not involve collapsing to make $\ka$ become $\aleph_\om$. The argument initially considers the set up for the $\AP(\lam)$ and $\neg\AP(\lam)$ cases separately, and then it finishes with an argument which works for either case. 

We may assume, by forcing with the $\ka$-directed closed poset $\bb{F}_{(\ka,\lam)}$ if necessary (see Lemma \ref{lemma:wcld}), that there is a weakly compact Laver diamond $\mf{l}$ on $\lam$. We use $\bb{M}_\mf{l}$ and $\bb{M}^*_\mf{l}$ to abbreviate, respectively, $\bb{M}_\mf{l}(\ka,\lam)$ and $\bb{M}^*_\mf{l}(\ka,\lam)$.

\subsubsection{$\AP(\lam)$ holds: setting up the argument.} Since $\ka$ is indestructibly supercompact and $\bb{M}_{\mf{l}}$ is $\ka$-directed closed, $\ka$ is still supercompact in the extension by $\bb{M}_{\mf{l}}$. In particular, $\ka$ is measurable in this extension, and so we may let $\dot{U}$ be an $\bb{M}_{\mf{l}}$-name for a normal measure on $\ka$.

Consider the forcing
$$
\bb{W}:=\bb{M}_{\mf{l}}\ast(\dot{\ps}_{\lam^+}\times\Prk(\dot{U})),
$$
where $\dot{\ps}_{\lam^+}$ is an $\bb{M}_{\mf{l}}$-name for a standard club adding iteration to achieve $\CSR(\lam)$ (see Fact \ref{fact:GLS}(2)). We will show that $\bb{W}$ forces $\TP(\lam)\we\AP(\lam)\we\CSR(\lam)$, in addition to forcing $\cf(\ka)=\om$ and $\ka^{++}=\lam$.

By Fact \ref{fact:GLS}, the poset $\bb{M}_{\mf{l}}\ast\dot{\ps}_{\lam^+}$ preserves all cardinals $\leq\ka^+$, forces $\lam=\ka^{++}$, and preserves all cardinals $\geq\lam^+$. Moreover, it forces $\AP(\lam)\we\CSR(\lam)$. Since $\ps_{\lam^+}$ is $\lam$-distributive in the $\bb{M}_{\mf{l}}$-extension, $\dot{U}$ is forced by $\bb{M}_{\mf{l}}\ast\dot{\ps}_{\lam^+}$ to still be a normal measure on $\ka$. Thus, in $V[\bb{M}_{\mf{l}}\ast\dot{\ps}_{\lam^+}]$, $\Prk(U)$ preserves all cardinals of that model and singularizes $\ka$. Since it preserves the cardinals of $V[\bb{M}_{\mf{l}}\ast\dot{\ps}_{\lam^+}]$, it preserves $\AP(\lam)$. Moreover, by Theorem \ref{th:CSRPreserve}, $\Prk(U)$ also preserves $\CSR(\lam)$. Thus in the $\bb{W}$-extension, we have that $\cf(\ka)=\om\we\CSR(\ka^{++})\we\AP(\ka^{++})$. It remains to show that $\bb{W}$ forces $\TP(\lam)$.

\begin{remark}
By Fact \ref{nicenamestrick}, it suffices to show that for all $\al<\lam^+$, $\TP(\lam)$ holds after forcing with
$$
\bb{W}_\al:=\bb{M}_{\mf{l}}\ast(\dot{\ps}_\al\times\Prk(\dot{U})).
$$
\end{remark}
 So fix an $\al<\lam^+$ and a $\bb{W}_\al$-name $\dot{T}$ for a $\lam$-tree. We will show that $\dot{T}$ is forced by $\bb{W}_\al$ to have a cofinal branch. 

In $V$, fix a weakly compact embedding $k:M\lra N$ with $\crit(k)=\lam$, where $M$ contains $\mf{l}$, $\bb{W}_\al$, $\dot{T}$, and the parameter $\vec{M}^\al$ from Lemma \ref{lemma:parameter} as elements.  Since $\mf{l}$ is a weakly compact Laver diamond, we also assume that $k(\mf{l})(\lam)=\dot{\ps}_\al$ and that $^{<\lam}N\seq N$. Let $\dot{q}^\al$ be the $k(\bb{M}_{\mf{l}})$-name in $N$ which satisfies the conclusion of Lemma \ref{lemma:ce}. 

Our general strategy is this: we will first construct a complete embedding with domain $\bb{W}_\al$, and then we will show that the resulting quotient of the target poset is forcing equivalent to a more manageable poset. We will then show that this more manageable poset satisfies the requisite branch lemma.

Recall, before reading the next lemma, that there is a dense embedding from $k(\bb{M}_\mf{l})$ to $\bb{M}_\mf{l}\ast\dot{\ps}_\al\ast\dot{\N}$ (see Lemma \ref{lemma:tailprojections}).

\begin{lemma}\label{lemma:ce,AP} The map $\iota$ given by
$$
(m,\dot{q},\dot{r})\mapsto (m,\dot{q},1_{\dot{\bb{N}}},\dot{r})
$$
is a complete embedding from $\bb{W}_\al$ to $(\bb{M}_\mf{l}\ast\dot{\ps}_\al\ast\dot{\bb{N}})\ast k(\Prk(\dot{U}))$.
\end{lemma}
\begin{proof}
We will verify that conditions (1)-(4) of Definition \ref{def:ce} hold, though we first check that $\iota$ maps from the poset on the left to the one on the right. The only part to check is that $(m,\dot{q},1_{\dot{\N}})$ forces that $\dot{r}\in k(\Prk(\dot{U}))$. Given an $N$-generic $G\ast H_\al\ast G^*_\N$ over $\bb{M}_\mf{l}\ast\dot{\ps}_\al\ast\dot{\bb{N}}$ with $(m,\dot{q})\in G\ast H_\al$, and letting $G^*$ denote its equivalent generic filter over $k(\bb{M})$, we may lift $k$ to $k:M[G]\to N[G^*]$. Then, letting $r:=\dot{r}[G]$, we have that $r\in\Prk(U)$, since $(m,\dot{q},\dot{r})$ is a condition. So by the elementarity of $k$, we conclude that $r=k(r)\in k(\Prk(U))$.

Item (1) of Definition \ref{def:ce} holds by definition. Item (2) follows by an argument similar to the one in the previous paragraph, using the elementarity of the lift of $k$. Thus we address item (3), proving the contrapositive. Suppose that $(m_i,\dot{q}_i,\dot{r}_i)$ for $i\in 2$ are conditions in $\bb{W}_\al$ whose $\iota$-images are compatible, and we show that they are compatible in $\bb{W}_\al$. Fixing a witness, let $(m,\dot{q},\dot{n},\dot{r})$ extend each $(m_i,\dot{q}_i,1_{\dot{\N}},\dot{r}_i)$.  Force over $N$ below $(m,\dot{q},\dot{n})$ to obtain an $N$-generic $G^*=G\ast H_\al\ast G^*_{\N}$. Now lift $k:M[G]\to N[G^*]$. Then $r$ is below $r_0$ and $r_1$ in $k(\Prk(U))$. By the elementarity of $k$, and since $k(r_i)=r_i$, there is a condition $r'\in M[G]$ extending $r_0$ and $r_1$. Let $(m',\dot{q}')$ be a condition in $G\ast H_\al$ extending $(m,\dot{q})$ which forces that $\dot{r}'$ extends each $\dot{r}_i$. Then $(m',\dot{q}',\dot{r}')$ extends $(m_i,\dot{q}_i,\dot{r}_i)$ for each $i\in 2$.

We finish the proof by showing that (4) of Definition \ref{def:ce} holds. Thus let $A$ be a maximal antichain of $\bb{W}_\al$; note that we are not assuming that $A$ is a member of $M$. We show that $\iota[A]$ is still maximal in the  target poset (it is an antichain by (3)). Let $(m_0,\dot{q}_0,\dot{n}_0,\dot{r}_0)$ be an arbitrary condition in the target poset. Let $G^*=G\ast H_\al\ast G_\N$ be an $N$-generic filter over $\bb{M}_\mf{l}\ast\dot{\ps}_\al\ast\dot{\bb{N}}$ containing $(m_0,\dot{q}_0,\dot{n}_0)$. Define
$$
\bar{A}:=\lb \bar{r}\in\Prk(U):(\exists (m,\dot{\bar{q}})\in G\ast H_\al)\;(m,\dot{\bar{q}},\dot{\bar{r}})\in A\rb.
$$
Then $\bar{A}$ is a maximal antichain in $\Prk(U)$ which lives in $V[G\ast H_\al]$. Since $\Prk(U)$ is $\ka$-linked in $V[G]$, and hence in $V[G\ast H_\al]$, $\bar{A}$ has size $\leq\ka$. Furthermore, $\ps_\al$ is $\lam$-distributive over $V[G]$, so $\bar{A}\in V[G]$. Since $M$ is closed under $<\lam$-sequences from $V$ and $\bb{M}_\mf{l}$ is $\lam$-c.c., $M[G]$ is closed under $<\lam$-sequences from $V[G]$. Thus $\bar{A}$ is a member of $M[G]$. Since $\bar{A}$ has size $\leq\ka$ and $\operatorname{crit}(k)=\lam>\ka$, we have that $k(\bar{A})=k[\bar{A}]=\bar{A}$ is a maximal antichain in $k(\Prk(U))$. Thus fix conditions $\bar{r}\in\bar{A}$ and $r_1\in k(\Prk(U))$ so that $r_1\leq\bar{r},r_0$. By definition of $\bar{A}$, we may find $(m,\dot{\bar{q}})\in G\ast H_\al$ with $(m,\dot{\bar{q}},\dot{\bar{r}})\in A$. Now let $(m_1,\dot{q}_1,\dot{n}_1)$ be a condition in $G^*$ below $(m_0,\dot{q}_0,\dot{n}_0)$ which forces that $\dot{r}_1\leq\dot{\bar{r}},\dot{r}_0$ and which satisfies that $(m_1,\dot{q}_1)$ extends $(m,\dot{\bar{q}})$. We now see that $(m_1,\dot{q}_1,\dot{n}_1,\dot{r}_1)$ is a condition extending both the starting condition $(m_0,\dot{q}_0,\dot{n}_0,\dot{r}_0)$ as well as $\iota(m,\dot{\bar{q}},\dot{\bar{r}})\in \iota[A]$.
\end{proof}

\begin{lemma}\label{lemma:ce,de}
Suppose that $W_\al=G\ast H_\al\ast R$ is $V$-generic over $\bb{W}_\al$. In $V[W_\al]$, let $\bb{S}_\mf{l}$ be the poset which consists of conditions $(n,\dot{r})$ in $\N\ast k(\Prk(\dot{U}))$ such that for all $\bar{r}\in R$, $n$ does not force, over $N[W_\al]$, that $\dot{r}$ is incompatible with $\bar{r}$; the order is the same as $\N\ast k(\Prk(\dot{U}))$.

Then $\bb{S}_{\mf{l}}$ is forcing equivalent to the quotient of $(\bb{M}_\mf{l}\ast\dot{\ps}_\al\ast\dot{\bb{N}})\ast  k(\Prk(\dot{U}))$ by $W_\al$, where this quotient is computed from the map $\iota$ from Lemma \ref{lemma:ce,AP}.
\end{lemma}
\begin{proof}
We introduce the abbreviation $\Q$ for the quotient of  $(\bb{M}_\mf{l}\ast\dot{\ps}_\al\ast\dot{\bb{N}})\ast  k(\Prk(\dot{U}))$ by $W_\al$. Our goal is to define a dense embedding from $\Q$ to $\bb{S}_\mf{l}$.

Let $\vp$ be the map from $\Q$ to $\bb{N}\ast k(\Prk(\dot{U}))$
given (up to replacing the names $\dot{r}$ and $\dot{q}$ by equivalent ones named by $\N$ in $N[G\ast H_\al]$) by
$$
(m,\dot{q},\dot{n},\dot{r})\mapsto (\dot{n}[G\ast H_\al],\dot{r}).
$$
We show that $\vp$ is a dense embedding from $\Q$ to $\bb{S}_\mf{l}$ through a series of claims.

The first of these claims is that if $(m,\dot{q},\dot{n},\dot{r})$ is in $\Q$, then $(n,\dot{r})$ is in $\bb{S}_\mf{l}$. To see this, fix a condition $\bar{r}\in R$, and we will show that $n$ does not force over $N[G\ast H_\al\ast R]$ that $\dot{r}$ is incompatible with $\bar{r}$. Otherwise, there is a condition $\bar{r}_1\in R$ which forces (over $N[G\ast H_\al]$) that $n$ forces the incompatibility of $\dot{r}$ and $\bar{r}$. We may extend to assume that $\bar{r}_1\leq\bar{r}$. Let $(m_1,\dot{q}_1)\in G\ast H_\al$ be chosen so that $(m_1,\dot{q}_1,\dot{\bar{r}}_1)$ is in $G\ast H_\al\ast R$ and so that $(m_1,\dot{q}_1)$ forces that $\dot{\bar{r}}_1\leq\dot{\bar{r}}$. By Fact \ref{fact:quotientStandard}(2), we may find a condition $(m_2,\dot{q}_2,\dot{n}_2,r_2)$ in $\Q$ which extends the starting condition $(m,\dot{q},\dot{n},\dot{r})$ as well as $(m_1,\dot{q}_1,1_\N,\dot{\bar{r}}_1)=\iota(m_1,\dot{q}_1,\dot{\bar{r}}_1)$. Then $n_2\leq n_1$, and $n_2$ forces that $\dot{r}_2$ extends $\bar{r}_1$ and $\dot{r}$. Since $\bar{r}_1\leq\bar{r}$, $n_2$ also forces $\dot{r}_2\leq\bar{r}$. Now force in $N$ below $n_2$ to get a generic $G_\N$, and set $G^*:=G\ast H_\al\ast G_\N$. Additionally, let $R^*$ be $N[G^*]$-generic over $k(\Prk(U))$ containing $r_2$. By the chain condition of $\Prk(U)$, the closure of $M[G]$, the distributivity of $\ps_\al$, and the elementarity of $k:M[G]\to N[G^*]$, we see that $\bar{R}^*:=R^*\cap\Prk(U)$ is in fact $V[G\ast H_\al]$-generic over $\Prk(U)$. Moreover, it contains $\bar{r}_1$, since $r_2\leq\bar{r}_1$. By choice of $\bar{r}_1$, we conclude that $n$ forces over $N[G\ast H_\al\ast\bar{R}^*]$ that $\dot{r}$ is incompatible with $\bar{r}$. But in $N[G^*\ast R^*]$, $r$ and $\bar{r}$ are compatible; hence they are compatible in $N[G\ast H_\al\ast (G_\N\times\bar{R}^*)]$. Since this is a generic extension of $N[G\ast H_\al\ast\bar{R}^*]$ in which $\dot{r}$ is compatible with $\bar{r}$, $n$ does not in fact force their incompatibility.

The second claim is that conditions (1) and (2) of Definition \ref{def:ce} hold. (1) holds by definition, so we check (2), i.e., that $\vp$ is order preserving. Thus suppose that $(m_1,\dot{q}_1,\dot{n}_1,\dot{r}_1)\leq(m_0,\dot{q}_0,\dot{n}_0,\dot{r}_0)$, with both in the quotient. Then since $(m_1,\dot{q}_1)\in G\ast H_\al$, the definition of the order entails that $(n_1,\dot{r}_1)$ extends $(n_0,\dot{r}_0)$ in $\N\ast k(\Prk(\dot{U}))$, and hence in $\bb{S}_\mf{l}$.

Before showing that (3) and (5) of Definition \ref{def:ce} hold, we need a claim that will be useful for verifying (3) and (5). Thus our third claim is that if $(n,\dot{r})\in\bb{S}_\mf{l}$ and $(m,\dot{q},\dot{\bar{r}})\in W_\al$ forces that $(\dot{n},\dot{r})\in\dot{\bb{S}}_{\mf{l}}$, then $(m,\dot{q},\dot{n},\dot{r})$ is in $\Q$. To see this, fix a condition $(m_0,\dot{q}_0,\dot{\bar{r}}_0)\in W_\al$, and we will show that $(m,\dot{q},\dot{n},\dot{r})$ is compatible with $\iota(m_0,\dot{q}_0,\dot{\bar{r}}_0)$. Since $(n,\dot{r})$ is in $\bb{S}_{\mf{l}}$, $n$ does not force that $\dot{r}$ is incompatible with $\bar{r}_0$. Thus we can find an $N[W_\al]$-generic filter $G^*_\N$ containing $n$ so that in $N[W_\al][G^*_\N]=N[G\ast H_\al\ast (R\times G^*_\N)]$, $\bar{r}_0$ and $r$ are compatible. Let $r_1$ be below them both. Then $N[G\ast H_\al\ast G^*_\N]$ satisfies that $r_1$ extends $\bar{r}_0$ and $r$. We may therefore find $n_1\leq_\N n$ which forces that $\dot{r}_1$ extends $\dot{r}$ and $\bar{r}_0$. Fix a condition $(m_1,\dot{q}_1)\in G\ast H_\al$ extending $(m,\dot{q})$ and $(m_0,\dot{q}_0)$ which forces that $\dot{n}_1$ forces that $\dot{r}_1$ extends $\dot{r}$ and $\dot{\bar{r}}_0$. Then $(m_1,\dot{q}_1,\dot{n}_1,\dot{r}_1)$ extends both the starting condition $(m,\dot{q},\dot{n},\dot{r})$ and $\iota(m_0,\dot{q}_0,\dot{\bar{r}}_0)$.

Our next claim is that (3) of Definition \ref{def:ce} holds. We will show the contrapositive. So assume that $(m_i,\dot{q}_i,\dot{n}_i,\dot{r}_i)$, for $i\in 2$, are conditions in the quotient so that $(n_0,\dot{r}_0)$ and $(n_1,\dot{r}_1)$ are compatible in $\bb{S}_\mf{l}$. We show that the $(m_i,\dot{q}_i,\dot{n}_i,\dot{r}_i)$ are compatible in $\Q$. Since $(n_0,\dot{r}_0)$ and $(n_1,\dot{r}_1)$ are compatible in $\bb{S}_\mf{l}$, fix $(n_2,\dot{r}_2)\in\bb{S}_\mf{l}$ which extends both. Let $(m_2,\dot{q}_2,\dot{\bar{r}}_2)\in G\ast H_\al\ast W_\al$ be a condition which forces that $(\dot{n}_2,\dot{r}_2)\in\dot{\bb{S}}_\mf{l}$ and which forces that $(\dot{n}_2,\dot{r}_2)$ extends each $(\dot{n}_i,\dot{r}_i)$; by extending further if necessary, we may assume that $(m_2,\dot{q}_2)$ extends each $(m_i,\dot{q}_i)$. Then by our third claim, $(m_2,\dot{q}_2,\dot{n}_2,\dot{r}_2)$ is in the quotient $\Q$. Since it is below both $(m_i,\dot{q}_i,\dot{n}_i,\dot{r}_i)$, this completes the proof of our fourth claim.

Our fifth and final claim is that (5) of Definition \ref{def:ce} holds. Thus fix $(n,\dot{r})\in\bb{S}_\mf{l}$, and we will find a condition in $\Q$ which maps below it. In fact, we will show that $\vp$ is surjective. Since $(n,\dot{r})\in\bb{S}_\mf{l}$, we may find a condition $(m,\dot{q},\dot{\bar{r}})\in W_\al$ which forces $(\dot{n},\dot{r})\in\dot{\bb{S}}_\mf{l}$. By our third claim, $(m,\dot{q},\dot{n},\dot{r})$ is in the quotient. Since $\vp(m,\dot{q},\dot{n},\dot{r})=(n,\dot{r})$, we're done.
\end{proof}

\subsubsection{$\neg\AP(\lam)$ holds: setting up the argument.} For this case, we recycle notation and let $\dot{U}$ be an $\bb{M}^*_{\mf{l}}$-name for a normal measure on $\ka$. Define
$$
\bb{W}^*:=\bb{M}^*_{\mf{l}}\ast(\dot{\ps}_{\lam^+}\times\Prk(\dot{U})),
$$
where $\dot{\ps}_{\lam^*}$ is an $\bb{M}^*_{\mf{l}}$-name for a standard club adding iteration to achieve $\CSR(\lam)$ (see Fact \ref{fact:GLS}(1)). We see that $\bb{W}^*$ preserves all cardinals $\leq\ka^+$ and all cardinals $\geq\lam^+$, forces $\lam=\ka^{++}$, and forces $\cf(\ka)=\om\we\neg\AP(\lam)\we\CSR(\lam)$ (using Fact \ref{th:GK} to see that the Prikry forcing preserves the failure of $\AP(\lam)$). So we must check that $\bb{W}^*$ forces $\TP(\lam)$. 
\begin{remark}
It suffices to verify that for each $\al<\lam^+$, $\TP(\lam)$ holds after forcing with
$$
\bb{W}^*_\al:=\bb{M}^*_{\mf{l}}\ast(\dot{\ps}_\al\times\Prk(\dot{U})).
$$
\end{remark}

Using the definition of $\mf{l}$, we may create a weakly compact embedding $k:M\lra N$ which satisfies the assumptions of Lemma \ref{lemma:ce}. Let $\dot{q}^\al$ be a $k(\bb{M}^*_{\mf{l}})$-name which satisfies the conclusion of Lemma \ref{lemma:ce}. The next lemma is similar to Lemmas \ref{lemma:ce,AP} and \ref{lemma:ce,de} both in the statement and the proof. We recycle notation from the previous subsection and let $\dot{\N}$ also denote, in this case, the name so that there is a dense embedding from $k(\bb{M}^*_\mf{l})$ to $\bb{M}^*_\mf{l}\ast\dot{\ps}_\al\ast\dot{\N}$

\begin{lemma}\label{lemma:ce,nAP}
 The map $\iota^*$ given by
$$
(m,\dot{q},\dot{r})\mapsto (m,\dot{q},1_{\dot{\bb{N}}},\dot{r})
$$
is a complete embedding from $\bb{W}^*_\al$ to $(\bb{M}^*_\mf{l}\ast\dot{\ps}_\al\ast\dot{\bb{N}}^*_\mf{l})\ast k(\Prk(\dot{U}))$.

Moreover, given a $V$-generic $W^*_\al=G\ast H_\al\ast R$ over $\bb{W}^*_\al$, there is a dense embedding from the quotient of $(\bb{M}^*_\mf{l}\ast\dot{\ps}_\al\ast\dot{\bb{N}}^*_\mf{l})\ast k(\Prk(\dot{U}))$ by $\iota^*[W^*_\al]$ to the poset $\bb{S}^*_\mf{l}$, defined analogously to $\bb{S}_\mf{l}$ from Lemma \ref{lemma:ce,de}.
\end{lemma}

\subsubsection{Both cases (without collapsing): quotient analysis and branch lemmas.}

Looking at what is in common with both of the above cases, we have the following data:
\begin{enumerate}
    \item a weakly compact Laver diamond $\mf{l}$;
    \item  a version of Mitchell forcing built relative to $\mf{l}$, which we abbreviate as $\bb{M}$; 
    \item $\bb{M}$-names $\dot{U}$ and $\dot{\ps}_{\lam^+}$ for, respectively, a normal measure on $\ka$ and an iterated club adding to achieve $\CSR(\lam)$;
    \item an ordinal $\al<\lam^+$;
    \item  a name $\dot{T}$ in $\bb{M}\ast(\dot{\ps}_\al\times\Prk(\dot{U}))$ for a $\lam$-tree;
    \item a weakly compact embedding $k:M\lra N$ where
    \begin{enumerate}
        \item $\crit(k)=\lam$;
        \item $M$ contains  the parameters $\mf{l}$, $\bb{M}\ast(\dot{\ps}_\al\times\Prk(\dot{U}))$, $\dot{T}$, and the sequence $\vec{M}^\al$ from Lemma \ref{lemma:parameter};
        \item $k(\mf{l})(\lam)=\dot{\ps}_\al$;
        \item $N$ is closed under $<\lam$-sequences and (because $\vec{M}^\al\in M$) contains $k\res\al$ as an element.
    \end{enumerate}
\item A dense embedding from $k(\bb{M})$ to $\bb{M}\ast\dot{\ps}_\al\ast\dot{\N}$ for some $\dot{\N}$, and $\dot{\N}$ is forced to be the projection of $\Add(\ka,[\lam,k(\lam))$ and a $\ka^+$-closed term forcing $\dot{\bb{T}}$ (see Lemma \ref{lemma:tailprojections});
\item a $k(\bb{M})$-name $\dot{q}^\al$ in $N$ for a condition in $k(\dot{\ps}_\al)$ satisfying Lemma \ref{lemma:ce}.
\end{enumerate} 
To deal with both cases simultaneously, we introduce the abbreviation
$$
\tilde{\bb{W}}_\al:=\bb{M}\ast(\dot{\ps}_\al\times\Prk(\dot{U})).
$$
Using this notation, we additionally have
\begin{enumerate}
    \item[(9)] a complete embedding $\tilde{\iota}$ from $\tilde{\bb{W}}_\al$ to $\bb{M}\ast\dot{\ps}_\al\ast\dot{\N}\ast k(\Prk(\dot{U}))$;
    \item[(10)] the quotient of $\bb{M}\ast\dot{\ps}_\al\ast\dot{\N}\ast k(\Prk(\dot{U}))$ by the complete embedding $\tilde{\iota}$ is forcing equivalent to $\bb{S}$, where $\bb{S}$ stands for $\bb{S}_\mf{l}$ in the $\AP(\lam)$ case and $\bb{S}^*_\mf{l}$ in the $\neg\AP(\lam)$ case. 
\end{enumerate}

Our goal is to show that there is a cofinal branch through $T$ in the extension by $\bb{S}$, and that $\bb{S}$ cannot add a cofinal branch through $T$. If we can complete these two tasks, we will complete the argument in the non-collapsing case.

\begin{lemma}\label{lemma:Sbranch} Let $W_\al$ be $V$-generic over $\tilde{\bb{W}}_\al$. Then $\bb{S}$ forces over $N[W_\al]$ that there is a cofinal branch $\dot{\tau}$ through $T$.
\end{lemma}
\begin{proof}
Since $\bb{S}$ is forcing equivalent to the quotient of $\bb{M}\ast\dot{\ps}_\al\ast\dot{\N}\ast k(\Prk(\dot{U}))$ by $W_\al$, it suffices to show that this quotient forces that there is a cofinal branch through $T$. Let $\Q$ denote the quotient. Forcing over $N[W_\al]$ by $\Q$ extends $W_\al=G\ast H_\al\ast R$ to an $N$-generic filter $G\ast H_\al\ast G^*_\N\ast R^*$ which contains $\tilde{\iota}[W_\al]$. Let $H^*_\al$ be an $N[G\ast H_\al\ast G^*_\N\ast R^*]$-generic filter over $k(\ps_\al)$ containing $q^\al$. Since $k(\bb{M})$ is forcing equivalent to $\bb{M}\ast\dot{\ps}_\al\ast\dot{\N}$, let $G_{k(\bb{M})}$ be an $N$-generic filter over $k(\bb{M})$ so that $N[G_{k(\bb{M})}]=N[G\ast H_\al\ast G^*_\N]$ and so that $G\seq G_{k(\bb{M})}$. Then we may lift $k:M[G]\to N[G_{k(\bb{M})}]$. Since $H^*_\al$ contains $q^\al$ and $q^\al\leq k[H_\al]$, we may lift to $k:M[G\ast H_\al]\to N[G_{k(\bb{M})}\ast H^*_\al]$. Finally, since $R\seq R^*$ (because $G\ast H_\al\ast G^*_\N\ast R^*$ contains $\tilde{\iota}[W_\al]$) we may lift $k$ to $k:M[G\ast (H_\al\times R)]\to N[G_{k(\bb{M})}\ast (H^*_\al\times R^*)]$. Since the restriction of $k(T)$ (which has height $k(\lam)>\lam$) to height $\lam$ equals $T$, $T$ contains a cofinal branch $\tau$ in $N[G_{k(\bb{M})}\ast (H^*_\al\times R^*)]$. However, $k(\ps_\al)$ is $k(\lam)$-distributive over $N[G_{k(\bb{M})}]$, and by Lemma \ref{lemma:distributiveEaston}, $k(\ps_\al)$ is still $k(\lam)$-distributive in $N[G_{k(\bb{M})}\ast R^*]$. Thus the cofinal branch $\tau$ is a member of $N[G_{k(\bb{M})}\ast R^*]$. Since this was an arbitrary generic extension by $\Q$, this completes the proof.
\end{proof}

\begin{remark}\label{rmk:term} 
To complete the proofs of Theorem \ref{thm:WeaklyCompactCases} (1) and (2) in the non-collapsing case, we will show that $\bb{S}$ does not add the branch $\tau$.

Here we also observe that the term forcing $\bb{T}$ from (7) is in fact $\ka^+$-closed in $V[G\ast H_\al]$, since $N[G\ast H_\al]$ is closed under $<\lam$-sequences from $V[G\ast H_\al]$.
\end{remark}

We wish to show that $\bb{S}$ satisfies an appropriate branch lemma. As we will be repeating these arguments in the next subsection, but in a more difficult context, we will just state the relevant lemmas, deferring the more difficult analogues of their proofs until later.

We first have a ``quotient analysis", i.e., an analysis of when a condition in $\Prk(U)$ forces a condition in $\N\ast k(\Prk(\dot{U}))$ out of $\dot{\bb{S}}$. This will take place in the extension by $\bb{M}\ast\dot{\ps}_\al$, i.e., before forcing with $\Prk(U)$, since $\bb{T}$ is $\ka^+$-closed in $V[G\ast H_\al]$.

\begin{lemma} In $V[G\ast H_\al]$, let $\bar{r}\in\Prk(U)$, and let $\dot{r}$ be a name in $\N$ for a condition in $k(\Prk(\dot{U}))$. Let $n$ be a condition in $\N$ so that $n$ decides the stem of $\dot{r}$, say as $\operatorname{stem}(\dot{r})$. Then $\bar{r}$ forces that $(n,\dot{r})$ is not in $\dot{\bb{S}}$ iff one of the following is true:
\begin{enumerate}
    \item the stems of $\dot{r}$ and $\bar{r}$ are incompatible (i.e., do not have a common end-extension);
    \item $\ell(\dot{r})>\ell(\bar{r})$ and $\operatorname{stem}(\dot{r})\res[\ell(\bar{r}),\ell(\dot{r}))\not\subseteq A_{\bar{r}}$;
    \item $\ell(\bar{r})>\ell(\dot{r})$ and $n$ forces in $\N$ that $\operatorname{stem}(\bar{r})\res[\ell(\dot{r}),\ell(\bar{r}))\not\subseteq\dot{A}_{\dot{r}}$.
\end{enumerate}
\end{lemma}

We now suppose, for a contradiction, that for some $V[G\ast H_\al]$-generic $R$, $\dot{\tau}$ names a new branch through $T$. We suppose for simplicity that the empty condition of $\Prk(U)$ forces this about $\dot{\tau}$. The next item shows that, below a certain condition, we only need to extend the closed term part $\bb{T}$ of that condition in order to decide information about $\dot{\tau}$. We will write conditions in $\bb{N}$ as pairs $(p,f)$ in order to isolate extensions in the term ordering.

\begin{lemma}\label{lemma:extendclosed1}
In $V[W_\al]$ (i.e., after the forcing $\Prk(U)$), there is a condition $(p,f,\dot{r})$ in $\bb{S}$ so that for all $x\subset T$, all $\be<\lam$, and all $(p',f',\dot{r}')$ which extend $(p,f,\dot{r})$ in $\bb{S}$, the following implication holds: if $f'\leq_{\bb{T}}f$ and $(p',f',\dot{r}')$ forces in $\bb{S}$ that $\dot{\tau}\res\be=\check{x}$, then in fact $(p,f',\dot{r})$ forces $\dot{\tau}\res\be=\check{x}$.
\end{lemma}

Now, back in $V[G\ast H_\al]$, let $\bar{r}\in\Prk(U)$ be a condition which forces that $(p,f,\dot{r})$ satisfies the conclusion of Lemma \ref{lemma:extendclosed1}. We will construct a splitting tree of conditions. As we move up in the splitting tree, we have to take into account what the Prikry conditions force about the conditions in the splitting tree, and so we use the closure of $\bb{T}$ to run though enough possibilities (max antichains).\footnote{There is a subtlety that we will address in the collapsing case. Namely, that term extensions of conditions in the quotient remain in the quotient.}

\begin{lemma}
In $V[G\ast H_\al]$, there exist
\begin{enumerate}
    \item a sequence $\la f_s:s\in 2^{<\ka}\ra$ of conditions below $f$ in $\bb{T}$ so that $s\sqsubseteq t$ implies $f_t\leq_\bb{T}f_s$;
    \item a sequence $\la A_s:s\in 2^{<\ka}\ra$ of maximal antichains in $\Prk(U)/\bar{r}$;
    \item for all $s\in 2^{<\ka}$, a sequence $\la\al_{s,e}:e\in A_s\ra$ of ordinals and a sequence $\la\la\ga_{s,e,0},\ga_{s,e,1}\ra:e\in A_s\ra$ of pairs of \emph{distinct} ordinals below $\ka^+$
\end{enumerate}
satisfying that for each $s\in 2^{<\ka}$, each $e\in A_s$, and each $i\in 2$,
$$
e\Vdash_{\Prk(U)}\;(p,f_{s^\frown i},\dot{r})\Vdash_{\dot{\bb{S}}}\; \dot{\tau}(\check{\al}_{s,e})=\check{\ga}_{s,e,i}.
$$
\end{lemma}

To complete the proof, we work in $V[G\ast H_\al]$. For each $g\in 2^\ka$, noting that $2^\ka=\ka^{++}$, let $f_g$ be a lower bound in $\bb{T}$ of the sequence $\la f_{g\res\xi}:\xi<\ka\ra$. As we will show in Corollary \ref{cor:termExtend} when we analyze this situation  in greater detail and when the collapses are present, $\bar{r}$ still forces that $(p,f_g,\dot{r})$ is in $\bb{S}$. Additionally, let $\al^*<\lam$ be above all of the $\al_{s,e}$. For each $g\in 2^\ka$, we let $e_g\leq_{\Prk(U)}\bar{r}$ be a condition, $\dot{w}_g$ a name for a condition extending $(p,f_g,\dot{r})$ in $\dot{\bb{S}}$, and $\ga_g<\ka^+$ an ordinal so that $e_g$ forces that $\dot{w}_g$ forces that $\dot{\tau}(\check{\al}^*)=\check{\ga}_g$.

Let $X\seq 2^\ka$ be of size $\ka^{++}$ so that $g\mapsto\operatorname{stem}(r_g)$ is constant on $X$. Using the fact that each level of $T$ has size $\ka^+$, find $g_0\neq g_1$ in $X$ so that $\tilde{\ga}:=\ga_{g_0}=\ga_{g_1}$. Let $s\in 2^{<\ka}$ be the initial segment on which $g_0$ and $g_1$ agree, and suppose (by relabelling if necessary) that $g_i$ extends $s^\frown\la i\ra$ for $i\in 2$. Let $\hat{e}\in\Prk(U)$ extend $e_{g_0}$ and $e_{g_1}$, as well as some element $e$ of the max antichain $A_s$. Let $R$ be $V[G\ast H_\al]$-generic over $\Prk(U)$ containing $\hat{e}$. Then $w_{g_0}$ and $w_{g_1}$ both force in $\bb{S}$ that $\dot{\tau}(\al^*)=\ga^*$. However, $w_{g_i}\leq (p,f_{s^{\frown}\la i\ra},r)$, and since $\hat{e}$ extends $e$ and $e$ forces that the $(p,f_{s^{\frown}\la i\ra},r)$ disagree about $\dot{\tau}(\check{\al}_{s,e})$, we have a contradiction to the definition of a tree ordering.

\subsection{Down to $\aleph_{\omega+2}$}\label{ss:wcDown}

In this subsection, we show how to incorporate collapses into the arguments of the previous subsection, providing all of the relevant details for the branch lemmas. Initially, our argument will separately consider some of the set-up for the $\AP$ and $\neg\AP$ cases. However, our argument will then feature a single argument which works for either case.

\subsubsection{$\AP(\lam)$ case: partial reduction.} Let $\bb{M}_\mf{l}$ abbreviate $\bb{M}_\mf{l}(\ka,\lam)$. In this case, we force over the model $V[G_\ka][\bb{F}_{(\ka,\lam)}]$ from Lemma \ref{lemma:gg} with the following poset:
$$
\bb{J}:=(\bb{M}_{\mf{l}}\times\bb{A}(\ka))\ast (\dot{\ps}_{\lam^+}\times\Pc(\dot{U},\dot{G}^g)),
$$
where $\bb{M}_{\mf{l}}\ast\dot{\ps}_{\lam^+}$ is as in the non-collapse case, and where $\dot{U}$ and $\dot{G}^g$ are $\bb{M}_\mf{l}\times\bb{A}(\ka)$ names, respectively, for a normal measure on $\ka$ and a $(\dot{U},2)$ guiding generic.
Recall that $G^g$ remains a guiding generic (and $U$ a normal measure on $\ka$) after forcing with $\ps_{\lam^+}$, by Corollary \ref{Cor:GGPres} and the fact that $\bb{A}(\ka)$ preserves that $\ps_{\lam^+}$ is $\lam$-distributive (Lemma \ref{lemma:distributiveEaston}).

Forcing with $\bb{M}_{\mf{l}}\ast\dot{\ps}_{\lam^+}$ gives us a model with $\AP(\lam)\we\CSR(\lam)\we\ka^{++}=\lam$. Forcing with $\bb{A}(\ka)$ over this model preserves $\AP(\lam)$ (since it preserves $\lam$), and preserves $\CSR(\lam)$ by Fact \ref{th:AddCSR}. In the extension by $(\bb{M}_{\mf{l}}\times\bb{A}(\ka))\ast\dot{\ps}_{\lam^+}$, $G^g$ is still a guiding generic, and so $\Pc(U,G^g)$ forces that $\ka=\aleph_\om$, that $\lam=\aleph_{\om+2}$, and preserves all cardinals $\geq\lam$. In particular, it preserves $\AP(\lam)$. And finally, by Theorem \ref{th:CSRPreserve}, forcing with $\Pc(U,G^g)$ over the extension by $(\bb{M}_{\mf{l}}\times\bb{A}(\ka))\ast\dot{\ps}_{\lam^+}$ preserves $\CSR(\lam)$.

Thus it suffices to show that $\bb{J}$ forces the tree property at $\lam$. As before, we only need a proper initial segment of the club adding. To see this, observe that $\ps_{\lam^+}$ is $\lam^+$-c.c. after forcing with $\bb{M}_{\mf{l}}$. Next, the forcing $\bb{A}(\ka)\ast\Pc(\dot{U},\dot{G}^g)$ preserves the $\lam^+$-c.c. of $\ps_{\lam^+}$ since it is $\ka^+$-Knaster. Therefore, given any $\bb{J}$-name $\dot{T}$ for a $\lam$-tree, there is an $\al<\lam^+$ so that $\dot{T}$ is a name in
$$
\bb{J}_\al:=(\bb{M}_{\mf{l}}\times\bb{A}(\ka))\ast(\dot{\ps}_\al\times\Pc(\dot{U},\dot{G}^g)).
$$
This completes the first reduction in the $\AP(\lam)$ case.

\subsubsection{$\neg\AP(\lam)$ case: partial reduction.}

Let $\bb{M}^*_\mf{l}$ abbreviate $\bb{M}^*_\mf{l}(\ka,\lam)$. In this case, we will force with the poset
$$
\bb{J}^*:=(\bb{M}^*_{\mf{l}}\times\bb{A}(\ka))\ast(\dot{\ps}_{\lam^+}\times\Pc(\dot{U},\dot{G}^g))
$$
over the model $V[G^*_\ka][\bb{F}_{(\ka,\lam)}]$ from Lemma \ref{lemma:gg}. By similar preservation arguments as in the case of $\bb{J}$ (using Fact \ref{th:GK} to see that $\neg\AP(\lam)$ is preserved), we know that $\bb{J}^*$ forces that $\ka=\aleph_\om$, that $\lam=\aleph_{\om+2}$, and that $\CSR(\lam)\we\neg\AP(\lam)$ holds. Moreover, in order to see that $\bb{J}^*$ forces $\TP(\lam)$, it suffices to show that for all $\al<\lam^+$, the poset
$$
\bb{J}^*_\al:=(\bb{M}^*_{\mf{l}}\times\bb{A}(\ka))\ast(\dot{\ps}_\al\times\Pc(\dot{U},\dot{G}^g))
$$
forces $\TP(\lam)$.

\subsubsection{A further reduction for both cases:}

For the remainder of the paper, we let $\bb{M}$ denote either $\bb{M}_{\mf{l}}(\ka,\lam)$ or $\bb{M}^*_{\mf{l}}(\ka,\lam)$. We also let $\hat{\bb{J}}$ denote either $\bb{J}_\al$ or $\bb{J}^*_\al$. We will show that $\hat{\bb{J}}$ forces $\TP(\lam)$. We wish, therefore, to study the relevant quotients and prove a suitable branch lemma. However, we are confronted with the fact that $\hat{\bb{J}}$ has size $\lam^+$, on account of forcing with $\bb{A}(\ka)$. In order to provide the appropriate lifts of a weakly compact embedding with critical point $\lam$, we need a further reduction to a poset of size $\lam$. This is our next task.

Fix a $\hat{\bb{J}}$-name $\dot{T}$ for a $\lam$-tree. Let $M^*\prec H(\theta)$, with $\theta$ sufficiently large, so that $|M^*|=\lam$, $\,^{<\lam}M^*\seq M^*$, and so that $M^*\cap\lam^+\in\lam^+$. We also require that $M^*$ contains the parameters $\hat{\bb{J}}$, $\dot{T}$, and $\mf{l}$, as well as the sequence $\vec{M}^\al$ from Lemma \ref{lemma:parameter}. Let $\pi:M^*\lra M$ be the transitive collapse.

We next observe that $\pi^{-1}$ provides a complete embedding of $\pi(\hat{\bb{J}})$ into $\hat{\bb{J}}$. In short, $\pi^{-1}$ is the identity on $\bb{M}\ast\dot{\ps}_\al$. Moreover, $M[G\ast H_\al]$ is still closed under $<\lam$ sequences in $V[G\ast H_\al]$ since $\bb{M}$ is $\lam$-c.c. and since $\ps_\al$ is $\lam$-distributive. From this and the $\ka^+$-c.c. of $\bb{A}(\ka)\ast\Pc(\dot{U},\dot{G}^g)$, it follows that if $X\seq\pi(\hat{\bb{J}})$ is a maximal antichain in $V$, then $\pi^{-1}[X]$ is still maximal in $\hat{\bb{J}}$.

One consequence of the previous paragraph is that we may lift $\pi$ to an extension $\pi:M^*[\hat{G}]\lra M[\pi(\hat{G})]$, where $\hat{G}$ is $V$-generic over $\hat{\bb{J}}$. However, $\pi(T)=T$. Thus $T$ lives in the model $V[\pi(\hat{G})]$, where we note that this is indeed a $V$-generic extension, once more using the previous paragraph.

We will prove that $T$ has a cofinal branch in the model $V[\pi(\hat{G})]$. The heart of the proof is to show that if $k:M\lra N$ is a weakly compact embedding, then (roughly) below a master condition for the club adding, the quotient $k(\pi(\hat{\bb{J}}))/\pi(\hat{\bb{J}})$ exists and doesn't add a branch to $T$. However, we have one more preliminary matter to deal with before looking at this quotient.

\subsubsection{The Prikry lemma for $\pi(\Pc(\dot{U},\dot{G}^g))$.} It will be helpful later to know that $\dot{\R}:=\pi(\Pc(\dot{U},\dot{G}^g))$ satisfies the Prikry lemma not only in $V[G\times\bar{A}]$, but also in $V[G\ast H_\al\times\bar{A}]$, where $G\ast H_\al\times\bar{A}$ is $V$-generic over $(\bb{M}\ast\dot{\ps}_\al)\times\pi(\bb{A}(\ka))$. This does not seem to immediately follow from the fact that $\pi^{-1}(\dot{\R})$ (the un-collapsed Prikry forcing) is forced to satisfy the Prikry lemma. Our strategy will be to first show that $\dot{\R}$ is forced by $\bb{M}\times\pi(\bb{A}(\ka))$ to satisfy the strong Prikry lemma. We then show that this still holds after forcing with $\ps_\al$, and finally we will use the strong Prikry lemma for $\R$ to prove the Prikry lemma, which will be as usual for Prikry-type forcings, once we know that $\pi(\dot{U})$ names a normal measure on $\ka$.

\begin{lemma}
$\R$ satisfies the strong Prikry lemma in $V[G\times\bar{A}]$.
\end{lemma}
\begin{proof}
Recall that $\pi:M^*\lra M$ is the transitive collapse of $M^*$, and that $M^*\prec H(\theta)$. Extend $\bar{A}$ to a $V[G]$-generic filter $A$ over $\bb{A}(\ka)$. Then we can lift $\pi$ to $\pi:M^*[G\times A]\lra M[G\times\bar{A}]$. Since $M^*[G\times A]\prec H(\theta)[G\times A]$, we know that $M^*[G\times A]\models ``\pi^{-1}(\R)$ satisfies the strong Prikry lemma." Thus $M[G\times\bar{A}]\models``\R$ satisfies the strong Prikry lemma." This is, of course, insufficient as it stands, since we need the result in $V[G\times\bar{A}]$.

To see that this is the case, fix $D\seq\R$ which is dense, open with $D\in V[G\times\bar{A}]$, and fix a condition $r\in\R$. Let $X\seq D$ be a maximal antichain in $D$, and hence in all of $\R$ by the density of $D$. $\R$ is $\ka$-centered in $V[G\times\bar{A}]$, and hence $\ka^+$-c.c. Thus $|X|\leq\ka$.

Recalling that $^{<\lam}M\seq M$ and that $\bb{M}\times\pi(\bb{A}(\ka))$ is $\lam$-c.c., we have that $V[G\times\bar{A}]\models\,^{<\lam}M[G\times\bar{A}]\seq M[G\times\bar{A}]$. Since $X\seq\R\seq M[G\times\bar{A}]$ has size $\leq\ka$, we have that $X\in M[G\times\bar{A}]$. Hence
$$
\bar{D}:=\lb s\in\R:(\exists x\in X)\;[s\leq_{\R}x]\rb
$$
is dense, open in $\R$ and a member of $M[G\times\bar{A}]$. Additionally, $\bar{D}\seq D$ since $D$ is open and $X\seq D$. Now we apply the fact that $\R$ satisfies the strong Prikry lemma in $M[G\times\bar{A}]$ to find a direct extension, say $r^*$, of $r$ and an $n\in\om$ so that every $n$-step extension of $r^*$ is in $\bar{D}$. Since $\bar{D}\seq D$, this completes the proof.
\end{proof}

\begin{lemma}\label{lemma:PrikryInExtension}
$\R$ satisfies the strong Prikry lemma in $V[G\ast H_\al\times\bar{A}]$.
\end{lemma}
\begin{proof}
Let $D\seq\R$ be dense, open, with $D\in V[G\ast H_\al\times\bar{A}]$, and let $r\in\R$. As before, let $X\seq D$ be a maximal antichain in $D$, and hence in $\R$. $|X|=\ka$, since $\R$ is $\ka^+$-c.c. in $V[G\ast H_\al\times\bar{A}]$. $\ps_\al$ is $\lam$-distributive in $V[G]$, and hence is still $\lam$-distributive in $V[G\times\bar{A}]$ by Lemma \ref{lemma:distributiveEaston}. Thus $X\in V[G\times\bar{A}]$. Let $\bar{D}:=\lb s\in\R:(\exists x\in X)\;[s\leq_{\R}x]\rb$, so that $\bar{D}\seq\R$ is dense, open, and in $V[G\times\bar{A}]$. But $\bar{D}\seq D$, so an $r^*\leq^*r$ and an $n\in\om$ witnessing the strong Prikry lemma for $r$ and $\bar{D}$ in $V[G\times\bar{A}]$ also witness the strong Prikry lemma for $r$ and $D$ in $V[G\ast H_\al\times\bar{A}]$.
\end{proof}

The following lemma then gives us our desired result. We have not included the whole proof, since it is standard for Prikry-type forcings to conclude the Prikry lemma from the strong Prikry lemma.

\begin{lemma}\label{lemma:pl4r}
$\R$ satisfies the Prikry lemma in $V[G\ast H_\al\times\bar{A}]$.
\end{lemma}
\begin{proof}
Recall that $\dot{U}$ is an $(\bb{M}\times\bb{A}(\ka))$-name for a normal measure on $\ka$. Thus $M^*$ satisfies this statement, being elementary in $H(\theta)$. Therefore, as $\pi:M^*\lra M$ is the transitive collapse, $M[G\times\bar{A}]\models ``\pi(\dot{U})[G\times\bar{A}]$ is a normal measure on $\ka$." However, $M[G\times\bar{A}]$ is $<\lam$-closed in $V[G\times\bar{A}]$, and hence $\bar{U}:=\pi(\dot{U})[G\times\bar{A}]$ is a normal measure on $\ka$ in $V[G\times\bar{A}]$. $\ps_\al$ is $\lam$-distributive in $V[G\times\bar{A}]$, and so $\bar{U}$ is still a normal measure on $\ka$ in $V[G\ast H_\al\times\bar{A}]$.

We next let $\vp$ be a statement in the forcing language of $\R$, and let 
$$
D:=\lb r\in\R:r\text{ decides }\vp\rb.
$$
So $D$ is dense, open. Fix $r\in\R$, and let $n$ be the least element of $\om$ for which there is some $r^*\leq^* r$ so that every $n$-step extension of $r^*$ is in $D$. Let such an $r^*$ be fixed. To see that $n=0$, one proceeds by means of a typical uniformization argument.
\end{proof}

\subsubsection{Quotient analysis, branch lemmas, and the final argument:}

In this final part of the paper, we will analyze the appropriate quotient and show that it doesn't add a branch to $T$. Recall that $\pi(\hat{\bb{J}})$ contains the full forcing $\bb{M}\ast\dot{\ps}_\al$, but only truncated portions of the forcing to add the Cohen subsets of $\ka$ and the $\Pc$ poset. That is to say, $\pi(\hat{\bb{J}})$ is of the form $(\bb{M}\times\pi(\bb{A}(\ka)))\ast(\dot{\ps}_\al\times\pi(\Pc(\dot{U},\dot{G}^g)))$. Also, recall that there is a dense embedding from $k(\bb{M})$ to $\bb{M}\ast\dot{\ps}_\al\ast\dot{\N}$ (see Lemma \ref{lemma:tailprojections}), where $\N$ is a projection of a $\ka^+$-closed (in $V[G\ast H_\al]$) term forcing $\bb{T}$ and the Cohen forcing $\Add(\ka,[\lam,k(\lam)))$.

\begin{notation}To simplify notation, we let $\bar{\bb{J}}$ denote $\pi(\hat{\bb{J}})$, $\dot{\R}$ denote $\pi(\Pc(\dot{U},\dot{G}^g))$, and $\bb{A}$ denote $\pi(\bb{A}(\ka))$.
\end{notation}

The next lemma combines into one statement the analogues of Lemmas \ref{lemma:ce,AP} and \ref{lemma:ce,de} (or Lemma \ref{lemma:ce,nAP} in the $\neg\AP$ case) and Lemma \ref{lemma:Sbranch}. The proofs are routine modifications of the proofs of these lemmas.

\begin{lemma}\label{lemma:embedcollapsecase}\hfill
\begin{enumerate}
    \item There is a complete embedding $\hat{\iota}$ from $(\bb{M}\times\bb{A})\ast(\dot{\R}\times\dot{\ps}_\al)$ to $[(\bb{M}\ast\dot{\ps}_\al\ast\dot{\N})\times k(\bb{A})]\ast k(\dot{\R})$ given by
    $$
    (m,a,\dot{r},\dot{q})\mapsto ((m,\dot{q},1_\N),a,\dot{r}).
    $$
    \item Let $(G\times A)\ast (R\times H_\al)$ be generic for the left-hand poset, and let $\Q$ denote the quotient of the right-hand poset by this generic. Then there is a dense embedding from $\Q$ to $\bb{S}$, where $\bb{S}$ consists of all triples $(n,a,\dot{r})$ in $(\N\times k(\bb{A}))\ast k(\dot{\R})$ so that for all $(\bar{a},\dot{\bar{r}})\in A\ast R$, $n$ does not force over $N[(G\times A)\ast (R\times H_\al)]$ that $(a,\dot{r})$ is incompatible with $(k(\bar{a}),\bar{r})$.
    \item $\bb{S}$ forces over $N[(G\times A)\ast (R\times H_\al)]$ that there is a cofinal branch through $T$.
\end{enumerate}
\end{lemma}

Thus it suffices to show that $\bb{S}$ does not add a cofinal branch through $T$. 

\begin{notation}
Let $\dot{\tau}$ denote, for the remainder of the paper, an $\bb{S}$-name for a cofinal branch through $\dot{T}$ so that $\dot{\tau}\in N$.
\end{notation}

We proceed to the ``quotient analysis" (or more precisely, an analysis of $\bb{S}$), which we carry out in detail; recall the notation $\pstem(s)$ from Definition \ref{def:PrikryWithCollapses}.

\begin{lemma}\label{lemma:qa} \emph{(Quotient Analysis)} In $V[G\ast H_\al]$, let $(\bar{a},\dot{\bar{r}})\in\bb{A}\ast\dot{\R}$ be a condition so that $\bar{a}$ decides $\stem(\dot{\bar{r}})$. Let $(m,a,\dot{r})$ be a condition in $(\N\times  k(\bb{A}))\ast k(\dot{\R})$ so that $(m,a)$ decides the value of $\stem(\dot{r})$. Then $(\bar{a},\dot{\bar{r}})$ forces that $(m,a,\dot{r})$ is not in $\dot{\bb{S}}$ iff one of the following conditions is true:
\begin{enumerate}
\item $k(\bar{a})$ is incompatible with $a$;
\item either $\pstem(\dot{r})$ is incompatible with $\pstem(\dot{\bar{r}})$, or they are compatible, but one of the following is the case:
\begin{enumerate}
\item[(2a)] for some $i<\min(\ell(\dot{r}),\ell(\dot{\bar{r}}))$, $f^{\dot{\bar{r}}}_i$ is incompatible with $f^{\dot{r}}_i$, or
\item[(2b)] the topmost collapse of one of the stems is not below the next Prikry point of the other stem;
\end{enumerate}
\item $\ell(\dot{r})>\ell(\dot{\bar{r}})$, but for all $\bar{a}'\leq_{\bb{A}}\bar{a}$ so that $k(\bar{a}')$ is compatible with $a$ in $k(\bb{A})$, there is some $\bar{a}''\leq_{\bb{A}}\bar{a}'$ with $k(\bar{a}'')$ also compatible with $a$ which forces, in $\bb{A}$, the following over $V[G\ast H_\al]$: for some $i\in [\ell(\dot{\bar{r}}),\ell(\dot{r}))$ either $\al^{\dot{r}}_i\notin\dom(\dot{F}^{\dot{\bar{r}}})$, or $\al^{\dot{r}}_i\in\dom(\dot{F}^{\dot{\bar{r}}})$, but one of the following is the case:
\begin{enumerate}
\item[(3a)] $f^{\dot{r}}_i$ is incompatible with $\dot{F}^{\dot{\bar{r}}}(\al^{\dot{r}}_i)$ as functions, or
\item[(3b)] $\dot{F}^{\dot{\bar{r}}}(\al^{\dot{r}}_i)$ is not in the appropriate collapse poset;
\end{enumerate}
\item $\ell(\dot{\bar{r}})>\ell(\dot{r})$, and $a$ is compatible with $k(\bar{a})$, but the condition $(m,a\cup k(\bar{a}))$ forces the following over $V[G\ast H_\al]$ in $\N\times k(\bb{A})$: for some $i\in [\ell(\dot{r}),\ell(\dot{\bar{r}}))$, either $\al^{\dot{\bar{r}}}_i\notin \dom(\dot{F}^{\dot{r}})$, or $\al^{\dot{\bar{r}}}_i\in\dom(\dot{F}^{\dot{r}})$, but one of the following is the case:
\begin{enumerate}
\item[(4a)] $f^{\dot{\bar{r}}}_i$ is incompatible with $\dot{F}^{\dot{r}}(\al^{\dot{\bar{r}}}_i)$ as functions, or
\item[(4b)] $\dot{F}^{\dot{r}}(\al^{\dot{\bar{r}}}_i)$ is not in the appropriate collapse poset.
\end{enumerate}
\end{enumerate}
\end{lemma}
\begin{proof}
First we will show that if $(\bar{a},\dot{\bar{r}})$ does not force $(m,a,\dot{r})$ out of $\dot{\bb{S}}$, then all of (1) through (4) are false. Thus let $\bar{A}\ast\bar{R}$ be $V[G\ast H_\al]$-generic over $\bb{A}\ast\dot{\R}$ with $(\bar{a},\dot{\bar{r}})\in\bar{A}\ast\bar{R}$ so that $(m,a,\dot{r})\in\bb{S}$. Then, in particular, $k(\bar{a})$ is compatible with $a$, so that (1) is false.

We introduce the abbreviation $\bar{J}:=G\ast H_\al\ast\bar{A}\ast\bar{R}$. Let $(\bar{m},\dot{\bar{q}})$ be a condition in $G\ast H_\al$ so that $(\bar{m},\dot{\bar{q}},\bar{a},\dot{\bar{r}})$ is in $\bar{J}$. Since $(m,a,\dot{r})\in\bb{S}$, $m$ does not force in $\N$, over $N[\bar{J}]$, that $(a,\dot{r})$ is incompatible with $(k(\bar{a}),\bar{r})$. Thus we may find $m^*\leq_\N m$ and $(a^*,\dot{r}^*)$ so that $m^*$ forces over $N[\bar{J}]$ that $(a^*,\dot{r}^*)$ extends $(a,\dot{r})$ and $(k(\bar{a}),\bar{r})$.

This implies that $m^*$ forces over $N[G\ast H_\al\ast\bar{A}]$ that $(a^*,\dot{r}^*)$ extends $(a,\dot{r})$ and $(k(\bar{a}),\bar{r})$. Moreover, we may extend $(m^*,a^*)$ in $\N\times (k(\bb{A})/\bar{A})$ if necessary to assume that it decides the value of $\stem(\dot{r}^*)$.

Now we show that (2) is false. Recalling that the relevant stems have been decided, we see that $\pstem(\bar{r})$ is compatible with $\pstem(\dot{r})$ since both are initial segments of $\pstem(\dot{r}^*)$. Next, if $i<\min(\ell(\dot{r}),\ell(\bar{r}))$, then $f^{\dot{r}^*}_i$ extends both $f^{\bar{r}}_i$ and $f^{\dot{r}}_i$, so that they are compatible. And finally, if either of $\dot{r}$ or $\bar{r}$ is shorter than the other, then the topmost collapse of the shorter of the two is below the next Prikry point of the longer of the two, since this Prikry point is in common with $\dot{r}^*$.

Before showing that (3) and (4) are false, we force in $\N\times(k(\bb{A})/\bar{A})$ below $(m^*,a^*)$ to complete $G\ast H_\al$ to a $V$-generic filter $G^*\times A^*$ over $k(\bb{M}\times\bb{A})$ which satisfies  that $k[G\times\bar{A}]=G\times k[\bar{A}]\seq G^*\times A^*$.

Towards showing that (3) is false, we assume that $\ell(\dot{r})>\ell(\bar{r})$ (otherwise there is nothing to prove), and we will first make an argument in $V[G^*\times A^*]$. We need to find an extension $\bar{a}'$ of $\bar{a}$ so that $k(\bar{a}')$ is compatible with $a$ and so that $\bar{a}'$ forces the negation of the statement specified in (3). Note that we have a lift of $k$ to $k:M[G\times\bar{A}]\lra N[G^*\times A^*]$. By choice of $(m^*,a^*)$, we know that  $r^*\leq_{k(\R)}r,\bar{r}$. Fix some $i$ in $[\ell(\bar{r}),\ell(\dot{r}))$. Then $\al^r_i=\al^{r^*}_i$, and consequently $\al^r_i\in\dom(F^{\bar{r}})$. Next, $f^{r^*}_i$ extends $f^r_i$ as well as $F^{\bar{r}}(\al^r_i)$, so they are compatible as functions. And finally, $f^{r^*}_i$ is in the appropriate collapse poset, and hence $F^{\bar{r}}(\al^r_i)$ is too. The argument of this paragraph has so far taken place in $V[G^*\times A^*]$, but these statements are absolute between this model and $V[G\ast H_\al\times \bar{A}]$. Since $\bar{a}\in\bar{A}$, we may find some $\bar{a}'\leq_{\bb{A}}\bar{a}$ with $\bar{a}'\in\bar{A}$ so that $\bar{a}'$ forces the negation of the statement specified in (3). Lastly, $k[\bar{a}']\in k[\bar{A}]\seq A^*$ and $a\in A^*$ (because $a^*\leq a$ and $a^*\in A^*$), which  completes the proof that (3) is false.

We finish this direction of the lemma by showing that (4) is false. So assume that $a$ is compatible with $k(\bar{a})$. We need to find an extension of $(m,a\cup k(\bar{a}))$ which forces that the statement specified in (4) is false. This is very similar to the proof showing that (3) is false: in $V[G^*\times A^*]$, we have that $r^*\leq r,\bar{r}$. Interchanging the roles of $r$ and $\bar{r}$ from that of the previous paragraph, we see that, in $V[G^*\times A^*]$, the statement specified in (4) is false. Since $V[G^*\times A^*]$ is a generic extension of $V[G\ast H_\al]$ by forcing below $(m^*,a^*)$ in $\N\times k(\bb{A})$, and since $(m^*,a^*)\leq (m,a\cup k(\bar{a}))$, there must be some extension of $(m,a\cup k(\bar{a}))$ which forces that the statement specified in (4) is false.

We have therefore shown that if $(\bar{a},\dot{\bar{r}})$ does not force $(m,a,\dot{r})$ out of $\dot{\bb{S}}$, then the disjunction of (1)-(4) is false.\\

Now show the other direction: supposing that (1)-(4) are all false, we will find a $V[G\ast H_\al]$-generic filter $\bar{A}\ast\bar{R}$ containing $(\bar{a},\dot{\bar{r}})$ so that, in the extension, $(m,a,\dot{r})$ is in $\bb{S}$. We will have three cases based upon comparing the lengths of the stems of $\dot{r}$ and $\dot{\bar{r}}$. A quick note: in order to avoid even heavier notation, we will be reusing the symbols $\bar{A}$, $\bar{R}$, $G^*$, $H^*_\al$, etc. in proving this direction. However, they are not intended to denote the same objects as in the proof of the first direction.\\

\underline{Case 1:} $\ell(\dot{\bar{r}})=\ell(r)$. First we force below $m$ in $\N$ to extend $G\ast H_\al$ to a $V$-generic filter $G^*$ over $k(\bb{M})$. Let $G^*_{\text{tail}}$ denote the $V[G\ast H_\al]$-generic filter added by $\N$. By assumption, (1) is false, and therefore $k(\bar{a})$ is compatible with $a$. Let $a^*\in k(\bb{A})$ extend both, and force in $k(\bb{A})$ below $a^*$ to get a $V[G^*]$-generic filter $A^*$ over $k(\bb{A})$ containing $a^*$. Then $\bar{A}:=k^{-1}[A^*\cap\ran(k)]$ is a $V[G]$-generic filter over $\bb{A}$ containing $\bar{a}$; in fact, by the distributivity of $\ps_\al$, $\bar{A}$ is also $V[G\ast H_\al]$-generic. Note in particular that we have a lift $k:M[G\times\bar{A}]\lra N[G^*\times A^*]$. 

As (2) is false, $\pstem(r)$ is compatible with $\pstem(\bar{r})$, and as these stems have the same length, they are equal. What is more, for all $i<\ell:=\ell(\bar{r})=\ell(r)$, $f_i^r$ is compatible with $f^{\bar{r}}_i$. Let $f^*_i:=f_i^r\cup f_i^{\bar{r}}$, and let $u$ be the stem $u:=\la f^*_0,\al^r_1,f^*_1,\dots,\al^r_{\ell-1},f^*_{\ell-1}\ra$. Let $\de<\ka$ be large enough so that $f^{r^*}_{\ell-1}$ is below $\de$. Finally, $k(F^{\bar{r}})=F^{\bar{r}}$, and so we may find an upper part (with respect to the forcing $k(\R)$), say $F^*$, which extends $F^{\bar{r}}$ and $F^r$ so that $\dom(F^*)\seq[\de,\ka)$. Then $r^*:=(u,F^*)$ extends both $\bar{r}$ and $r$. Now we force over $V[G^*\times A^*]$ in $k(\R)$ below $r^*$ to get a generic filter $R^*$. Note that $\bar{R}:=R^*\cap\R$ is $V[G\times\bar{A}]$-generic over $\R$ and contains $\bar{r}$; in fact $\bar{R}$ is $V[G\ast H_\al\times \bar{A}]$-generic, using the distributivity of $\ps_\al$ in $V[G\times A]$ (Lemma \ref{lemma:distributiveEaston}). Let us abbreviate $(G^*_{\text{tail}}\times A^*)\ast R^*$ by $J^*$. We then have that $(m,a,\dot{r})\in J^*$. Since $k[\bar{A}\ast \bar{R}]\seq A^*\ast R^*$ and since $G^*_{\text{tail}}$ is generic over $N[G\ast H_\al\times\bar{A}]$, we see that $(m,a,\dot{r})$ is in $\bb{S}$. This completes the proof in this case.\\

\underline{Case 2:} $\ell(\dot{\bar{r}})<\ell(r)$. Let $G^*_{\text{tail}}$ be as in Case 1. By assumption, (3) is false, and so (i) there is some $\bar{a}'\leq_{\bb{A}}\bar{a}$ so that $k(\bar{a}')$ is compatible with $a$, and (ii) for any $\bar{a}''\leq_{\bb{A}}\bar{a}'$ with $k(\bar{a}'')$ compatible with $a$, $\bar{a}''$ fails to force the statement specified in (3). Let $\bar{a}'_1$ be the $\bb{A}$-condition $\bar{a}'\cup k^{-1}[a\cap\ran(k)]$. Since $k(\bar{a}'_1)$ is compatible with $a$, we may extend $\bar{a}'_1$ to a condition $\bar{a}''$ which forces that the statement in (3) is false; $k(\bar{a}'')$ is still compatible with $a$. More precisely, $\bar{a}''$ forces, over $V[G]$, the following: for all $i\in[\ell(\dot{\bar{r}}),\ell(\dot{r}))$, $\al^{\dot{r}}_i\in\dom(\dot{F}^{\dot{\bar{r}}})$, $f^{\dot{r}}_i$ is compatible with $\dot{F}^{\dot{\bar{r}}}(\al^{\dot{r}}_i)$ as functions, and $\dot{F}^{\dot{\bar{r}}}(\al^{\dot{r}}_i)$ is in the appropriate collapse poset.

Let $a^*\in k(\bb{A})$ extend $k(\bar{a}'')$ and $a$, and let $A^*$ be a $V[G^*]$-generic filter over $k(\bb{A})$ containing $a^*$. Then $\bar{A}:=k^{-1}[A^*\cap\ran(k)]$ is a $V[G\ast H_\al]$-generic filter over $\bb{A}$ which contains $\bar{a}''$. By our assumption on $\bar{a}''$, we know that for all $i\in [\ell(\bar{r}),\ell(r))$, $f^*_i:=f^r_i\cup F^{\bar{r}}(\al^r_i)$ is a condition in the appropriate collapse poset. Additionally, since (2) is false and $\ell(\bar{r})<\ell(r)$, $\pstem(\bar{r})$ is an initial segment of $\pstem(r)$. Furthermore, for all $i<\ell(\bar{r})$, we may let $f^*_i:=f^{\bar{r}}_i\cup f^r_i$, which is a condition. And finally, $\al^r_{\ell(\bar{r})}$ is above $f^{\bar{r}}_{\ell(\bar{r})-1}$. Let $u$ be the stem $\la f^*_0,\al^r_1,f^*_1,\dots,\al^r_{\ell(r)-1},f^*_{\ell(r)-1}\ra$. To find an appropriate upper part, we again use $k(F^{\bar{r}})=F^{\bar{r}}$ to get some $F^*$ (an upper part with respect to the poset $k(\R)$) extending $F^{\bar{r}}$ and $F^r$. We may also assume that $\min(\dom(F^*))$ is above $f^*_{\ell(r)-1}$, the top collapse of $u$. Then $r^*=(u,F^*)$ is in $k(\R)$ and extends $r$ as well as $\bar{r}$. The rest of the proof in Case 2 is similar to the proof of Case 1: fixing $R^*$ containing $r^*$, we have that $(m,a^*,\dot{r}^*)$, and hence $(m,a,\dot{r})$, is in the generic $(G^*_{\text{tail}}\times A^*)\ast R^*$. Then just as in Case 1, $(m,a,\dot{r})$ is in $\bb{S}$.\\

\underline{Case 3:} $\ell(\dot{\bar{r}})>\ell(r)$. Since (4) is false, we know that there is some $(m^*,a^*)$ which extends $(m,a\cup k(\bar{a}))$ in $\N\times k(\bb{A})$ and which forces the following: for all $i\in[\ell(\dot{r}),\ell(\dot{\bar{r}}))$, $\al^{\dot{\bar{r}}}_i\in\dom(\dot{F}^{\dot{r}})$, $f^{\dot{\bar{r}}}_i$ is compatible with $\dot{F}^{\dot{r}}(\al^{\dot{\bar{r}}}_i)$ as functions, and $\dot{F}^{\dot{r}}(\al^{\dot{\bar{r}}}_i)$ is in the appropriate collapse poset.

Forcing below $m^*$ in $\N$, we complete $G\ast H_\al$ to a $V$-generic filter (here we're recycling notation) $G^*$ over $k(\bb{M})$. Let $G^*_{\text{tail}}$ be the generic filter added over $V[G\ast H_\al]$ by $\N$.

We next force in $k(\bb{A})$ below $a^*$ to get a filter $A^*$ which is $V[G^*]$-generic containing $a^*$. Let $\bar{A}:=k^{-1}[A^*\cap\ran(k)]$.

In $V[G^*\ast A^*]$, using an argument entirely similar to that of Case 2 (by interchanging the roles of $\bar{r}$ and $r$), we see that $r$ and $\bar{r}$ are compatible in $k(\R)$. Let $r^*$ extend both, let $R^*$ be $V[G^*\times A^*]$-generic containing $r^*$, and let $\bar{R}=R^*\cap\R$. Then our starting condition $(m,a,\dot{r})$ is in $\bb{S}$, and $(\bar{a},\dot{r})$ is in $\bar{A}\ast\bar{R}$. This completes the proof of Case 3 and thereby the proof of the lemma.
\end{proof}

For the remainder of the paper, we write conditions in $\bar{\bb{J}}$ and $k(\bar{\bb{J}})$ as four-tuples, for instance, as $(p,f,a,\dot{r})$, where $(p,f)\in k(\bb{M})$. This will help facilitate extensions in the term ordering $\bb{T}$.

We recall that the term forcing $\bb{T}$ is $\ka^+$-closed in $V[G\ast H_\al]$. However, when working with the quotient and its forcing equivalent cousin $\bb{S}$, we need to maintain that the Mitchell part of the condition in $\bb{S}$ doesn't force incompatibility of the $k(\bb{A}\ast\dot{\bb{R}})$-part with some condition in $A\ast R$ (a $V[G\ast H_\al]$-generic over $\bb{A}\ast\dot{\R}$). Nonetheless the term ordering for conditions which do not force this incompatibility is highly closed. The next few items are dedicated to this. For the first of these, we recall that $\N$ projects to $\Add(\ka,[\lam,k(\lam)))$.

\begin{lemma}\label{lemma:densedecide}
In $V[G\ast H_\al]$, there is a dense set $D$ of conditions $(p,f,a,\dot{r})$ in $(\bb{N}\times k(\bb{A}))\ast k(\dot{\R})$ for which there exist nice $(\Add(\ka,[\lam,k(\lam)))\times k(\bb{A}))$-names $\dot{X}$ and $\dot{F}$ so that
$$
(p,f,a)\Vdash\dom(\dot{F}^{\dot{r}})=\dot{X}\we\dot{F}^{\dot{r}}=\dot{F}.
$$
Moreover, given a $V[G\ast H_\al]$-generic $A\ast R$ over $\bb{A}\ast\dot{\R}$, $D$ is still dense in $\bb{S}$.
\end{lemma}
\begin{proof}
The first statement in the lemma is a typical Easton's lemma argument. Recall that there is a projection from the three-fold product $(\Add(\ka,[\lam,k(\lam)))\times k(\bb{A}))\times\bb{T}$ onto $\N\times k(\bb{A})$. Moreover, $\bb{T}$ is $\ka^+$-closed in $V[G\ast H_\al]$, and so by Easton's lemma, $\bb{T}$ is $\ka^+$-distributive in the extension by the product of the two Cohen posets. With this in mind, fix a condition $(p,f,a,\dot{r})$ in $(\bb{N}\times k(\bb{A}))\ast k(\dot{\R})$, and let $G^{\text{tail}}\times\cal{A}$ be $V[G\ast H_\al]$-generic over $\N\times k(\bb{A})$ containing $(p,f,a)$. Let $\cal{A}'$ be the generic which $G^{\text{tail}}$ induces over $\Add(\ka,[\lam,k(\lam))$. Then $\dom(F^r)$ as well as $F^r$ are members of the extension by $\cal{A}'\times\cal{A}$. Hence we can find nice $(\Add(\ka,[\lam,k(\lam)))\times k(\bb{A}))$-names $\dot{X}$ and $\dot{F}$ in $V[G\ast H_\al]$ so that $\dot{X}[\cal{A}\times\cal{A}']=\dom(F^r)$ and $\dot{F}[\cal{A}\times\cal{A}']=F^r$. Let $(p',f',a')$ be an extension of $(p,f,a)$ in $G^{\text{tail}}\times\cal{A}$ which forces that these interpretations coincide. Then $(p',f',a',\dot{r})$ is an extension of $(p,f,a,\dot{r})$ in $D$.


The second statement in the lemma essentially follows from part (1) of this lemma, Fact \ref{fact:quotientStandard} and Lemma \ref{lemma:embedcollapsecase}. Before sketching the details, recall that in Lemma \ref{lemma:embedcollapsecase}, we are writing (names for) conditions in $\dot{\N}$ using single letters, but that as stated before the current lemma, we are now writing conditions in $\N$ as pairs $(p,f)$. With this in mind, recall the dense embedding from Lemma \ref{lemma:embedcollapsecase}. Furthermore, let $\dot{D}$ be an $\bb{M}\times\dot{\ps}_\al$-name for $D$. Then there is a dense set of conditions $(m,\dot{q},(\dot{p},\dot{f}),a,\dot{r})$ in $[(\bb{M}\ast\dot{\ps}_\al\ast\dot{\N})\times k(\bb{A})]\ast k(\dot{\R})$ so that, letting $\bar{a}\in\bb{A}$ be the condition $k^{-1}[a\cap\ran(k)]$, we have that $(m,\dot{q},\bar{a})$ forces that $(\dot{p},\dot{f},a,\dot{r})$ is in $\dot{D}$. By Fact \ref{fact:quotientStandard} and Lemma \ref{lemma:embedcollapsecase}, we get the desired conclusion.
\end{proof}

\begin{Cor}\label{cor:termExtend}
Suppose that $(p,f,a,\dot{r})$ is a condition in $D$ (see Lemma \ref{lemma:densedecide}) so that $(p,f,a)$ decides the value of $\stem(\dot{r})$. Next, suppose that $(\bar{a},\dot{\bar{r}})\in\bb{A}\ast\dot{\R}$, that $\bar{a}$ decides the value of $\stem(\dot{\bar{r}})$, and that $(\bar{a},\dot{\bar{r}})$ does not force that $(p,f,a,\dot{r})\notin\dot{\S}$. Finally, suppose that $f'\leq_{\bb{T}}f$. Then $(\bar{a},\dot{\bar{r}})$ does not force that $(p,f',a,\dot{r})\notin\dot{\S}$.

Consequently,
\begin{enumerate}
    \item if $(\bar{a},\dot{\bar{r}})$ forces that $(p,f,a,\dot{r})\in\dot{\bb{S}}$, then it also forces that $(p,f',a,\dot{r})\in\dot{\bb{S}}$;
    \item in $V[G\ast H_\al\ast A\ast R]$, where $A\ast R$ is generic for $\bb{A}\ast\dot{\R}$, if $(p,f,a,\dot{r})$ is in $\bb{S}$ and $f'\leq_{\bb{T}}f$, then $(p,f',a,\dot{r})$ is in $\bb{S}$ too.
\end{enumerate}
\end{Cor}
\begin{proof}
Suppose for a contradiction that $(\bar{a},\dot{\bar{r}})$ does force that $(p,f',a,\dot{r})\notin\dot{\S}$. Then one of (1)-(4) of Lemma \ref{lemma:qa} holds. Since $(\bar{a},\dot{\bar{r}})$ does not force $(p,f,a,\dot{r})$ out of $\dot{\bb{S}}$, and since the relevant stems have been decided, we know that (1)-(3) of Lemma \ref{lemma:qa} are false with respect to $(\bar{a},\dot{\bar{r}})$ and $(p,f,a,\dot{r})$. Thus (1)-(3) of Lemma \ref{lemma:qa} are false with respect to $(\bar{a},\dot{\bar{r}})$ and the further condition $(p,f',a,\dot{r})$.

Thus (4) of that lemma must hold with respect to $(\bar{a},\dot{\bar{r}})$ and $(p,f',a,\dot{r})$. Since $(p,f,a,\dot{r})$ is in $D$ (see Lemma \ref{lemma:densedecide}), let $\dot{X}$ and $\dot{F}$ be the $(\Add(\ka,[\lam,k(\lam)))\times k(\bb{A}))$-names which $(p,f,a)$ forces are equal to $\dom(\dot{F}^{\dot{r}})$ and $\dot{F}^{\dot{r}}$ respectively. By (4) of Lemma \ref{lemma:qa}, letting $a'$ abbreviate $a\cup k(\bar{a})$, the condition $(p,f',a')$ forces that the statement specified in (4) holds. It therefore forces that this statement holds with respect to the names $\dot{X}$ and $\dot{F}$. However, we then see that $(p,a')$ forces that the statement in (4) holds of $\dot{X}$ and $\dot{F}$, and thus $(p,f,a')$ must force this about $\dot{X}$ and $\dot{F}$. We therefore see that (4) of Lemma \ref{lemma:qa} holds with respect to $(p,f,a,\dot{r})$ and $(\bar{a},\dot{\bar{r}})$, which implies that $(\bar{a},\dot{\bar{r}})$ forces $(p,f,a,\dot{r})\notin\dot{\S}$, a contradiction.

For (1) of the ``consequently" part, suppose that some extension $(\bar{a}',\dot{\bar{r}}')$ of $(\bar{a},\dot{\bar{r}})$ forces that $(p,f',a,\dot{r})\notin\dot{\bb{S}}$; we may assume that $\bar{a}'$ decides the value of $\stem(\dot{\bar{r}}')$. Then, applying the first part of the corollary, we conclude that $(\bar{a}',\dot{\bar{r}}')$ forces that $(p,f,a,\dot{r})\notin\dot{\bb{S}}$, which contradicts the fact that $(\bar{a},\dot{\bar{r}})$ forces this condition into $\dot{\bb{S}}$.

Item (2) of the ``consequently" part of the corollary is now immediate.
\end{proof}

The next lemma shows that we can find a condition in $\S$ so that we only have to extend the term part of the condition to decide information about the branch $\dot{\tau}$; this will be useful later when we build a splitting tree of conditions since the term ordering is $\ka^+$-closed in $V[G\ast H_\al]$.

\begin{lemma}\label{lemma:extendTerm}
For any $V[G\ast H_\al]$-generic $\bar{A}\ast\bar{R}$ over $\bb{A}\ast\dot{\R}$, $V[G\ast H_\al\ast\bar{A}\ast\bar{R}]$ satisfies that there is a condition $(p,f,a,\dot{r})\in\S\cap D$  so that the following holds: for any $x\subset T$, any $\be<\lam$, and any $(p',f',a',\dot{r}')\leq_{\S}(p,f,a,\dot{r})$ with $(p',f',a',\dot{r}')$ also in $D$, if $f'\leq_{k(\bb{T})}f$ and if $(p',f',a',\dot{r}')\Vdash_{\bb{S}}\dot{\tau}\res\be=x$, then $(p,f',a,\dot{r})\Vdash_{\bb{S}}\dot{\tau}\res\be=x$.
\end{lemma}
\begin{proof}
Suppose otherwise, and let $\bar{A}\ast\bar{R}$ be a counterexample. Then for any condition $(p,f,a,\dot{r})\in\bb{S}\cap D$ there is some $x\subset T$, some $\be<\lam$, and some $(p',f',a',\dot{r}')\leq_{\bb{S}}(p,f,a,\dot{r})$ with $(p',f',a',\dot{r}')$ in $D$ which satisfy that $f'\leq_{\bb{T}}f$ and that $(p',f',a',\dot{r}')\Vdash_{\bb{S}}\dot{\tau}\res\be=x$, but where $(p,f',a,\dot{r})$ does not force in $\bb{S}$ that $\dot{\tau}\res\be=x$. Thus there exist a condition $(p^*,f^*,a^*,\dot{r}^*)$ extending $(p,f',a,\dot{r})$ in $\bb{S}$ with $f^*\leq_{\bb{T}}f'$ and some $x^*\subset T$ so that $x^*\neq x$ and $(p^*,f^*,a^*,\dot{r}^*)\Vdash_{\bb{S}}\dot{\tau}\res\be =x^*$. We may also extend, if necessary, to assume that $(p^*,f^*,a^*,\dot{r}^*)$ is in $D$. Lastly, note that $(p',f^*,a',\dot{r}')$ is still in $D$.

In summary, given any condition $(p,f,a,\dot{r})\in\bb{S}\cap D$, we may find $p_i$, $a_i$, $\dot{r}_i$, and $x_i$ for $i\in 2$, as well as $f^*$ and $\be$ so that
\begin{enumerate}
    \item[(a)] $f^*$ extends $f$ in the term ordering and $(p_i,f^*,a_i,\dot{r}_i)$ extends $(p,f,a,\dot{r})$;
    \item[(b)] $(p_i,f^*,a_i,\dot{r}_i)\in\bb{S}\cap D$;
    \item[(c)] $(p_i,f^*,a_i,\dot{r}_i)\Vdash_{\bb{S}}\dot{\tau}\res\be=x_i$;
    \item[(d)] $x_0\neq x_1$.
\end{enumerate}

Back in $V[G\ast H_\al]$, let $(\bar{a},\dot{\bar{r}})$ be a condition in $\bar{A}\ast\bar{R}$ which forces that the above holds. We will now construct, by recursion on $\be<\ka^+$ in $V[G\ast H_\al]$, various objects satisfying certain conditions. The objects are $p^i_\be$, $a^i_\be$, $\dot{r}^i_\be$, $\dot{x}^i_\be$, for $i\in 2$, as well as $\dot{x}_\be,\ga_\be,f_\be,\bar{a}_\be$, and $\dot{\bar{r}}_\be$. They will satisfy the following:
\begin{enumerate}
    \item $\la\ga_\be:\be<\ka^+\ra$ is strictly increasing, and if $\be$ is limit, then $\ga_\be>\sup_{\de<\be}\ga_\de$;
    \item $\la f_\be:\be<\ka^+\ra$ is decreasing in $\bb{T}$;
    \item $(\bar{a}_\be,\dot{\bar{r}}_\be)$ is below $(\bar{a},\dot{\bar{r}})$ in $\bb{A}\ast\dot{\R}$, and $\bar{a}_\be$ decides $\stem(\dot{\bar{r}}_\be)$;
    \item $(p^i_\be,f_\be,\dot{a}^i_\be)$ decides the value of $\stem(\dot{r}^i_\be)$, and $\stem(\dot{\bar{r}}_\be)\leq^*\stem(\dot{r}^i_\be)$ (that this is direct will be important later when we amalgamate conditions in $\bb{S}$);
\end{enumerate}
and $(\bar{a}_\be,\dot{\bar{r}}_\be)$ forces that
\begin{enumerate}
    \item[(5)] $(p^i_\be,f_\be,a^i_\be,\dot{r}^i_\be)\in\dot{\bb{S}}$;
    \item[(6)] $(p^i_\be,f_\be,a^i_\be,\dot{r}^i_\be)$ forces in $\dot{\bb{S}}$ that $\dot{\tau}\res\sup_{\de<\be}\ga_\de=\dot{x}_\be$ and that $\dot{\tau}\res\ga_\be=\dot{x}^i_\be$;
    \item[(7)] $\dot{x}^0_\be\neq \dot{x}^1_\be$.
\end{enumerate}
Suppose that $\be<\ka^+$ (either limit or successor) and that we have constructed these objects for all $\de<\be$. Since $\la f_\de:\de<\be\ra$ is decreasing in $\bb{T}$, we may find a lower bound $f'_\be$. Let $A\ast R$ be an arbitrary $V[G\ast H_\al]$-generic filter containing $(\bar{a},\dot{\bar{r}})$, and observe that $(1,f'_\be,1,1)$ is in $\bb{S}$, since, by Lemma \ref{lemma:qa}, no extension of $(\bar{a},\dot{\bar{r}})$ can force otherwise. Thus we may find an extension $(p^*_\be,f^*_\be,a^*_\be,\dot{r}^*_\be)$ of $(1,f'_\be,1,1)$ which is in $\bb{S}\cap D$ and which decides, in $\bb{S}$, the value of $\dot{\tau}\res\sup_{\de<\be}\ga_\de$, say, as $x_\be$. Applying the summary from the second paragraph of the proof, we may then find $p^i_\be$, $a^i_\be$, $\dot{r}^i_\be$, for $i\in 2$, as well as $x^0_\be\neq x^1_\be$,  $\ga_\be$, and $f_\be$ so that $f_\be\leq_{\bb{T}}f^*_\be$; so  that each $(p^i_\be,f_\be,a^i_\be,\dot{r}^i_\be)$ is in $\bb{S}\cap D$; so that $(p^i_\be,f_\be,a^i_\be,\dot{r}^i_\be)$ extends $(p^*_\be,f^*_\be,a^*_\be,\dot{r}^*_\be)$ and forces in $\bb{S}$ that $\dot{\tau}\res\ga_\be=x^i_\be$; and so that $(p^i_\be,f_\be,a^i_\be)$ decides the value of $\stem(\dot{r}^i_\be)$. Note that since $x^0_\be\neq x^1_\be$, we must have that $\ga_\be>\sup_{\de<\be}\ga_\de$. 

Let $\bar{r}_{\be,0}\in R$ be a condition which forces that $(p^*_\be,f^*_\be,a^*_\be,\dot{r}^*_\be)\Vdash_{\dot{\bb{S}}}\dot{\tau}\res\sup_{\de<\be}\ga_\de=\dot{x}_\be$; that $\dot{x}^0_\be\neq\dot{x}^1_\be$; and that the following $\R$-forcing-language sentence $\vp_\be$ holds (using canonical names for objects in $V[G\ast H_\al]$):
\begin{enumerate}
    \item[$(\vp_\be):$] for $i\in 2$, $(p^i_\be,f_\be,a^i_\be,\dot{r}^i_\be)$ is in $\dot{\bb{S}}$ and $(p^i_\be,f_\be,a^i_\be,\dot{r}^i_\be)\Vdash_{\dot{\bb{S}}}\dot{\tau}\res\ga_\be=\dot{x}^i_\be$.
\end{enumerate}
Next, we may extend (and relabel) $\bar{r}_{\be,0}$ if necessary to assume that $\stem(\bar{r}_{\be,0})\leq\stem(\dot{r}^i_\be)$ for $i\in 2$; in particular $\ell(\bar{r}_{\be,0})$ is at least as large as the other two lengths. We modify the above conditions, still working in $V[G\ast H_\al\ast A\ast R]$, in order to secure item (4) above.

We first extend $(p^0_\be,f_\be,a^0_\be,\dot{r}^0_\be)$ in $\bb{S}$ (if $\ell(\bar{r}_{\be,0})=\ell(\dot{r}^0_\be)$, then this extension is trivial). We introduce the abbreviation
$$
\bar{\al}_j:=\al^{\bar{r}_{\be,0}}_j,\;\text{ for all }j<\ell(\bar{r}_{\be,0}).
$$ 
Indeed, extend $(p^0_\be,f_\be,a^0_\be,\dot{r}^0_\be)$ to a condition $(p^{0,*}_\be,f^{0,*}_\be,a^{0,*}_\be,\dot{r}^{0,*}_\be)$ so that:
\begin{itemize}
    \item $(p^{0,*}_\be,f^{0,*}_\be,a^{0,*}_\be)$ decides the values of $\dot{F}^{\dot{r}^0_\be}(\bar{\al}_j)$ for all $j\in [\ell(\dot{r}^0_\be),\ell(\bar{r}_{\be,0}))$, say as $h_{\be,0,j}$;
    \item $f^{0,*}_\be$ extends $f_\be$ in $\bb{T}$;
    \item $\dot{r}^{0,*}_\be$ is a condition with stem equal to 
    $$
    \stem(\dot{r}^0_\be)\,^{\frown}\la \bar{\al}_{\ell(\dot{r}^0_\be)}, h_{\be,0,\ell(\dot{r}^0_\be)},\dots,\bar{\al}_{\ell(\bar{r}_{\be,0})-1}, h_{\be,0,\ell(\bar{r}_{\be,0})-1} \ra.
    $$
\end{itemize}
Next observe that $(p^1_\be,f^{0,*}_\be,a^1_\be,\dot{r}^1_\be)$ is an extension of $(p^1_\be,f_\be,a^1_\be,\dot{r}^1_\be)$ which is in $\bb{S}$, since the latter condition is in $\bb{S}\cap D$ and the former only extends on the $\bb{T}$-part (see Corollary \ref{cor:termExtend}). We may therefore extend it, still in $\bb{S}$, to a condition $(p^{1,*}_\be,f^{**}_\be,a^{1,*}_\be,\dot{r}^{1,*}_\be)$ so that:
\begin{itemize}
    \item $(p^{1,*}_\be,f^{**}_\be,a^{1,*}_\be)$ decides the values of $\dot{F}^{\dot{r}^1_\be}(\bar{\al}_j)$ for all $j\in [\ell(\dot{r}^1_\be),\ell(\bar{r}_{\be,0}))$, say as $h_{\be,1,j}$;
    \item $f^{**}_\be$ extends $f^{0,*}_\be$ (and hence $f_\be$) in $\bb{T}$;
    \item $\dot{r}^{1,*}_\be$ is a condition with stem equal to 
    $$
    \stem(\dot{r}^1_\be)\,^{\frown}\la \bar{\al}_{\ell(\dot{r}^1_\be)}, h_{\be,1,\ell(\dot{r}^1_\be)},\dots,\bar{\al}_{\ell(\bar{r}_{\be,0})-1}, h_{\be,1,\ell(\bar{r}_{\be,0})-1} \ra.
    $$
\end{itemize}
Note that $(p^{0,*}_\be,f^{**}_\be,a^{0,*}_\be,\dot{r}^{0,*}_\be)$ is an extension of $(p^{0,*}_\be,f^*_\be,a^{0,*}_\be,\dot{r}^{0,*}_\be)$ and is in $\bb{S}$.

Next, we amalgamate the relevant stems in $R$; that is to say, we find a condition $\bar{r}^*_{\be,1}\in R$ which is a direct extension of $\bar{r}_{\be,0}$ so that $\stem(\bar{r}^*_{\be,1})$ is a direct (stem) extension of $\stem(\dot{r}^{i,*}_\be)$ for $i\in 2$. The last main task to finish the construction of the objects at stage $\be$ is to find a direct extension of $\bar{r}^*_{\be,1}$ which forces over $V[G\ast H_\al\ast A]$ that all of the above facts hold. Here are the details.

In $V[G\ast H_\al\ast A]$, we apply the Prikry Lemma for $\R$ (see Lemma \ref{lemma:pl4r}) to find a direct extension $\bar{r}^*_\be\leq^*\bar{r}_{\be,1}$ so that $\bar{r}^*_\be$ decides the following $\R$-forcing-language sentence $\vp^*_\be$ (again using canonical names for objects in $V[G\ast H_\al]$):
\begin{enumerate}
    \item[$(\vp^*_\be)$:] for $i\in 2$, $(p^{i,*}_\be,f^{**}_\be,a^{i,*}_\be,\dot{r}^{i,*}_\be)$ is in $\dot{\bb{S}}$, and $(p^{i,*}_\be,f^{**}_\be,a^{i,*}_\be,\dot{r}^{i,*}_\be)\Vdash_{\dot{\bb{S}}}\dot{\tau}\res\ga_\be=\dot{x}^i_\be$.
\end{enumerate}
We claim that $\bar{r}^*_\be$ forces $\vp^*_\be$. Suppose otherwise, so that by assumption, $\bar{r}^*_\be$ forces $\neg\vp^*_\be$. Applying the Prikry lemma successively (four times), we may directly extend and relabel to assume that $\bar{r}^*_\be$ decides the truth value of each of the four conjuncts in $\vp^*_\be$; since $\bar{r}^*_\be$ forces that $\vp^*_\be$ is false, there must be some conjunct $\psi_\be$ of $\vp^*_\be$ which $\bar{r}^*_\be$ forces is false.

We claim that $\psi_\be$ is not the sentence ``$(p^{i,*}_\be,f^{**}_\be,a^{i,*}_\be,\dot{r}^{i,*}_\be)$ is in $\dot{\bb{S}}$" for some $i$. Suppose otherwise, and we will derive a contradiction. Since $V[G\ast H_\al\ast A\ast R]$ satisfies that $(p^{i,*}_\be,f^{**}_\be,a^{i,*}_\be,\dot{r}^{i,*}_\be)$ is in $\bb{S}$, $V[G\ast H_\al\ast A]$ satisfies that $(p^{i,*}_\be,f^{**}_\be,a^{i,*}_\be)$ is in $\N\times(k(\bb{A})/\bar{A})$. Thus we may force below $(p^{i,*}_\be,f^{**}_\be,a^{i,*}_\be)$ in that poset to obtain a $V[G\ast H_\al]$-generic filter $G^*_\N\times A^*$ containing $(p^{i,*}_\be,f^{**}_\be,a^{i,*}_\be)$, noting that $(G\ast H_\al\ast  G^*_\N)\times A^*$ contains the $k$-image of $G\times A$. Recalling that (the lift of) $k$ fixes conditions in $\R$, we claim that $\bar{r}^*_\be$ is compatible with $r^{i,*}_\be$. By construction, they have the same length. Next, $\stem(\bar{r}^*_\be)\leq^*\stem(\bar{r}^*_{\be,1})\leq^*\stem(\dot{r}^{i,*}_\be)$. Thus $\bar{r}^*_\be$ and $r^{i,*}_\be$ are compatible in $k(\R)$. Thus we may find a $V[G^*\times A^*]$-generic filter $R^*$ containing $\bar{r}^*_\be$ and $r^{i,*}_\be$. Then $(p^{i,*}_\be,f^{**}_\be,a^{i,*}_\be,\dot{r}^{i,*}_\be)$ is in the filter $(G^*_\N\times A^*)\ast R^*$. Thus in $V[G\ast H_\al\ast A\ast (R^*\cap\R)]$, $(p^{i,*}_\be,f^{**}_\be,a^{i,*}_\be,\dot{r}^{i,*}_\be)$ is in $\bb{S}$, and since $\bar{r}^*_\be\in (R^*\cap\R)$, this contradicts the fact that $\bar{r}^*_\be$ forces that $\psi_\be$ is false.

In light of the previous paragraph and the fact that $\bar{r}^*_\be$ decides the truth values of the four conjuncts in $\vp^*_\be$, we must then have that $\bar{r}^*_\be$ forces $(p^{i,*}_\be,f^{**}_\be,a^{i,*}_\be,\dot{r}^{i,*}_\be)\in\dot{\bb{S}}$ for $i\in 2$. Consequently, $\psi_\be$ (the conjunct which $\bar{r}^*_\be$ forces is false) must be the sentence ``$(p^{i,*}_\be,f^{**}_\be,a^{i,*}_\be,\dot{r}^{i,*}_\be)\Vdash_{\dot{\bb{S}}}\dot{\tau}\res\ga_\be=\dot{x}^i_\be$" for some $i\in 2$. However, $(p^{i,*}_\be,f^{**}_\be,a^{i,*}_\be,\dot{r}^{i,*}_\be)$ extends $(p^i_\be,f_\be,a^i_\be,\dot{r}^i_\be)$. And moreover, $\bar{r}^*_\be$ extends $\bar{r}_{\be,0}$ which forces that $(p^i_\be,f_\be,a^i_\be,\dot{r}^i_\be)\Vdash_{\dot{\bb{S}}}\dot{\tau}\res\ga_\be=\dot{x}^i_\be$. As a result, $(p^{i,*}_\be,f^{**}_\be,a^{i,*}_\be,\dot{r}^{i,*}_\be)$ extends $(p^i_\be,f_\be,a^i_\be,\dot{r}^i_\be)$ \emph{in the poset} $\dot{\bb{S}}$, and thus $\bar{r}^*_\be$ must force that $(p^{i,*}_\be,f^{**}_\be,a^{i,*}_\be,\dot{r}^{i,*}_\be)\Vdash_{\dot{\bb{S}}}\dot{\tau}\res\ga_\be=\dot{x}^i_\be$. This is a contradiction.

To summarize, we have a condition $\bar{r}^*_\be$ in $\R$ which forces $\vp^*_\be$. Since $(p^{i,*}_\be,f^{**}_\be,a^{i,*}_\be,\dot{r}^{i,*}_\be)$ extends  $(p^*_\be,f^*_\be,a^*_\be,\dot{r}_\be)$, since $\bar{r}_{\be,0}$ forces that $(p^*_\be,f^*_\be,a^*_\be,\dot{r}^*_\be)\Vdash_{\dot{\bb{S}}}\dot{\tau}\res\sup_{\de<\be}\ga_\de=\dot{x}_\be$, and since $\bar{r}^*_\be$ extends $\bar{r}_{\be,0}$, we have that $\bar{r}^*_\be$ forces that $(p^{i,*}_\be,f^{**}_\be,a^{i,*}_\be,\dot{r}^{i,*}_\be)\Vdash_{\dot{\bb{S}}}\dot{\tau}\res\sup_{\de<\be}\ga_\de=\dot{x}_\be$. Similarly, $\bar{r}^*_\be$ still forces $\dot{x}^0_\be\neq\dot{x}^1_\be$, since $\bar{r}_{\be,0}$ does.

This part of the argument has taken place in $V[G\ast H_\al\ast A]$. Thus we may find a condition $\bar{a}_\be$ in $A$ which extends $\bar{a}$ so that $(\bar{a}_\be,\dot{\bar{r}}^*_\be)$ forces the above and so that $\bar{a}_\be$ decides the value of $\stem(\dot{\bar{r}}^*_\be)$. This completes the $\ka^+$-length construction of the objects satisfying (1)-(7).\\

To complete the proof, we use a freezing out argument which allows us to amalgamate and eventually obtain problematic conditions in the quotient. First, let $X_0\in [\ka^+]^{\ka^+}$ be chosen so that for all $i\in 2$ and $\be<\de$ in $X_0$, $\stem(\dot{r}^i_\be)=\stem(\dot{r}^i_\de)$ and $\stem(\dot{\bar{r}}_\be)=\stem(\dot{\bar{r}}_\de)$. Since $a^i_\be\cup k(\bar{a}_\be)$ is a condition in $k(\bb{A})$ for each $i\in 2$ and each $\be\in X_0$, we may apply the fact that $k(\bb{A})$ is $\ka^+$-Knaster to find $X_1\in [X_0]^{\ka^+}$ so that for each $i\in 2$, the elements on the sequence $\la a^i_\be\cup k(\bar{a}_\be):\be\in X_1\ra$ are pairwise compatible. Finally, let $X_2\in [X_1]^{\ka^+}$ so that for each $i\in 2$, the (Cohen) conditions on the sequence $\la p^i_\be:\be\in X_2\ra$ are pairwise compatible.

For the remainder of the proof, fix $\be<\de$ in $X_2$. Set $p^i:=p^i_\be\cup p^i_\de$, set $a^i:=a^i_\be\cup k(\bar{a}_\be)\cup a^i_\de\cup k(\bar{a}_\de)$, and let $\dot{r}^i$ be a name forced to equal the weakest condition extending $\dot{r}^i_\be$ and $\dot{r}^i_\de$, if such exists, and forced to equal 1 otherwise. Since $a^i$ decides the value of both of the stems of $\dot{r}^i_\be$ and $\dot{r}^i_\de$ and forces that these stems are equal, $a^i$ forces that $\dot{r}^i$ has the same stem. Next, define $\bar{a}':=\bar{a}_\be\cup\bar{a}_\de$, and let $\dot{\bar{r}}'$ name the weakest extension of $\dot{\bar{r}}_\be,\dot{\bar{r}}_\de$, noting that $\bar{a}'$ forces that $\stem(\dot{\bar{r}}')=\stem(\dot{\bar{r}}_\be)=\stem(\dot{\bar{r}}_\de)$. In particular, $\bar{a}'$ decides the value of $\stem(\dot{\bar{r}}')$.

Our goal is to find a generic extension of $V[G\ast H_\al]$ in which for each $i\in 2$, $(p^i,f_\de,a^i,\dot{r}^i)$ is in the quotient.

To see this, recall (from condition (5) of the construction) that $(\bar{a}_\be,\dot{\bar{r}}_\be)$ forces that $(p^i_\be,f_\be,a^i_\be,\dot{r}^i_\be)\in\dot{\S}$. Thus, for all $\tilde{a}\leq_{\bb{A}}\bar{a}_\be$, $k(\tilde{a})$ is compatible with $a^i_\be$. Similarly for $\bar{a}_\de$. Thus for all $\tilde{a}\leq\bar{a}'$, $k(\tilde{a})$ is compatible with $a^i_\be\cup a^i_\de$. Now force in $\bb{A}$ below $\bar{a}'$ to get a generic $\bar{A}'$, noting that each $a^i$ is in $k(\bb{A})/\bar{A}'$. In $V[G\ast H_\al\ast\bar{A}']$, apply Lemma \ref{lemma:pl4r} to let $u$ be a direct extension of $\bar{r}'$ which decides whether or not $(p^0,f_\de,a^0,\dot{r}^0)$ is in $\dot{\bb{S}}$. We assert that $u$ forces that it is in $\dot{\bb{S}}$. In $V[G\ast H_\al\ast\bar{A}']$, we have that $(p^0,f_\de,a^0)$ is in the poset $\N\times (k(\bb{A})/\bar{A}')$. Thus, forcing below $(p^0,f_\de,a^0)$, we complete $G\ast H_\al\ast\bar{A}'$ to a generic $G^*\times A^*$ which contains $(p^0,f_\de,a^0)$. Then $u$ is compatible with $r^0$ since $\ell(u)=\ell(\bar{r}')=\ell(r^0)$, and since $\stem(u)\leq^*\stem(\bar{r}')\leq^*\stem(r^0)$. Now let $R^*$ be a generic for $k(\R)$ containing $r^0$ and $k(u)=u$, and we see that $(p^0,f_\de,a^0,\dot{r}^0)$ is in $(G^*\times A^*)\ast R^*$, and hence $V[G\ast H_\al\ast\bar{A}'\ast (R^*\cap\R)]$ satisfies that it is in $\bb{S}$. Since $u\in (R^*\cap\R)$ and since $R^*\cap\R$ is generic over $V[G\ast H_\al\ast \bar{A}']$, this shows that $u$ does not force $(p^0,f_\de,a^0,\dot{r}^0)$ out of $\bb{S}$, and by the choice of $u$, $u$ must therefore force $(p^0,f_\de,a^0,\dot{r}^0)$ into $\bb{S}$.

Still working in $V[G\ast H_\al\ast \bar{A}']$, we apply a similar argument with respect to the condition $(p^1,f_\de,a^1,\dot{r}^1)$ to find a direct extension $\bar{r}^*$ of $u$ which forces that $(p^1,f_\de,a^1,\dot{r}^1)$ is in $\dot{\bb{S}}$. Now force below $\bar{r}^*$ in $\R$ to get a generic $\bar{R}^*$ so that $(p^i,f_\de,a^i,\dot{r}^i)$ is in $\bb{S}$ for each $i\in 2$. In particular, $(p^i,f_\de,a^i,\dot{r}^i)$ forces that $\dot{\tau}\res\sup_{\xi<\de}\ga_\xi=x_\de$. However, the weaker condition $(p^i_\be,f_\be,a^i_\be,\dot{r}^i_\be)$ forces that $\dot{\tau}\res\ga_\be=x^i_\be$, and $x^0_\be\neq x^1_\be$. This contradicts the fact that $\dot{\tau}$ is forced to be a branch.
\end{proof}

We now follow the arguments in the literature for constructing a splitting tree of conditions. Let $(\bar{a},\dot{\bar{r}})$ be a condition in $\bb{A}\ast\dot{\R}$ and $(p,f,a,\dot{r})$ a condition in $D$ so that $(\bar{a},\dot{\bar{r}})$ forces, over $V[G\ast H_\al]$, that $(p,f,a,\dot{r})$ satisfies the conclusion of Lemma \ref{lemma:extendTerm}. We may construct the following objects:
\begin{itemize}
\item a sequence $\la f_s:s\in 2^{<\ka}\ra$ which is decreasing in $\bb{T}$ and below $f$;
\item for each $s\in 2^{<\ka}$, a maximal antichain $X_s$ in $(\bb{A}\ast\dot{\R})/(\bar{a},\dot{r})$, and a sequence $\la\al_{s,x}:x\in X_s\ra$ of ordinals;
\item for each $s\in 2^{<\ka}$, a sequence of $\la(\ga_{s,x,0},\ga_{s,x,1}):x\in X_s\ra$ with $\ga_{s,x,0}\neq\ga_{s,x,1}$ for all $x\in X_s$, and with $\ga_{s,x,0},\ga_{s,x,1}$ on the same level;
\end{itemize}
satisfying that for all $s\in 2^{<\ka}$ and all $x\in X_s$,
$$
x\Vdash_{\bb{A}\ast\dot{\R}}\;(p,f_{s^{\frown}\la i\ra},a,\dot{r})\Vdash_{\dot{\bb{S}}}\dot{\tau}(\al_{s,x})=\ga_{s,x,i}.
$$
In addition to Lemma \ref{lemma:extendTerm}, the construction uses the $\ka^+$-closure of $\bb{T}$ in $V[G\ast H_\al]$, the fact that $\bb{A}\ast\dot{\R}$ is $\ka^+$-c.c., and Corollary \ref{cor:termExtend}; the last of these is used to see that as  we extend in $\bb{T}$, the condition is still forced to be in the quotient.

Now we are almost done. Let $\al^*:=\sup\lb\al_{s,x}:s\in 2^{<\ka}\we x\in X_s\rb$, noting that $\al^*<\lam$. For each $g\in 2^\ka$ (and note, crucially, that $2^\ka=\ka^{++}=\lam$ in $V[G\ast H_\al]$), we let $f_g$ be a lower bound for $\la f_s:s\sqsubset g\ra$ in the term ordering so that for some $\ga_g$ and some $x_g\in\bb{A}\ast\dot{\R}$,
$$
x_g\Vdash_{\bb{A}\ast\dot{\R}}(p,f_g,a,\dot{r})\Vdash_{\dot{\bb{S}}}\dot{\tau}(\check{\al}^*)=\check{\ga}_g.
$$
Using the fact that $2^\ka=\ka^{++}$, we may find $g\neq h$ so that $\ga_g=\ga_h$ and so that $x_g$ is compatible with $x_h$. Let $s$ be the largest initial segment on which $g$ and $h$ agree, where we assume (by relabeling if necessary) that $s^\frown\la 0\ra\sqsubset g$ and $s^\frown\la 1\ra\sqsubset h$. Now fix $x\in\bb{A}\ast\dot{\R}$ extending $x_g$ and $x_h$, and let $\bar{A}\ast\bar{R}$ be generic containing $x$. We next observe that $(p,f_g,a,\dot{r})$ and $(p,f_h,a,\dot{r})$ are in $\bb{S}$. Indeed, $x$ extends $(\bar{a},\dot{\bar{r}})$ and $(\bar{a},\dot{\bar{r}})$ forces that $(p,f,a,\dot{r})\in\dot{\bb{S}}$. Thus, by Corollary \ref{cor:termExtend}(1) and the fact that $f_g$ and $ f_h$ extend $f$ in $\bb{T}$, $x$ also forces both $(p,f_g,a,\dot{r})$ and $(p,f_h,a,\dot{r})$ into $\dot{\bb{S}}$.

Since $\ga:=\ga_g=\ga_h$, we have that both of these conditions force, in $\bb{S}$, that $\lb\bar{\ga}\in T:\bar{\ga}\leq_T\ga\rb$ is an initial segment of $\dot{\tau}$. Recalling that $X_s$ is maximal in $(\bb{A}\ast\dot{\R})/(\bar{a},\dot{\bar{r}})$, let $z\in X_s\cap(\bar{A}\ast\bar{R})$. Then $z$ forces that $(p,f_{s^\frown\la i\ra},a,\dot{r})$ forces that $\dot{\tau}(\al_{s,z})=\ga_{s,z,i}$. As 
$$
(p,f_g,a,\dot{r})\leq (p,f_{s^\frown\la 0\ra},a,\dot{r})\;\text{ and }\; (p,f_h,a,\dot{r})\leq (p,f_{s^\frown\la 1\ra},a,\dot{r})
$$
and as $\al_{s,z}\leq\al^*
$
we must have $\ga_{s,z,i}\leq_T\ga$ for each $i\in 2$. This contradicts the fact that $\ga_{s,z,0}$ and $\ga_{s,z,1}$ are distinct nodes of $T$ on the same level.\\

We close this paper with two questions.

\begin{que}
Is $\CSR(\ka^{++})$ indestructible under all $\ka^+$-Knaster forcings, or even all $\ka^+$-c.c. forcings?
\end{que}

For the next item, the results from \cite{FHS2:large} and \cite{FHS1:large} will likely be relevant.

\begin{que}
How large can one make $2^{\aleph_\om}$ in a model of $\TP(\aleph_{\om+2})\we\CSR(\aleph_{\om+2})$?
\end{que}


\begin{thebibliography}{0}



\bibitem{ABR:tree}
Uri Abraham. Aronszajn trees on $\aleph_2$ and $\aleph_3$. {\it Annals of Pure and Applied Logic}. {\bf 24} (1983) no. 3, 213-230.

\bibitem{BNG}
Omer Ben-Neria and Thomas Gilton. Club stationary reflection and the special Aronszajn tree property. {\it Canadian Journal of Mathematics}. {\bf 75} (2023) no. 3, 854-911.

\bibitem{CummingsHandbook}
James Cummings. Iterated forcing and elementary embeddings. In Matthew Foreman and Akihiro Kanamori, editors, {\it Handbook of Set Theory}, volume 2, pages 885-1148. Springer, 2010.


\bibitem{CummingsForeman}
James Cummings and Matthew Foreman. The tree property. {\it Advances in Mathematics}. {\bf 133} (1998), 1-32.

\bibitem{ScalesSquares}
James Cummings, Matthew Foreman, and Menachem Magidor. Scales, squares and stationary reflection. {\it Journal of Mathematical Logic}. {\bf 1} (2001) no. 1, 35-98.


\bibitem{8fold}
James Cummings, Sy-David Friedman, Menachem Magidor, Assaf Rinot,  and Dima Sinapova. The eightfold way. {\it The Journal of Symbolic Logic}. {\bf 83} (2018) no. 1, 349-371.


\bibitem{FontHayut}
Laura Fontanella and Yair Hayut. Square and Delta Reflection. 
{\it Annals of Pure and Applied Logic}. {\bf 167} (2016) no. 8, pp. 663 - 683.



\bibitem{FHS2:large}
Sy-David Friedman, Radek Honzik, and {\v S}{\'a}rka Stejskalov{\'a}. The tree property at $\aleph_{\omega+2}$ with a finite gap. {\it Fundamenta Mathematicae}. {\bf 251} (2020) no. 3, 219-244.


\bibitem{FHS1:large}
Sy-David Friedman, Radek Honzik, and {\v S}{\'a}rka Stejskalov{\'a}. The tree property at the double sucessor of a singular cardinal with a larger gap. {\it Annals of Pure and Applied Logic}. {\bf 169} (2018) 548-564.


\bibitem{Gilton-thesis}
Thomas D. Gilton. {\it On the Infinitary Combinatorics of Small Cardinals and the Cardinality of the Continuum}. PhD thesis, University of California, Los Angeles, 2019.



\bibitem{GLS:8fold}
Thomas Gilton, Maxwell Levine, and \v{S}\'{a}rka Stejskalov\'{a}. Trees and stationary reflection at double successors of regular cardinals. {\it Journal of Symbolic Logic}.


\bibitem{GK:8fold}
Thomas Gilton and John Krueger. A note on the eightfold way. {\it Proceedings of the American Mathematical Society}. {\bf 148}. (2020) no. 3 1283-1293.


\bibitem{GITIKo2}
Moti Gitik. The negation of singular cardinal hypothesis from {$o(\kappa) = \kappa^{++}$}. {\it Annals of Pure and Applied Logic}. {\bf 43} (1989) 209-234.

\bibitem{GitikHandbook}
Moti Gitik. Prikry-Type Forcings. In Matthew Foreman and Akihiro Kanamori, editors, {\it Handbook of Set Theory}, volume 2, pages 1351-1447. Springer, 2010.


\bibitem{GK:a}
Moti Gitik and John Krueger. Approachability at the second successor of a singular cardinal. {\it The Journal of Symbolic Logic}. {\bf 74} (2009) no. 4 1211-1224.

\bibitem{GS:sch}
Moti Gitik and Assaf Sharon, On SCH and the approachability property. {\it Proceedings of the American Mathematical Society}. {\bf 136} (2008), no. 1 311-320.


\bibitem{Hamkins:LotteryPreparation}
Joel David Hamkins. The lottery preparation. {\it Annals of Pure and Applied Logic}. {\bf 101} (2000), no. 2-3, 103-146.


\bibitem{Hamkins:LaverDiamond}
Joel David Hamkins. A class of strong diamond principles. {\it ArXiv e-prints}, arXiv:math/0211419, 2002.



\bibitem{H:prikry}
Radek Honzik. Global singularization and the failure of {SCH}. {\it Annals of Pure and Applied Logic}. {\bf 161} (2010) no. 7 895-915.


\bibitem{HS:ind}
Radek Honzik, and {\v S}{\'a}rka Stejskalov{\'a}. Indestructibility of the tree property. {\it The Journal of Symbolic Logic}. {\bf 85} (2020) no. 1 467-485.

\bibitem{HS:u}
Radek Honzik and {\v S}{\'a}rka Stejskalov{\'a}. Small $\mathfrak{u}(\kappa)$ at singular $\kappa$ with compactness at $\kappa^{++}$. \emph{Archive for Mathematical Logic}. {\bf 61} (2022) 33-54.


\bibitem{JENfine}
R. Bj{\"{o}}rn Jensen. The Fine Structure of the Constructible Hierarchy. {\it Annals of Mathematical Logic}. {\bf 4} (1972) no. 3 229-308.


\bibitem{KunenBible}
Kenneth Kunen. Set Theory. Studies in Logic: Mathematical Logic and Foundations. Vol 34, College Publications, London, 2011, viii+401 pp. 

\bibitem{LambieHansonHayut}
Chris Lambie-Hanson and Yair Hayut. Simultaneous stationary reflection and square sequences. {\it Journal of Mathematical Logic}. {\bf 17} (2017) no. 2.

\bibitem{Laver:indestructible}
Richard Laver. Making the supercompactness of $\kappa$ indestructible under
$\kappa$-directed closed forcing. {\it Israel Journal of Mathematics}. {\bf 29} (1978) no. 4 385–388.

\bibitem{M:sr}
Menachem Magidor. Reflecting stationary sets. {\it The Journal of Symbolic Logic}. {\bf 47} (1982) no. 4 755-771.



\bibitem{M:tree}
William J. Mitchell. {A}ronszajn trees and the independence of the transfer property. {\it Annals of Mathematical Logic} {\bf 5} (1972) no. 1, 21-46.

\bibitem{SigmaPrikry1}
Alejandro Poveda, Assaf Rinot, and Dima Sinapova. Sigma Prikry Forcing I: The Axioms. {\it Canadian Journal of Mathematics} {\bf 73} (2021) no. 5 1205-1238.

\bibitem{SigmaPrikry2}
Alejandro Poveda, Assaf Rinot, and Dima Sinapova. Sigma Prikry Forcing II:  Iteration Scheme. {\it Journal of Mathematical Logic} {\bf 22} (2022) no. 3, article number 2150019.


\bibitem{ShelahApproachability}
Saharon Shelah. On successors of singular cardinals, pp. 357-80 in {\it Logic Colloquium '78 (Mons, 1978)}, vol. 97 of {\it Studies in the Logic and Foundations of Mathematics}, North Holland, Amsterdam, 1979.


\bibitem{SU:tree}
Dima Sinapova and Spencer Unger.  The tree property at $\aleph_{\omega^2+1}$ and $\aleph_{\omega^2+2}$. {\it Journal of Symbolic Logic}. {\bf 82} (2018) 669-682.


\end{thebibliography}
\end{document}